\documentclass[11pt,letterpaper]{amsart}
 \usepackage[]{amsmath, amsthm, amsfonts, amssymb, amscd,mathtools} 
\usepackage[mathscr]{eucal}
\usepackage{hyperref}
\usepackage{enumitem}
\usepackage{xypic}
\usepackage[usenames,dvipsnames]{color}
\usepackage{todonotes}
\usepackage{float}

\newcommand{\A}{\mathbb{A}}
\newcommand{\B}{\mathbb{B}}
\newcommand{\C}{\mathbb{C}}
\newcommand{\D}{\mathbb{D}}

\newcommand{\N}{\mathbb{N}}

\newcommand{\Q}{\mathbb{Q}}
\newcommand{\R}{\mathbb{R}}

\newcommand{\Z}{\mathbb{Z}}

\newcommand{\caA}{\mathcal{A}}
\newcommand{\caB}{\mathcal{B}}
\newcommand{\caC}{\mathcal{C}}

\newcommand{\caF}{\mathcal{F}}

\newcommand{\caH}{\mathcal{H}}
\newcommand{\caI}{\mathcal{I}}
\newcommand{\caJ}{\mathcal{J}}

\newcommand{\caL}{\mathcal{L}}

\newcommand{\caO}{\mathcal{O}}

\newcommand{\caR}{\mathcal{R}}

\newcommand{\caT}{\mathcal{T}}

   \def\WbDiv{\operatorname{W-b-Div}_{\R}}
\def\CbDiv{\operatorname{C-b-Div}_{\R}}

   \DeclareMathOperator {\rec} {rec}
   \DeclareMathOperator {\tr} {tr}
   
   \DeclareMathOperator {\ord} {ord}
 \DeclareMathOperator {\supp}{supp}
 
 \DeclareMathOperator {\Div}{Div_{\R}}
  \DeclareMathOperator {\dv}{div}
  
  \DeclareMathOperator {\Sp}{Sp}
  \DeclareMathOperator {\SL}{SL}
  \DeclareMathOperator {\GL}{GL}

 \DeclareMathOperator{\im}{Im}

\DeclareMathOperator{\Proj}{Proj}

\DeclareMathOperator{\Id}{Id}

\DeclareMathOperator{\vol}{Vol}
\DeclareMathOperator{\inter}{int}

\DeclareMathOperator{\bdiv}{b\text{-}div}

\newcommand{\cusp}{\text{\rm cusp}}
\newcommand{\tor}{\text{\rm tor}}

\newcommand{\can}{\text{\rm can}}

\newcommand{\inv}{\text{\rm inv}}

\numberwithin{equation}{section}

\theoremstyle{plain}
\newtheorem{prop}{Proposition}[section]

\newtheorem{cor}[prop]{Corollary}
\newtheorem{lem}[prop]{Lemma}
\newtheorem{thm}[prop]{Theorem}

\theoremstyle{definition}
\newtheorem{df}[prop]{Definition}

\newtheorem{notation}[prop]{Notation}

\newtheorem{introthm}{Theorem}
\newtheorem{introcor}[introthm]{Corollary}

\theoremstyle{remark}
\newtheorem{rmk}[prop]{Remark}
\newtheorem{ex}[prop]{Example}

 \thanks{A.\ Botero was supported by the collaborative research 
center SFB 1085 \emph{Higher Invariants - Interactions between
  Arithmetic Geometry and Global Analysis} funded by the Deutsche
Forschungsgemeinschaft. J.\ I.\ Burgos was partially supported by MINISTERIO
DE CIENCIA E INNOVACION research projects PID2019-108936GB-C21 and 
ICMAT Severo Ochoa project CEX2019-000904-S.
D.\ Holmes was supported by NWO grants VI.Vidi.193.006 and 613.009.103.
}
\begin{document}
\author[Botero]{Ana Mar\'ia Botero}
\address{University of Bielefeld,
Faculty of Mathematics,
Universitätsstraße 25, 33615 Bielefeld, Germany
}
\email{abotero@math.uni-bielefeld.de}

\author[Burgos Gil]{Jos\'e Ignacio Burgos Gil}
\address{Instituto de Ciencias Matem\'aticas (CSIC-UAM-UCM-UCM3),
  Calle Nicol\'as Ca\-bre\-ra~15, Campus UAM, Cantoblanco, 28049 Madrid,
  Spain} 
\email{burgos@icmat.es}

\author[Holmes]{ David Holmes}
\address{Mathematical Institute,  
Leiden University,
PO Box 9512, 
2300 RA Leiden, 
The Netherlands}
\email{holmesdst@math.leidenuniv.nl}

\author[de Jong]{Robin de Jong}
\address{Mathematical Institute,  
Leiden University,
PO Box 9512, 
2300 RA Leiden, 
The Netherlands}
\email{rdejong@math.leidenuniv.nl}

\title[Rings of Siegel--Jacobi forms are not finitely generated]{Rings of Siegel--Jacobi forms of bounded relative index are not finitely generated}

\begin{abstract}
We show that the ring of Siegel--Jacobi forms of fixed degree and of
fixed or bounded ratio between
weight and index is not finitely generated. Our main tool is the
theory of toroidal b-divisors and their relation to convex
geometry. As a byproduct of our methods, we prove a conjecture of
Kramer about the representation of all Siegel--Jacobi forms as
sections of certain line bundles and we  recover a
formula due to Tai for the asymptotic dimension 
of the space of Siegel--Jacobi forms of given ratio between weight and
index.  
\end{abstract}
\maketitle
\tableofcontents

\section{Introduction}

Siegel modular varieties arise as moduli spaces of abelian varieties of a fixed dimension~$g$ equipped with a principal polarization and level structure. They carry distinguished line bundles whose global sections are Siegel modular forms of degree~$g$. Siegel modular forms constitute an important class of automorphic forms and generalize the classical modular forms on finite index subgroups of $\SL_2(\Z)$ to higher dimensions. The study of Siegel modular forms is of fundamental importance in number theory and algebraic geometry.

Siegel modular varieties come with universal abelian varieties, and these similarly carry distinguished line bundles. Their global sections are known as \emph{Siegel--Jacobi forms}. 
Importantly, Siegel--Jacobi forms appear as Fourier coefficients of Siegel modular forms in higher degrees.

The systematic study of Siegel--Jacobi forms goes back to the
1980s. The case of degree one is dealt with in the book \cite{EZ} by
Eichler and Zagier and  in many papers from Zagier's school. For
higher degree, foundations were laid by 
Ziegler~\cite{Ziegler-J} and Dulinski~\cite{dulinski}. Later Kramer
developed the arithmetic theory 
of Siegel--Jacobi forms \cite{Kramer_Crelle}. An important aspect in
his work is the consideration of toroidal compactifications of
the universal abelian variety following the work of Mumford and its
collaborators \cite{toroidal}, \cite{AVRT:compact}, \cite{tai}, and
in the arithmetic setting by Faltings and
Chai~\cite{fc}. Around the same time Runge contributed by further
studying the geometric aspects of Siegel--Jacobi forms \cite{runge}.
Other work related to Siegel--Jacobi forms of higher degree has
appeared in \cite{yamazaki, yang1, yang2, yang4, yang3}.

\subsection{Statement of the main result}

Let $g$ be a positive integer.
Let  $\Gamma \subseteq \operatorname{Sp}(2g,\Z)$ be a finite index subgroup. Siegel modular forms with respect to $\Gamma$ are classified by weight, whereas Siegel--Jacobi forms with respect to $\Gamma$ are classified by two invariants, namely \emph{weight} and \emph{index}. We denote by $J_{k,m}(\Gamma)$ the space of Siegel--Jacobi forms of weight $k$ and index $m$ with respect to $\Gamma$. Then  $\bigoplus_{k,m}J_{k,m}(\Gamma)$ is naturally a bigraded $\C$-algebra which is known not to be finitely generated (see \cite{EZ} for the case of degree one, and Proposition \ref{prop:bi-gr-alg} below for the general case). 

From now on we fix a neat subgroup $\Gamma \subseteq \operatorname{Sp}(2g, \Z)$. Assume $g \ge 2$. Then Runge claims in \cite[Theorem 5.5]{runge} that, for suitable integers $r$, the rings $\bigoplus_{m\le 2rk}J_{k, m}(\Gamma)$ are finitely generated $\C$-algebras. 

Our main aim in this paper is to disprove Runge's claim. In fact we show that this algebra is \emph{never} finitely generated. In Remark \ref{rem:runge-mistake} we explain the oversight in Runge's proof. 

\begin{introthm}\label{th:intro-not-fin-gen}({Theorem~\ref{thm:2}} and Proposition~\ref{prop:disprove-runge})
Let $g \geq 1$. For each $k>0$ and $m>0$ the graded algebra $\bigoplus_{\ell}J_{\ell k, \ell m}(\Gamma)$ is not finitely generated. Therefore for any $r\in \mathbb Q_{>0}$ the graded algebra $\bigoplus_{m\le rk}J_{k, m}(\Gamma)$ is not finitely generated. 
\end{introthm}

Our main tool to prove Theorem \ref{th:intro-not-fin-gen} is the
theory of toroidal Weil and Cartier b-divisors (where b stands for birational). A Weil
b-divisor on a projective variety $X$ can be thought of as an infinite
tower of divisors living over all smooth birational models of $X$, and
compatible under pushforward (see Section \ref{sec:b-divisors} for a
precise definition and further details). A Weil b-divisor is said to be
\emph{Cartier} if it is determined on one such birational model, in
other words, if all the elements in the tower can be obtained by pullback and pushforward of a divisor on a single model. 

If $X$ carries a toroidal structure then the theory of Weil and Cartier b-divisors can be naturally simplified by considering only the toroidal blowups of $X$. The corresponding Weil toroidal b-divisors can then be seen as conical functions on a conical polyhedral complex associated to $X$. This enables the study of toroidal Weil and Cartier b-divisors by techniques from convex geometry \cite{BoteroBurgos}. For example, if a toroidal Weil b-divisor is Cartier then the corresponding conical function is piecewise linear.  In the following we consider $\R$-Weil and Cartier b-divisors. We omit the $\R$ in the notation in the Introduction for simplicity purposes.

With the neat subgroup $\Gamma$ one has associated a fibration of
principally polarized complex abelian varieties 
\[
\pi \colon \caB(\Gamma) \longrightarrow \caA(\Gamma) \, . 
\]
Here $\caA(\Gamma)$ is the Siegel modular variety associated to $\Gamma$.
For each pair of given integers $k, m$ the complex variety
$\caB(\Gamma)$ carries a line bundle of Siegel--Jacobi forms
$L_{k,m}$. It is endowed with a natural invariant metric, which we
denote by $h$ (see Section \ref{sec:geom-interpr-line}). As we will
see in Propositions \ref{prop:5} and \ref{prop:7} one can choose a
toroidal compactification $\overline{\caB}(\Gamma)$ of $\caB(\Gamma)$
in such a way that $L_{k,m}$ extends, non-uniquely, as an algebraic line bundle
$\overline{L}_{k,m}$ on $\overline{\caB}(\Gamma)$ and $h$ extends as a
\emph{toroidal psh} metric on  $\overline{L}_{k,m}$ (see Section
\ref{sec:toroidal-psh-metrics} for the definition of toroidal psh
metrics).   

As follows from the work done in our paper~\cite{BBHJ}, given a non-zero rational section $s$ of $\overline{L}_{k,m}$ we have an associated toroidal Weil b-divisor $\D(\overline{L}_{k,m}, s,h)$ on $\overline{\caB}(\Gamma)$. This b-divisor does not depend on the actual choice of toroidal psh extension $\overline{L}_{k,m}$. The Weil b-divisor $\D(\overline{L}_{k,m}, s,h)$ corresponds to a convex function on the conical polyhedral complex attached to $\overline{\caB}(\Gamma)$. 

The key point is that, via an explicit computation, we can show that
this conical function is \emph{not} piecewise linear. This  implies
that the toroidal Weil b-divisor $\D(\overline{L}_{k,m}, s,h)$ is not
Cartier (Corollary \ref{cor:1}). We then use a characterization of
Siegel--Jacobi cusp forms in terms of the invariant metric
(Proposition~\ref{prop:6}) to deduce that the algebra
$\bigoplus_{\ell}J_{\ell k, \ell m}(\Gamma)$ is not finitely
generated.

Even if the graded algebra
$\bigoplus_{\ell}J_{\ell k, \ell m}(\Gamma)$ is not finitely
generated, the situation is somewhat under control.

\begin{introthm}(Proposition~\ref{prop:10})
Let $g \geq 1$. For each $k>0$ and $m\ge 0$ the graded algebra
$\bigoplus_{\ell}J_{\ell k, \ell m}(\Gamma)$ is of almost integral type.
\end{introthm}

We note that this result is enough to prove \cite[Lemma
2.9]{BruinierWesterholt:modularityJacobi} that in the original version
used \cite[Theorem 5.5]{runge}.

\subsection{Kramer's conjecture}

Our second main result is a proof of the complex version of a
conjecture of Kramer \cite[Remark 2.19]{Kramer_Crelle}. To state the
conjecture, we fix $g \ge 2$, a weight $k$, an
index $m$, and a neat subgroup $\Gamma\subseteq
\operatorname{Sp}(2g,\Z)$. We define $\caT$ to be the set of triples
$(\Sigma, \Pi, \phi)$ where  \begin{enumerate} 
\item $\Sigma$ is an \emph{admissible cone decomposition} in the sense of Definition \ref{def:7}; this is the combinatorial data that determines a toroidal compactification $\overline{\caA}(\Gamma)_\Sigma$; 
\item
$\Pi$ is an admissible cone decomposition over $\Sigma$ in the sense of Definition \ref{def:8}; this is the combinatorial data that determines a toroidal compactification $\overline{\caB}(\Gamma)_\Pi$ with a map to $\overline{\caA}_\Sigma$; 
\item $\phi$ is an \emph{admissible polarization function} on $\Pi$ in the sense of Definition \ref{def:10}, such that $m\phi$ is integral in the sense of Proposition \ref{prop:4}; this determines an extension $\overline L_{k,m,\phi}$ of the line bundle $L_{k,m}$  of Siegel--Jacobi forms to $\overline{\caB}(\Gamma)_\Pi$. 
\end{enumerate}

Given a triple $(\Sigma, \Pi, \phi)$ in $ \caT$, restricting a section $f \in H^0\left(\overline{\caB}(\Gamma)_\Pi,
  \overline L_{k,m,\phi}\right)$ to the interior $\caB(\Gamma)$
yields, by the Koecher principle (here we use $g \ge 2$), a
Siegel--Jacobi form of weight $k$ and index $m$; in other words, there
is a natural injective map  
\begin{equation*}
H^0\left(\overline{\caB}(\Gamma)_\Pi, \overline L_{k,m,\phi}\right) \hookrightarrow J_{k,m}(\Gamma). 
\end{equation*}
The set $\caT$ is directed, by saying that $(\Sigma, \Pi, \phi) \le
(\Sigma', \Pi', \phi')$ if $\Sigma'$ refines $\Sigma$, $\Pi'$ refines
$\Pi$, and $\phi \ge \phi'$. If $(\Sigma, \Pi, \phi) \le (\Sigma',
\Pi', \phi')$ then there is a natural injection
\begin{equation}\label{eq:36}
H^0(\overline{\caB}(\Gamma)_{\Pi}, \overline L_{k,m,\phi}) \hookrightarrow H^0(\overline{\caB}(\Gamma)_{\Pi'}, \overline L_{k,m,\phi'}). 
\end{equation}
In \cite[Remark 2.19]{Kramer_Crelle}, Kramer defines
\[
J^{\can}_{k,m}(\Gamma) \subseteq J_{k,m}(\Gamma)
\]
to be the smallest sub-vector space containing all images of the
$H^0\left(\overline{\caB}(\Gamma)_{\Pi'}, \overline L_{k,m,\phi'}\right)$, and
conjectures  that  
\[
J^{\can}_{k,m}(\Gamma) = J_{k,m}(\Gamma). 
\]
Note that using the directed set structure defined above together with
\eqref{eq:36} one has
\begin{equation}\label{eq:colim_to_SJ}
J_{k,m}^{\can}(\Gamma)=
  \varinjlim_{(\Sigma, \Pi ,\phi) \in \caT } H^{0}\left(\overline
    \caB(\Gamma)_{\Pi},\overline  L_{k, m,\phi}\right). 
\end{equation}
In fact, we prove something stronger than Kramer's conjecture.  
\begin{introthm}[Corollary \ref{cor:6}] Let $g\ge 2$.
Suppose that $\phi$ is \emph{sufficiently negative} in the sense of Theorem \ref{thm:3}. Then the natural map 
\begin{equation*}
H^0(\overline{\caB}(\Gamma)_\Pi, \overline L_{k,m,\phi}) \hookrightarrow J_{k,m}(\Gamma)
\end{equation*}
is a bijection. Moreover, whenever $\Sigma$, $\Pi$ are chosen sufficiently fine, a sufficiently negative $\phi$ exists.  In particular, 
\[
J_{k,m}(\Gamma)^{\can} =  J_{k,m}(\Gamma). 
\]
\end{introthm}

\subsection{Asymptotic dimension formulae}

Tai's celebrated work \cite{tai} implies an asymptotic formula for the dimension of the space of Siegel--Jacobi forms of given ratio between weight and index. We recall that the main aim of \cite{tai} is to show that the moduli space $\caA_g$ of principally polarized abelian varieties of dimension $g$ is of general type for $g\geq 9$. Tai's proof proceeds via a study of the pushforward of the line bundle of Siegel--Jacobi forms to $\caA_g$, and an application of Mumford's version of the Hirzebruch proportionality principle \cite{hi} on the non-compact pure Shimura variety $\caA_g$. 

In the present work we arrive at a form of the proportionality principle on the universal abelian variety itself, via the machinery of b-divisors. As above we choose a
toroidal compactification $\overline{\caB}(\Gamma)$ of $\caB(\Gamma)$
in such a way that $L_{k,m}$ extends as an algebraic line bundle
$\overline{L}_{k,m}$ on $\overline{\caB}(\Gamma)$ and $h$ extends as a
toroidal psh metric on  $\overline{L}_{k,m}$. We will see that we can compute the asymptotic dimension of the space $J_{k,m}(\Gamma)$ of Siegel--Jacobi forms as  the degree of the toroidal Weil b-divisor $\D(\overline{L}_{k,m}, s,h)$ on $\overline{\caB}(\Gamma)$.

\newcommand{\udim}{d}

Let $g \in \Z_{ \geq 1}$ and set  $G = g(g+1)/2$ and $\udim = G + g$. The Weil b-divisor  $\D(\overline{L}_{k,m}, s,h)$ has a well-defined degree $\D(\overline{L}_{k,m}, s,h)^\udim$ in $\R_{\ge 0}$ (combine Remark~\ref{rk:degree} and Lemma~\ref{lem:nef}).
\begin{introthm}\label{th:intro-vol-jac}({Corollary~\ref{cor:asy}})
 Let $\D(\overline{L}_{k,m}, s,h)$ be the toroidal Weil b-divisor on the toroidal compactification $\overline{\caB}(\Gamma)$ associated to the line bundle $L_{k,m}$ of Siegel--Jacobi forms of weight~$k$ and index~$m$, the rational section $s$ and the invariant metric $h$. Then the Hilbert--Samuel type formula
\begin{equation} \label{eq:HilbSam}
\D(\overline{L}_{k,m}, s,h)^\udim = \limsup_{\ell \to \infty}\frac{\dim J_{\ell k, \ell m}(\Gamma)}{\ell^\udim/\udim!}
\end{equation}
holds, where $\D(\overline{L}_{k,m}, s,h)^\udim$ is the degree of the Weil b-divisor $\D(\overline{L}_{k,m}, s,h)$.
\end{introthm}
Using Chern--Weil theory for line bundles with (sufficiently nice) psh metrics as developed in \cite{BBHJ} we can compute the degree on the left hand side in \eqref{eq:HilbSam} explicitly  using integrals of smooth differential forms on the open universal abelian variety $\caB(\Gamma)$. Then by applying the equality in \eqref{eq:HilbSam} we obtain an explicit asymptotic dimension formula. The resulting formula is compatible with the one implicit in Tai's work \cite{tai}, see Remark~\ref{rem:tai-formula}.

\begin{introcor}\label{th:intro-asy-dim} ({Corollary \ref{th:asy-dim}})
 The asymptotic growth of the dimension of the space $J_{\ell k, \ell
    m}(\Gamma)$  is given by the following formulae: 
\begin{equation*}
\begin{split}
\limsup_{\ell \to \infty}\frac{\dim J_{\ell k, \ell m}(\Gamma)}{\ell^\udim/\udim!} 
& =  (-1)^{G}\udim! m^{g}k^{G}[\Gamma _{0}\colon  \Gamma ] \prod_{k=1}^{g}\frac{\zeta (1-2k)}{(2k-1)!!}\\
 & =  (-1)^{\udim}\udim!m^{g}k^{G}2^{G-g} 
  [\Gamma _{0}\colon \Gamma ]\prod_{k=1}^{g}\frac{(k-1)!B_{2k}}{(2k)!}\\
  & = V_g \cdot \udim!m^gk^G2^{-G-1}\pi^{-G}[\Gamma_0 \colon \Gamma],
\end{split}
\end{equation*}
where $\Gamma_0=\Sp(2g,\Z)$,
where $B_{2k}  = \frac{(-1)^{k+1} 2(2k)!}{(2\pi)^{2k}}\zeta(2k)$ are the Bernoulli numbers and
\[
V_g = (-1)^\udim 2^{g^2+1}\pi^G\prod_{k=1}^g\frac{(k-1)!B_{2k}}{(2k)!}
\]
is the symplectic volume computed by Siegel in \cite[Section VIII]{siegel}.
\end{introcor}

Note that we are assuming that $\Gamma $ is neat. In particular
$-\Id\not \in \Gamma $. In fact the above formulae are also true for
arbitrary $\Gamma \subset \Gamma _{0}$ of finite index not containing
$-\Id$. If $\Gamma $ contains $-\Id$ the right hand side has to be
multiplied by $2$.

We mention that
Theorem~\ref{th:intro-vol-jac} and Corollary~\ref{th:intro-asy-dim} generalize results proved in \cite{bkk} for the case $g=1$, $k=m=4$ and $\Gamma=\Gamma(N) \subset \SL_2(\Z)$ a principal congruence subgroup. As we shall see in Remark \ref{rmk:limsup2},  the $\limsup$ in Theorem~\ref{th:intro-vol-jac} and Corollary~\ref{th:intro-asy-dim} is actually a $\lim$ for sufficiently divisible $\ell$. 
 
We expect that our method to compute asymptotic dimensions can be generalized to other spaces of automorphic forms on mixed Shimura varieties. Indeed, the b-divisorial approach to Chern--Weil theory from \cite{BBHJ} continues to hold in these cases whenever the natural invariant metrics give rise to sufficiently nice psh line bundles.

\subsection{Outline of the paper}
In Section \ref{sec:basic-definitions} we set up the basic
definitions, including those of b-divisors, toroidal b-divisors, and
toroidal psh metrics (a special class of singular psh metrics). The
main new result in this section is a characterization of the Lelong
numbers of a toroidal psh metric in terms of the conical convex
function defining the toroidal psh metric, see Lemma \ref{lemm:16}. 

In Section \ref{sec:psh-metrics-b} we study properties of and relations between graded linear series, b-divisors and psh metrics. The relations are summarized in Diagram~\ref{diag:relations}. We then show that these relations become stronger in the toroidal case, see Theorem \ref{thm:1} and Corollary \ref{cor:2}. We describe an example to show that the toroidal assumption here is really necessary (Remark \ref{rem:exa-non-tor}). The compatibility relations in the toroidal case are key ingredients in the proofs of Theorems \ref{th:intro-not-fin-gen} and \ref{th:intro-vol-jac}. 

In Section \ref{sec:jacobi-forms} we recall the definitions of Siegel--Jacobi (cusp) forms. We discuss in detail the universal family $\pi \colon\caB(\Gamma) \to \caA(\Gamma)$ of principally polarized abelian varieties, and describe Siegel--Jacobi forms as global sections of a line bundle $L_{k,m}$ over $\caB(\Gamma)$. We also give an explicit description of the invariant metric $h$ on this line bundle. 

In Section~\ref{sec:toro-comp} we describe the theory of toroidal
compactifications of $\caA(\Gamma)$ and $\caB(\Gamma)$ and discuss
extensions of $L_{k,m}$ and $h$ over the
compactifications. Here we give a proof of Kramer's
  conjecture (Corollary \ref{cor:6}). Moreover, we prove that one can
choose the toroidal compactifications and the extensions in such a way
that the extended invariant metric is toroidal psh (Proposition
\ref{prop:7}). This leads to the observation that the b-divisor
associated to the line bundle of Siegel--Jacobi forms with its
invariant metric is toroidal (Corollary~\ref{cor:toroidal}).

In Section~\ref{sec:proofs_main} we give the proof of our main
result. We start by  giving an explicit formula for the relevant
conical function (Lemma \ref{lemm:13}). We then observe  that this
function is not piecewise linear. It follows that the associated
toroidal b-divisor  is not Cartier (see Corollary \ref{cor:1}), from
which we deduce Theorem~\ref{th:intro-not-fin-gen}.

In the final Section~\ref{sec:asymptotic} we define the volume of a
graded linear series and prove Theorem~\ref{th:intro-vol-jac}. Among
other things we use the Hilbert--Samuel formula for toroidal
b-divisors from \cite{BoteroBurgos}. Finally, we use Chern--Weil
theory for singular psh metrics as developed in \cite{BBHJ} to compute
the degree of our b-divisor in terms of the non-pluripolar volume of
the positive current associated to our psh metric, leading to
Corollary \ref{th:intro-asy-dim}. 

\subsection*{Acknowledgements}

We are very grateful to Gerard van der Geer, J\"urg Kramer and Don
Zagier  for helpful discussions. We are also very grateful to the
anonymous referees for their comments and suggestions. 

\section{Basic definitions}
\label{sec:basic-definitions}

In this section we recall the definitions of the basic tools we will
use in the paper. Throughout the paper $X$ will denote a connected complex projective manifold of dimension~$d$. 

\subsection{Graded linear series}
\label{sec:graded-linear-series}

Let $F=K(X)$ be the field of rational functions on $X$. The following
definition is taken from \cite{kk}.

\begin{df}\label{def:1} A \emph{graded linear series} (on $X$) is a graded
  subalgebra  $A\subset F[t]$. Let $L\subset F$ be a finite
  dimensional $\C$-vector space. The graded linear series $A_{L}$ is
  the graded subalgebra of $F[t]$ generated by $L\cdot t$ in degree 1 (note that $A_L$ is generated by finitely many elements of degree 1). 

  A graded linear series $A$ is called of \emph{integral type} if
  there is a finite dimensional linear subspace $L\subset F$ such that
  $A_L\subset A$ and $A$ is finite over $A_{L}$.

  A graded linear series $A$ is called of \emph{almost integral type}
  if there is a graded linear series $A'$ of integral type such that
  $A\subset A'$.
\end{df}

If $A$ is a graded linear series we write
$A=\bigoplus_{\ell \in \mathbb Z_{\ge 0}} A_{\ell}t^{\ell}$, where $A_{\ell}\subset F$.
\begin{rmk}
  If $X'\dashrightarrow X$ is a birational map then a graded linear series on $X$ is a graded linear series on
  $X'$. 
\end{rmk}

Let $D$ be a Cartier divisor on $X$, with integer, rational or real
coefficients. Then one writes 
\begin{displaymath}
  \caL(D)=\{0\not = f\in F\mid \dv(f)+D\ge 0\}\cup \{0\}
\end{displaymath}
and
\begin{displaymath}
  \caR(D)=\bigoplus_{\ell \in \mathbb Z_{\ge 0}}\caL(\ell D)t^{\ell}\subset F[t] \, . 
\end{displaymath}

The following result follows from \cite[Theorems 3.7 and 3.8]{kk} (using that $X$ is normal and projective).

\begin{prop}\label{prop:2}
  The graded linear series $\caR(D)$ is of almost integral type. If
  moreover $D$ is very ample, then 
  $\caR(D)$ is of integral type. 
\end{prop}

\subsection{Toroidal structures}
\label{sec:toroidal-structures}

\begin{df}\label{def:modification}
A morphism $\pi\colon X' \to X$ of connected projective complex manifolds is a
\emph{modification} if it is proper and birational. 
\end{df}

\begin{df}
  A \emph{(smooth) toroidal structure} on $X$ is the choice of a triple $(X_1,\pi_1,D)$ consisting of a smooth
  modification $\pi_{1} \colon X_{1 }\to X$ and a simple normal crossing
  divisor $D$ on $X_{1}$ that contains the exceptional locus of
  $\pi_{1} $.  If $(X_2,\pi _{2},D')$ is another toroidal structure on $X$, we say
  that $(X_2,\pi _{2},D')$ \emph{is above} $(X_1,\pi _{1},D)$ if $\pi_{2}$ factors as
  \begin{displaymath}
    \xymatrix{
      X_{2}\ar[rd]^{\pi_{2}}\ar[d]_{\pi _{21}}&\\
      X_{1 }\ar[r]^{\pi_{1} } & X
      }
    \end{displaymath}
    and $\pi _{21}^{-1}(D)\subset D'$.
\end{df}
If no confusion arises,  we will denote a toroidal structure $(X_1,\pi_1,D)$ just by the divisor $D$. 

The open subset $U=X_{1 }\setminus D$ can be identified with an open
subset of $X$ also denoted by $U$, and the inclusion $U\subset X_{1
}$ is an example of a toroidal embedding without self intersection. We
will use freely the theory
of toroidal embeddings from \cite{toroidal}. For a  concise description
of what we will need the reader is referred to \cite{BoteroBurgos}.

\begin{rmk}
One can envisage more general toroidal structures where
$U\subset X_{1}$ is an arbitrary toroidal embedding (e.g.~without assuming $X_1$ smooth), but we will not need 
them in this paper. 
\end{rmk}

Following \cite{toroidal}, to any toroidal structure $(X_1,\pi_1,D) $
we can associate a conical polyhedral complex $\Pi (X_{1},D)$. The
rays of $\Pi (X_{1},D)$ are in one-to-one correspondence with the
irreducible components of $D$. If $D_{i}$ is an irreducible component
of $D$, we denote by $\rho _{D_{i}}$ the corresponding ray. Given irreducible components
$D_{1},\dots ,D_{r}$  of $D$, then the set of cones of
$\Pi (X_{1},D)$ of dimension $r$ having $\rho _{D_{1}},\dots, \rho
_{D_{r}}$ as edges is in bijection with the set of irreducible
components of $D_{1}\cap\dots \cap D_{r}$. Each irreducible component
of an intersection of components of $D$ is called a \emph{stratum} of
the toroidal structure. Thus the set of strata is in bijection with
the set of cones. This bijection reverses dimensions and inclusions of
closures.  
The conical polyhedral
complex $\Pi (X_{1},D)$ has a canonical integral structure and we
denote by $v_{D_{i}}$ the primitive generator of $\rho _{D_{i}}$.

If $(X_{2},\pi _{2},D')$ is a toroidal structure above $(X_{1},\pi _{1},D)$, then
there is a continuous map
\begin{displaymath}
  r_{D,D'}\colon \Pi (X_{2},D')\to \Pi (X_{1},D), 
\end{displaymath}
that sends cones to cones, is linear on each cone and is compatible
with the integral structure on each cone.  
Moreover it is functorial in the sense that, if $(X_{3},\pi _{3},D'')$ is a toroidal structure above $(X_{2},\pi _{2},D')$, then
\begin{equation}\label{eq:17}
  r_{D,D'}\circ r_{D',D''}=r_{D,D''}. 
\end{equation}
The map $r_{D,D'}$ is constructed in greater generality in
\cite[Theorem 1.1]{ulirsch}, but the case we need admits a
particularly simple description as follows.

Let $\pi _{12}\colon
X_{2}\to X_{1}$ be the map of smooth modifications, $Y$ a stratum of
$(X_{2},D')$ and $Z$ the minimal stratum of $(X_{1},D)$
containing $\pi _{12}(Y)$. Let $D'_{1},\dots,D'_{r}$ be the set of
components of $D'$ such that $Y$ is an irreducible component of
$D'_{1}\cap\dots\cap D'_{r}$ and similarly let $D_{1},\dots, D_{s}$ be the set
of components of $D$ containing $Z$. Let $\sigma $ be the cone
corresponding to $Y$  and $\tau $ the cone corresponding to $Z$. Then
the restriction $r_{D,D'}|_{\sigma }\colon \sigma \to \tau$ is the
unique linear map satisfying
\begin{displaymath}
  r_{D,D'}(v_{D'_{i}})=\sum
  _{j=1}^{s}\ord_{D'_{i}}(\pi _{12}^{\ast}D_{j})v_{D_{j}}. 
\end{displaymath}
It is clear that these maps for the different cones of $\Pi (X_{2},D')$ 
glue together to give a map with the listed properties. 

To
prove the
functoriality \eqref{eq:17}, let $\pi _{23}\colon X_{3}\to X_{2}$ and
$\pi _{12}\colon X_{2}\to X_{1}$ be maps of modifications. Let $P$ be a
component of $D''$, let $D'_{1},\dots,D'_{r}$ be  the set of components of $D'$
containing the image of $P$ in $X_{2}$ and let $D_{1},\dots,D_{s}$ be the
set of components
of $D$ containing the image of $P$ in $X_{1}$. For each $j$,
$\pi _{12}^{-1}(D_{j})\subset D'$, and we have
\begin{displaymath}
  \ord_{P}((\pi _{12}\circ\pi _{23})^{\ast}D_{j})=
  \sum_{i=1}^{r}\ord_{P}((\pi _{23})^{\ast}D'_{i})
  \ord_{D'_{i}}(\pi _{12}^{\ast}D_{j}).
\end{displaymath}
Therefore
\begin{align*}
  r_{D,D''}(v_{P})
  &=
    \sum_{j=1}^{s}\ord_{P}((\pi _{12}\circ \pi _{23})^{\ast}D_{j})v_{D_{j}}\\
  &=
    \sum_{j=1}^{s}\sum _{i=1}^{r}\ord_{P}(\pi_{23}^{\ast}D'_{i})
    \ord_{D'_{i}}(\pi_{12}^{\ast}D_{j})v_{D_{j}}\\
  &=
    \sum_{i=1}^{r}\ord_{P}(\pi_{23}^{\ast}D'_{i})
    r_{D,D'}(v_{D'_{i}})\\
  &=
    r_{D,D'}\left(
    \sum_{i=1}^{r}\ord_{P}(\pi_{23}^{\ast}D'_{i})v_{D'_{i}}
    \right)\\
  &=
    r_{D,D'}\circ r_{D',D''}(P).
\end{align*}

Let $(X_1,\pi_1,D)$ be a toroidal structure on $X$.
Among the modifications of $X$ there is a distinguished class:
the \emph{allowable modifications} of the toroidal embedding $X_1 \setminus D \hookrightarrow
X_1$. See for instance \cite[Ch.~II, \S~2,
Definition~3]{toroidal} for the precise definition. 

The allowable
modifications are in one-to-one correspondence with the rational
subdivisions of the conical
polyhedral complex $\Pi
$ associated to $D$. Any such modification will be called toroidal (with respect to the
toroidal structure $D$). Note that, if $(X_{2},\pi _{2},D')$ is toroidal
with respect to $(X_{1},\pi _{1},D)$, then $D'$ is above
$D$ and the map $r_{D,D'}$ defined above induces a
homeomorphism between 
the underlying topological spaces $|\Pi (X_{2},D')|\to |\Pi (X_{1},D)|$,
compatible with the fact that $\Pi (X_{2},D')$ is a subdivision of $\Pi
(X_{1},D)$. 

We end this section with two definitions. 

\begin{df}
  Let $(X_1,\pi_1,D)$ be a toroidal structure on $X$. A \emph{toroidal
  exceptional prime divisor} of $(X_1,\pi_1,D)$ is an irreducible component of any
exceptional divisor of a toroidal modification of $(X_1,\pi_1,D)$. 
\end{df}

\begin{df}
  The \emph{set of rational points} $\Pi (X_{1},D)(\Q)$ is the set of
  points of $ \Pi (X_{1},D)$ that have rational coordinates when expressed in terms of the
  primitive vectors $v_{D_{j}}$. 
\end{df}

\subsection{b-divisors}
\label{sec:b-divisors}

In this section we discuss (Cartier and Weil) $\R$-b-divisors on $X$ as well as their
toroidal counterparts. This is essentially Shokurov's notion of
birational divisors (or b-divisors)  \cite{sh-prel}. This section is purely algebraic, so, if preferred, the reader can work with finite-type algebraic varieties over any field of characteristic zero.

 For other terminologies and properties concerning b-divisors we refer
 to \cite{bff} and \cite{bfj-val} (see also \cite{BoteroBurgos} and
\cite{bo} for the toroidal and the toric cases, respectively).

We write $\Div\left(X\right)$ for the set of Weil divisors on $X$ with
real coefficients, viewed as a real vector space. We endow it with
the direct limit topology with respect to its finite dimensional
subspaces. Explicitly, a sequence of divisors $(D_{i})_{i\ge
  0}$ converges to a divisor $D$ in $\Div\left(X\right)$ if there is a divisor $A$,
such that $\supp (D_{i})\subset \supp(A)$ for all $i\ge 0$ and
$(D_{i})_{i\ge 0}$ converges to $D$ in the finite dimensional vector
space of real divisors with support contained in $\supp(A)$.

\begin{df}\label{def:modi}
The set of \emph{models of $X$} is
\[
R(X) \coloneqq \left\{\pi \colon X_{\pi}\to X \; \big{|} \; \pi \text{
    is a smooth modification}\right\}.
\]
If $(X_{1},\pi _{1},D)$ is a toroidal structure on $X$ then the set
of \emph{toroidal models of $X$} (with respect to $D$) is  
\[
R^{\tor}(X,D) \coloneqq \left\{\pi \colon X_{\pi}\to X \mid \pi \text{
    smooth modification, toroidal w.r.t.} \, D \right\}.
\]
In particular if $\pi \in R^{\tor}(X,D)$ then it factors through
$\pi_{1}$. 
\end{df}

We view both $R(X)$ and $R^{\tor}(X,D)$ as full subcategories of the
category of complex 
varieties over $X$, in particular morphisms are over $X$. Maps of
models are unique if they exist, and are necessarily proper and
birational. Hironaka's resolution of singularities implies that $R(X)$ is a
directed set, where we set $\pi' \geq \pi$ if there exists a morphism
$\mu \colon X_{\pi'} \to X_{\pi}$. Similarly $R^{\tor}(X,D)$ is
directed by the existence of a smooth common refinement of any two
subdivisions.  

Consider a pair $\pi' \geq \pi$ in $R(X)$, and let $\mu \colon X_{\pi'}\to X_{\pi}$
be the corresponding modification. We have a pullback map
\[
\mu^* \colon \Div\left(X_{\pi}\right) \longrightarrow \Div\left(X_{\pi'}\right)
\]
and a pushforward map 
\[
\mu_* \colon \Div\left(X_{\pi'}\right) \longrightarrow \Div\left(X_{\pi}\right)
\]
between the associated divisor groups. Both  maps are continuous.  

\begin{df}\label{b-divisor}
The group of \emph{Cartier $\R$-b-divisors on} $X$ is the direct limit 
\[
\CbDiv(X) \coloneqq \varinjlim_{\pi \in R(X)} \Div\left(X_{\pi}\right),
\]
in the category of topological vector spaces, with maps given by
the pullback maps. The resulting topology
is called the \emph{strong} topology.

The group of \emph{Weil $\R$-b-divisors on} $X$ is the inverse limit 
\[
\WbDiv(X) \coloneqq \varprojlim_{\pi \in R(X)} \Div\left(X_{\pi}\right),
\]
in the category of topological vector spaces, with maps given by
the pushforward maps. The resulting topology
is called the \emph{weak} topology. Since the space of Cartier $\R$-b-divisors is contained in the space of Weil $\R$-b-divisors,  
we sometimes refer to Weil $\R$-b-divisors simply as
$\R$-\emph{b-divisors}. We also sometimes omit the ring of coefficients $\R$ from the notation. 
\end{df}

Usually we will denote $\R$-b-divisors with a blackboard bold 
font ($\mathbb D$) in order to distinguish them from classical $\R$-divisors
that will be denoted with a slanted font ($D$).

\begin{rmk} \label{rem:2}
  As a set, $\CbDiv(X)$ can be seen as the disjoint union of the sets
  $\Div\left(X_{\pi}\right)$ modulo the equivalence  relation which sets two
  divisors equal if they coincide after pullback to a common modification.
  The set $\WbDiv(X)$ can be seen as the subset of $\prod_{\pi\in R(X)
  }\Div(X_{\pi})$ given by the elements $\D=\left(D_{\pi }\right)_{\pi \in
    R(X)}$ satisfying 
  the compatibility condition that for each $\pi '\ge \pi $ we have
  $\mu _{\ast}D_{\pi '}=D_{\pi }$, where $\mu 
  $ is the corresponding map of modifications.
\end{rmk}

For any variety $Y$ and normal
crossings divisor $E$ on $Y$, we denote by $\Div(Y,E)$ the real vector
space of $\R$-divisors on $Y$ whose 
support is contained in $E$. We endow it with the Euclidean topology.

Let $(X_{1},\pi _{1},D)$ be a toroidal structure; we make analogous definitions of Weil and Cartier $\R$-b-divisors.

\begin{df}\label{t-b-divisor}
The group of \emph{toroidal Cartier $\R$-b-divisors on} $X$ (with
respect to $D$) is the direct limit 
\[
\CbDiv(X,D)^{\tor} \coloneqq \varinjlim_{\pi \in R^{\tor}(X,D)}
\Div\left(X_{\pi},\mu _{\pi }^{-1}(D)\right),
\]
where $\mu _{\pi }\colon X_{\pi }\to X_{1}$ is the unique map of
modifications. 
The limit is again in the category of topological vector spaces,
with maps given by
the pullback maps.

The group of \emph{toroidal Weil $\R$-b-divisors on} $X$  (with
respect to $D$) is the inverse limit 
\[
\WbDiv(X,D)^{\tor}\coloneqq \varprojlim_{\pi \in R^{\tor}(X,D)}
\Div\left(X_{\pi},\mu_{\pi } ^{-1}(D)\right),
\]
also in the category of topological vector spaces, with maps given by
the pushforward maps. Again, we sometimes simply refer to these as \emph{toroidal b-divisors}. 
\end{df}

We recall from \cite{BoteroBurgos} that the space
$\WbDiv(X,D)^{\tor}$ can be identified with the space of all conical
real valued functions on $\Pi (X_{1},D)(\Q)$, while
$\CbDiv(X,D)^{\tor}$ can be identified with the space of all continuous conical piecewise
linear functions on $\Pi (X_{1},D)$ with rational domains of linearity.

Clearly $R(X,D)^{\tor} \subset R(X)$. Moreover, for $\pi \in R(X,D)^{\tor}$
there are canonical maps
\begin{align*}
  &\Div\left(X_{\pi},\mu _{\pi }^{-1}(D)\right)\longrightarrow
  \Div\left(X_{\pi}\right),\\
  &\Div\left(X_{\pi}\right)\longrightarrow
  \Div\left(X_{\pi},\mu _{\pi }^{-1}(D)\right),
\end{align*}
  where the last one sends any prime divisor not contained in
  $\mu _{\pi }^{-1}(D)$ to zero. Hence  
there is a canonical inclusion
\begin{displaymath}
  \CbDiv(X,D)^{\tor}\longrightarrow \CbDiv(X)
\end{displaymath}
and a canonical projection
\begin{displaymath}
  \WbDiv(X) \longrightarrow \WbDiv(X,D)^{\tor}.
\end{displaymath}

We now construct a section 
\begin{equation}\label{eq:section}
  \iota\colon \WbDiv(X,D)^{\tor}\longrightarrow \WbDiv(X),
\end{equation}
of the canonical projection.
Let $\D\in \WbDiv(X,D)^{\tor}$ be a toroidal Weil $\R$-b-divisor and choose 
a smooth modification $\pi\in R(X)$ above $\pi_{1} $ and a 
prime divisor $P$ of $X_{\pi }$. We have to define $\ord_{P}(\iota
(\D))$ for any such $P$.

Let $\varphi_{\D}$ be the conical function on $\Pi (X_{\pi_{1} },D)(\Q)$
corresponding to the Weil $\R$-b-divisor $\D$. The function $\varphi_{\D}$
is characterized by the following property: If $E$ is a toroidal
exceptional  prime
divisor in some toroidal modification of $(X_{\pi_{1} },D)$, and
$v_{E}\in \Pi (X_{\pi _{1}},D)$
is the corresponding primitive vector, then
\begin{equation}\label{eq:conical_fun}
  \ord_{E}\D = -\varphi_{\D}(v_{E}).
\end{equation}
Let $(X_{2},\pi _{2},D')$ be a toroidal structure above $(X_{1},\pi _{1},D)$ that
has a map $f\colon X_{2}\to X_{\pi}$ over $X$ and such that
$f^{-1}(P)\subset D'$. Let $\widehat P_2$ be the strict transform of
$P$ in $X_{2}$. It is an irreducible component of $D'$.  Let
$v_{\widehat P_2}\in \Pi (X_{2},D')$ be the primitive vector corresponding to
$\widehat P_2$.  Then we define
\begin{displaymath}
  \ord_{\iota(\D)}P= -\varphi_{\D}(r_{D,D'}(v_{\widehat P_2})).
\end{displaymath}
To see that this is independent of the choices, let $(X_{3},\pi
_{3},D'')$ be another toroidal structure above $D$ and let
$\widehat P_{3}$ be the strict transform of $P$ in $X_{3}$ and
$v_{\widehat P_{3}}$ the corresponding primitive vector. Since
$r_{D',D''}(v_{\widehat P_{3}})=v_{\widehat P_2}$, from \eqref{eq:17} we obtain
\begin{displaymath}
  r_{D,D''}(v_{\widehat P_{3}})= r_{D,D'}(v_{\widehat P_2}). 
\end{displaymath}
Thus $\ord_{\iota(\D)}(P)$ does not depend on the choice of $D'$.

We will use the following terminology.
\begin{df}\label{def:toroidal-b} A Weil $\R$-b-divisor $\D\in \WbDiv(X)$ is called
\emph{toroidal with respect to  $(X_{1},\pi _{1},D)$} if it belongs to
$\iota(\WbDiv(X,D)^{\tor})$. It is called \emph{toroidal} if it is
toroidal  with respect to some toroidal structure. This gives rise to the
set of toroidal Weil $\R$-b-divisors
\begin{displaymath}
  \WbDiv(X)^{\tor}=\lim_{\substack{\longrightarrow\\(X_{1},\pi _{1},D)}}
  \iota\left(\WbDiv(X,D)^{\tor}\right) \subset \WbDiv(X). 
\end{displaymath}
Here the direct limit is taken with respect to the order $(X_{2},\pi
_{2},D') \geq (X_{1},\pi _{1},D)$ iff $D'$ is above $D$ and the
maps are just the pullback maps of divisors along the unique modification maps.
\end{df}
\begin{rmk}
Using resolution of singularities, one can see that  every Cartier $\R$-b-divisor is toroidal
with respect to some  toroidal structure. So, toroidal Weil $\R$-b-divisors
are in between Cartier and  Weil $\R$-b-divisors.
\end{rmk}
\begin{df}
 A Cartier $\R$-b-divisor is called \emph{nef} if it is nef in any
modification of $X$ where it is determined, and a Weil $\R$-b-divisor is
called \emph{nef} if it belongs to the closure (with respect to the weak topology) of
the space of nef Cartier $\R$-b-divisors. 
\end{df}
The following is a recent result by Dang and Favre: 
\begin{thm}[{\cite[Theorem A]{Da-Fa20}}]
Any nef Weil $\R$-b-divisor is the limit of a decreasing sequence of nef Cartier $\R$-b-divisors. 
\end{thm}
\begin{rmk}
If moreover the nef Weil $\R$-b-divisor is toroidal then by \cite[Lemma 5.9]{BoteroBurgos} the sequence can be taken to consist of toroidal nef Cartier $\R$-b-divisors (with respect to the same toroidal structure). 
\end{rmk}
\begin{rmk} \label{rk:degree} Using any approximating decreasing sequence of nef Cartier $\R$-b-divisors, each nef Weil $\R$-b-divisor $\D$ on $X$ has a well-defined degree $\D^d$ in $\R_{\ge 0}$; see \cite[Theorem~3.2]{Da-Fa20}.
\end{rmk}

\subsection{Psh functions, psh metrics, Lelong numbers}
\label{sec:psh-metrics-lelong}

 We briefly recall the notion of plurisubharmonic (psh) functions and metrics. A more thorough
 discussion with examples can be found in \cite{BBHJ}.

 The following definition is taken from \cite[Chapter 3]{ochia}. 
Assume for the moment that $X$ is any complex manifold.
\begin{df}\label{def:5} 
  Let $U$ be an open coordinate subset of $X$, identified with an
  open subset of $\C^{d}$. 
  A function $\varphi\colon U\to
  \R\cup \{-\infty\}$ is called \emph{plurisubharmonic (psh)} if 
 the following two conditions are satisfied:
  \begin{enumerate}
  \item the function $\varphi$ is upper-semicontinuous and not identically equal to
    $-\infty$ on  any connected component of $U$;
  \item for every $z\in U$ and every $a\in \C^{d}$ the function 
    \begin{displaymath}
     \C \to \C \, , \qquad  \zeta \longmapsto \varphi(z+a\zeta)\in \R\cup \{-\infty\} 
    \end{displaymath}
    is either identically $-  \infty$, or subharmonic in each connected
    component of the open set $\{\zeta \in \C\mid z+a\zeta\in
    U\}$. 
  \end{enumerate}
  Now let $U \subset X$ be an arbitrary open subset. A function $\varphi \colon U\to
  \R\cup \{-\infty\}$ is called $psh$ if $U$ can be covered by open coordinate subsets
  $U_{i}$ so that each restriction $\varphi|_{U_{i}}$ is psh.
\end{df}

We next discuss the notion of psh metrics. Let $L$ be a holomorphic line bundle on $X$.
\begin{df} \label{def:psh}
Let 
$\{(U_{i},s_{i})\}$ be a trivialization of $L$, with
  transition functions $\{g_{ij}\}$.  
A \emph{hermitian metric} $h$ on $L$ is a collection
of 
measurable functions 
\begin{displaymath}
\varphi_i \colon U_i \to \R \cup
  \left\{\pm \infty\right\}  \, ,
\end{displaymath}
 such that the identities
  \begin{equation}\label{eq:10}
e^{-\varphi_i} = |g_{ij}|e^{-\varphi_j}
\end{equation}
hold on all $U_i \cap U_j$.  The \emph{norm} $h(s_i)$ is given by the formula
\begin{displaymath}
  \varphi_{i}(z)=-\log h(s_{i}(z)),\quad z\in U_{i} \, .
\end{displaymath}
From the identities (\ref{eq:10}) we find
\begin{displaymath}
  \log h(s_{i})- \log h(s_{j})=\log|s_{i}/s_{j}| 
\end{displaymath}
on $U_i \cap U_j$.
More generally, when $s$ is any generating section of $L$ locally near a point $z \in X$ we define its norm $h(s)$ 
via 
\begin{displaymath}
  \log h(s(z))=\log|s(z)/s_{i}(z)| + \log h(s_{i}(z))=
  \log|s(z)/s_{i}(z)| - \varphi_{i}(z) 
\end{displaymath}
whenever $z\in  U_{i}$. This is easily seen to be independent of the choice of $i$.

We call the metric $h$ \emph{singular}
(resp.~\emph{psh}, \emph{continuous}, \emph{smooth}) if the functions
$\varphi_i$ are all locally integrable (resp.~psh, continuous,
smooth).  
\end{df}
An important measure of the singularities of a psh function is given by its Lelong
numbers.
\begin{df}\label{def:lelong_number}
  Let $U\subset X$ be an open coordinate subset, let $\varphi$ be a psh
  function on $U$, and let $x \in U$ be a point.
  Then the \emph{Lelong number} of $\varphi$ at $x$ is given as 
  \[
    \nu(\varphi,x) = \sup\left\{\gamma \geq 0
      \; \big{|} \; \varphi(z) \leq \gamma
      \log |z-x| + O(1) \; \text{near} \; x \right\}.
  \]
  The notion of Lelong number readily generalizes to the context of psh metrics. Let $L$ be a line bundle on $X$ equipped with a psh metric $h$ and let $x\in
  X$ be a point. Choose an open coordinate subset $x\in U\subset X$ and a generating section
  $s$ of $L$ on $U$. Then we put
  \begin{displaymath}
    \nu (h,x) = \nu(-\log h(s),x).
  \end{displaymath}
  It is readily verified that this is independent of the choice of open set $U$ and generating section $s$.
\end{df}

From now on assume again that $X$ is connected and projective.
Let $\pi \colon X'\to X$ be a smooth modification of $X$. Let
$(L,h)$ be a line bundle with a psh metric on $X$. It is easy to see that then
$(\pi^{\ast}L,\pi^{\ast}h)$ is a line bundle with a psh metric on
$X'$. For $x\in X'$ we then put
\begin{displaymath}
  \nu (h,x)=\nu (\pi^{\ast}h,x) \, .
\end{displaymath}
We note that if $P$ is a prime divisor of $X'$ and $x,y$ are very general
points on $P$, then the equality
\begin{displaymath}
  \nu (h,x)=\nu (h,y)
\end{displaymath}
holds. This leads to the following definition.

\begin{df}\label{def:2}
    Let $(L,h)$ be a line bundle equipped with a psh metric on $X$. Let 
    $\pi \in R(X)$ be a smooth modification, and let $P$ be  a prime divisor of $X_{\pi }$.
    Then we write
  \begin{displaymath}
    \nu(h,P)=\inf_{x\in P}\nu(h,x) \, .
  \end{displaymath}
\end{df}
The number $\nu(h,P)$ from Definition~\ref{def:2} allows the following alternative description.
\begin{lem}\label{lemm:8}
    Let $(L,h)$ be a line bundle provided with a psh metric on $X$, 
    let $\pi \in R(X)$ and let $P$ be a prime
    divisor of $X_{\pi }$. For every point 
 $x\in P$, for every generating section  $s$ of $L$ around $x$ and 
  for every local equation $g$ for $P$ at $x$, the equality 
    \begin{displaymath}
  \nu(h,P) = \sup\{\gamma \ge 0 \mid -\log(h(s))-\gamma \log| g| \text{
    bounded above near $x$}\}
\end{displaymath}
holds.
\end{lem}
\begin{proof}
  This follows from \cite[Proposition~10.5]{boucksom:singpsh}. 
\end{proof}

\subsection{Some notions of convex analysis}
\label{sec:convex-analysis}

In this section we recall some notions from convex analysis and prove
some lemmas that will be useful later. Our basic reference for
convex analysis is \cite{rockafellar70:_convex_analy}.

We fix a finite dimensional real vector space $N_{\R}$ and let $M_{\R}$ be its
dual. We recall the definition of closed convex function
(see \cite[Theorem 7.1]{rockafellar70:_convex_analy}).
\begin{df}
  Let $C\subset N_{\R}$ be a convex set. A convex function $f\colon C\to \R\cup \{\infty\}$ is called a
  \emph{closed convex} function if the equivalent conditions
  \begin{enumerate}
  \item the epigraph $\{(x,y)\in C\times \R\mid y\ge f(x)\}\subset
    N_{\R}\times \R$  is
    a closed subset;
  \item \label{item:15} the function $\widetilde  f\colon N_{\R}\to
    \R\cup \{\infty\}$, given by
    $\widetilde f(x)=f(x)$ for $x\in C$ and $\widetilde f(x)=\infty$
    if $x\not \in C$ is lower semicontinuous; 
  \end{enumerate}
  are satisfied.
\end{df}
Note that if $C$ is closed then condition \ref{item:15} is equivalent
to $f$ being lower semicontinuous. 
\begin{df}
  Let $C\subset N_{\R}$ be a convex
  set. The \emph{recession cone} of $C$ is the set of
direction vectors of all rays contained in $C$:
  \begin{displaymath}
    \rec(C) = \{y\in N_{\R}\mid C+y\subset C\}.
  \end{displaymath}
  The set $\rec(C)$ is a convex cone, which is closed if $C$ is
  closed, and polyhedral if $C$ is polyhedral.
\end{df}

\begin{df}
  Let $C\subset N_{\R}$ be a convex set  and $\sigma =\rec(C)$ its recession
  cone. Let $g\colon C\to \R$ be a closed convex function on
  $C$. The \emph{recession function} of $g$ is the function
  $\rec(g)\colon \sigma \to \R\cup\{\infty\}$ given, for any $x\in C$
  and any $y\in \sigma $ by
  \begin{equation}\label{eq:5}
    \rec(g)(y)=\lim_{\lambda \to +\infty}\frac{g(x+\lambda y)-g(x)}{\lambda }.
  \end{equation}
\end{df}

\begin{lem} \label{lemm:5} Let $C\subset N_{\R}$ be a convex set, 
  $\sigma =\rec(C)$ 
  its recession
  cone and $g\colon C\to \R$ a closed convex function on
  $C$.
  \begin{enumerate}
  \item \label{item:2} The recession function $\rec(g)$ is
    well defined. That is, the limit \eqref{eq:5} exists and does not
    depend on $x$. Moreover $\rec(g)$ is a closed convex function.
  \item \label{item:3} If the function $g$ is bounded above, then
    $\rec(g)$ takes non-positive values. In particular it takes finite
    values. 
  \item \label{item:4} If the function $g$ is Lipschitz
    continuous with constant $\alpha $ then $\rec(g)$ is also
    Lipschitz continuous with the same constant. 
  \end{enumerate}
\end{lem}
\begin{proof}
  The statement in (\ref{item:2}) is the content of
  \cite[Theorem~8.5]{rockafellar70:_convex_analy}.
  To prove (\ref{item:3}), let $x\in C$ and $y\in \rec(C)$ and consider
  two real numbers  $\lambda ,\mu >0$. We have
  \begin{displaymath}
    x+\lambda y =\frac{\mu }{\lambda +\mu }x +
    \frac{\lambda }{\lambda +\mu }(x+(\lambda +\mu )y).
  \end{displaymath}
  The convexity of $g$ yields
  \begin{equation}\label{eq:4}
    g(x+\lambda y )\le \frac{\mu }{\lambda +\mu }g(x) +
    \frac{\lambda }{\lambda +\mu }g(x+(\lambda +\mu )y).
  \end{equation}
  Since $g$ is assumed to be bounded above, say by a
  constant $B$, equation \eqref{eq:4} implies
  \begin{displaymath}
    g(x+\lambda y)\le
    \inf_{\mu} \frac{\mu }{\lambda +\mu }g(x) +
    \frac{\lambda }{\lambda +\mu } B = g(x),
  \end{displaymath}
  which implies that $\rec(g)$ takes non-positive values, and in
  particular finite values.  

  Statement (\ref{item:4}) follows from the computation
  \begin{displaymath}
  |\rec(g)(u)-\rec(g)(u')| = \lim_{\lambda \to
    +\infty}\frac{1}{\lambda }
  |g(v+\lambda u)-g(v+\lambda u')|\le \alpha \|u-u'\|. 
\end{displaymath}
\end{proof}

It is clear that $\rec(g)$ is conical, in the sense that, for all real
$\lambda \ge 0$ and all $y\in \sigma $
\begin{equation}
  \label{eq:6}
  \rec(g)(\lambda y)=\lambda \rec(g)(y).
\end{equation}
Recall that, if $\sigma \subset N_{\R}$ is a cone, the dual cone $\sigma
^{\vee}\subset M_{\R}$ is given by
\begin{displaymath}
  \sigma ^{\vee}=\{m\in M_{\R}\mid m(x)\ge 0,\ \forall x\in \sigma \}.
\end{displaymath}

\begin{lem}\label{lemm:15}
  Let $C$ be a convex set and $\sigma =\rec(C)$ its recession
  cone. Let $g\colon C\to \R$ be a bounded-above closed convex
  function on
  $C$ and $\rec(g)$ its recession function.
  \begin{enumerate}
  \item \label{item:5} For every $x\in C$ the function on $\sigma $ given by
    \begin{displaymath}
      y\longmapsto g(x+y)-\rec(g)(y)
    \end{displaymath}
    is bounded above.
  \item \label{item:6} Assume that $C$ is closed and $g$ is Lipschitz continuous. Then,
    for every $m\in \inter(\sigma^{\vee}) $ the function  
    \begin{displaymath}
      y\longmapsto g(x+y)-\rec(g)(y)+m(y)
    \end{displaymath}
    is bounded below.
  \item \label{item:7} Assume that $\sigma $ is full dimensional. Then
   for every $0\not = m\in \sigma^{\vee}$ the function 
    \begin{displaymath}
      y\longmapsto g(x+y)-\rec(g)(y)+m(y)
    \end{displaymath}
    is not bounded above.  
  \end{enumerate}
\end{lem}
\begin{proof}
  Let $\lambda >0$. The convexity of $g$ yields
  \begin{displaymath}
    g(x+y)\le \frac{1}{\lambda }g(x+\lambda
    y)+\left(1-\frac{1}{\lambda }\right)g(x) 
  \end{displaymath}
  which implies
  \begin{displaymath}
    g(x+y)-g(x)\le \frac{g(x+\lambda y)-g(x)}{\lambda }.
  \end{displaymath}
  Hence $g(x+y)-\rec(g)(y)\le g(x)$ proving (\ref{item:5}).

  We prove (\ref{item:6}) by contradiction. If the function in
  (\ref{item:6}) is not bounded below, we can find a sequence of points
  $y_{i}\in \sigma $ for $i\ge 1$ satisfying $\|y_{i}\|\ge i$ and
  \begin{equation}\label{eq:7}
    g(x+y_{i})-\rec(g)(y_{i})+m(y_{i})\le -i.
  \end{equation}
  After choosing a subsequence we can further assume that the sequence
  $(y_{i}/\|y_{i}\|)$ converges to a point $y_{0}\in \sigma $.  We
  write $\lambda _{i}=\|y_{i}\|$ and $y_{i}^{0}=y_{i}/\lambda _{i}$.
Equation \eqref{eq:7} implies that
\begin{equation}\label{eq:8}
  \limsup _{i\to \infty} \frac{1}{\lambda _{i}}g(x+\lambda
  _{i}y^{0}_{i})-\rec(g)(y^{0}_{i})+m(y^{0}_{i})\le 0.
\end{equation}
Using that $g$ is Lipschitz continuous we get
\begin{equation}
  \label{eq:9}
  \lim_{i\to \infty}
  \frac{g(x+\lambda_{i}y^{0}_{i})-g(x+\lambda_{i}y_{0})}{\lambda
    _{i}}=0. 
\end{equation}
  By Lemma \ref{lemm:5}.\ref{item:3} the function $\rec(g)$ is (Lipschitz)
  continuous. Since $m$ is also continuous we obtain
  \begin{equation}
    \label{eq:11}
    \lim_{i\to \infty}-\rec(g)(y^{0}_{i})+m(y^{0}_{i})
    +\rec(g)(y_{0})-m(y_{0})=0.
  \end{equation}
By Lemma \ref{lemm:5}.\ref{item:2} we know that
\begin{equation}
  \label{eq:12}
  \lim_{i\to \infty} \frac{1}{\lambda _{i}}g(x+\lambda_{i}y_{0})-\rec(g)(y_{0})=0.
\end{equation}
Since $m\in \inter(\sigma ^{\vee})$ and $0\not= y_{0}\in \sigma $ we
deduce
\begin{equation}
  \label{eq:13}
  m(y_{0})>0.
\end{equation}
Summing up equations \eqref{eq:9} to \eqref{eq:13} we obtain 
\begin{displaymath}
    \lim_{i\to \infty} \frac{1}{\lambda _{i}}g(x+\lambda
  _{i}y^{0}_{i})-\rec(g)(y^{0}_{i})+m(y^{0}_{i})= m(y_{0})>0,
\end{displaymath}
contradicting the inequality \eqref{eq:8}.

We next prove (\ref{item:7}). The fact that $\sigma $ is full
dimensional implies that $\{0\}$ is a face of $\sigma ^{\vee}$. Since
$0\not = m\in \sigma^{\vee}$, there 
is a $y_{0}\in \sigma $ such that $m(y_{0}) = \epsilon >0$. Fix
$x\in C$ and consider the one variable bounded above convex function
\begin{displaymath}
  g_{0}(t)=g(x+ty_{0})
\end{displaymath}
Since $g_{0}$ is a bounded above convex function it is Lipschitz
continuous for $t\ge 1$. The recession function of $g_{0}$ is
\begin{displaymath}
  \rec(g_{0})(t)=\rec(g)(t y_{0}).
\end{displaymath}
We can apply (\ref{item:6}) to this function to
deduce that
\begin{math}
  g_{0}(t)-\rec(g_{0})(t) + (\epsilon /2) t
\end{math}
is bounded below. Therefore
\begin{displaymath}
  g(x+ty_{0})-\rec(g)(t y_{0})+ m(ty_{0})=
  g_{0}(t)-\rec(g_{0})(t) + \epsilon  t
\end{displaymath}
can not be bounded above.
\end{proof}

\subsection{Toroidal psh metrics and their Lelong numbers}
\label{sec:toroidal-psh-metrics}

We fix now a toroidal structure  $(X_{1},\pi _{1},D)$ on $X$ and write $U=X
_{1}\setminus D$. Recall that $X$ denotes a connected complex projective manifold and that 
we can view $U$ as an open subset of $X$.

\begin{df}\label{def:9} (Compare with \cite[Definition~3.10]{BBHJ})
  A psh metric $h$ on a line bundle $L$ on $X$ is called \emph{toroidal with
    respect to $D$} if the following is satisfied.
    \begin{itemize}
    \item $h$ is locally bounded on $U$.
    \item for every
  point $p\in D$ there is an open coordinate polydisc $W$ with coordinates
  $(z_{1},\dots,z_{d})$  with $|z_{i}|< e^{-c}$ for some constant $c$ such that
  \begin{displaymath}
    W\cap D=\{z_{1}\cdots z_{r}=0\},
  \end{displaymath}
  a generating section $s$ of $L$ on $W$, a bounded function $\gamma $
  on $W$ and a bounded above convex Lipschitz continuous function $g$ on the cone $\R_{\ge
    c}^{r}$ 
  such that
  \begin{displaymath}
    -\log h(s)(z_{1},\dots,z_{d})=
    \gamma (z_{1},\dots,z_{d}) + g(-\log|z_{1}|,\dots,-\log|z_{r}|).
  \end{displaymath}
  \end{itemize}
  We call $D$ \emph{a singularity divisor of} $h$.
  \end{df}
  \begin{rmk}
  If $h$ is toroidal with respect to $D$ and $D'$ is a toroidal
  structure above $D$ then $h$ is also toroidal with respect to $D'$. 
\end{rmk}

Let $L$ be a line bundle with toroidal psh metric $h$ (with respect to $D$). 
We now explain how to compute the Lelong numbers of $h$
along toroidal exceptional divisors. 

Let $\pi\in R(X,D)^{\tor }$ be a toroidal model of $(X,D)$, $\mu \colon
X_{\pi }\to X_{1}$ the corresponding map of modifications and $P$ a
prime toroidal exceptional divisor in $X_{\pi }$. Then $P$ corresponds
to a rational  
ray $\rho _{P}$ in a cone $\sigma$ of $\Pi (X_{1},D)$. Let $y\in P$ be a
generic point and let $x\in X_{1}$ be
the image of $y$. Let $W$ be a coordinate neighborhood of
$x$ as in Definition \ref{def:9}, $s$ a generating section of $L$
around $x$ and $g$ the convex function of the
same definition. The cone $\sigma $ is smooth and can be identified
with $\R_{\ge 0}^{r}$ using its integral structure. This
identification is canonical up to the ordering of the variables that
can be fixed by the choice of coordinates. Then the function
$\rec(g)$ is canonically a function on $\sigma $.

\begin{lem} \label{lemm:16}
  Let $v_{P }$ be the primitive generator of $\rho _{P}$. Then
  \begin{displaymath}
    \nu (h,P)=-\rec(g)(v_{P})
  \end{displaymath}
\end{lem}
\begin{proof}
  Choose a small coordinate neighborhood $V$ of $y$ with coordinates
  $(x_{1},\dots,x_{d})$ such that $x_{1}$ is a local equation for
  $P$. Since $y$ is generic and $V$ small we can assume that
  \begin{equation}
    \label{eq:14}
    \pi ^{-1}(D)\cap V = P\cap V.
  \end{equation}

  Let $v_{P}=(a_{1},\dots,a_{r})$. Then the map $\pi $ can be
  written as
  \begin{displaymath}
    \pi
    (x_{1},\dots,x_{d})=
    (x_{1}^{a_{1}}u_{1},\dots,x_{1}^{a_{r}}u_{r},u_{r+1},\dots,u_{d}),
  \end{displaymath}
  where the $u_{i}$ are functions on $V$. The condition \eqref{eq:14}
  implies that $u_{i}$ does not vanish on $V$ for $i=1,\dots,r$. Then
  \begin{displaymath}
    -\log h(s) = \gamma
    +g(-\log|u_{1}|-a_{1}\log|x_{1}|,\dots,-\log|u_{r}|-a_{r}\log|x_{1}|).  
  \end{displaymath}
  Then Lemma \ref{lemm:15}.\ref{item:5} implies that
  \begin{displaymath}
    -\log h(s) -\rec(g)(v_{P })(-\log|x_{1}|)
  \end{displaymath}
  is bounded above, but by Lemma \ref{lemm:15}.\ref{item:7} for every $\epsilon >0$
  \begin{displaymath}
    -\log h(s) -\rec(g)(v_{P })(-\log|x_{1}|)+\epsilon (-\log|x_{1}|)
  \end{displaymath}
  is not bounded above.
  By Lemma \ref{lemm:8} we obtain that
  \begin{displaymath}
    \nu (h,P)=-\rec(g)(v_{P }) 
  \end{displaymath}
  as required.
\end{proof}

\section{Psh metrics, b-divisors and graded linear series}
\label{sec:psh-metrics-b}

In this section we relate graded linear series, b-divisors and psh
metrics to one another. Recall that $X$ denotes a
connected complex projective manifold of dimension $d$.

\subsection{From psh-metrics to graded linear series}
\label{sec:psh-metrics-b-1}

Let $F=K(X)$ denote the field of rational functions on $X$. Let
$U\subset X$ be a dense Zariski open
subset and $(L, h)$ a line bundle on $X$ with a
psh metric (see Definition \ref{def:psh}). We assume that $(L, h)|_{U}$ is
locally bounded (for example continuous or smooth). 
Following Nystr\"om \cite[Section~2.8]{Nystrom_growth} we define
\begin{displaymath}
  H^{0}(X,L,h)
  = \{s\in H^{0}(X,L)\mid h(s) \text{ bounded}\}.
\end{displaymath}

For $s$ a non-zero rational section of $L$ we define
\begin{align*}
  \caL(L,s,h) &= \{f\in F\mid fs\in H^{0}(X,L, h)\},\\
  \caR(L,s,h)
  &=\bigoplus_{\ell \in \mathbb Z_{\ge 0}}
   \caL(L^{\otimes\ell},s^{\otimes\ell},h^{\otimes\ell})t^{\ell}
    \subset F[t] \, . 
\end{align*}
Note that the latter is a graded linear series in the sense of Definition \ref{def:1}. 

\subsection{From graded linear series to b-divisors}
\label{sec:from-graded-linear}

Let $A=\bigoplus_{\ell \geq 0} A_{\ell}t^{\ell}\subset F[t]$ be a
graded linear series of almost integral type, $\pi \in R(X)$ and
$X_\pi$ the corresponding model of $X$.

\begin{lem}\label{lemm:6} The set
  \begin{displaymath}
  S= \{D\in \Div(X_{\pi })\; \big{|}\; \forall \ell \ge 0,\ \forall f\in A_{\ell}\setminus\{0\},  \ell D+\dv(f)\ge 0\}
  \end{displaymath}
  is not empty.
\end{lem}
\begin{proof}
Since $A$ is of almost integral type there is a graded linear series
$A'$ of integral type containing $A$. Every graded linear series of
integral type is finitely generated. Therefore there are nonzero elements
$f_{i}\in F$, $i=1,\dots,r$ and positive integers $\ell_i$, $i=1,\dots,r$ such that $A$ is contained in the
subalgebra of $F[t]$ generated by the set $f_{i}t^{\ell_{i}}$, $i=1,\dots,r$.
Therefore
\begin{displaymath}
  S\supset \{D\in \Div(X_{\pi })\; \big{|}\; \ell_{i} D+\dv(f_{i})\ge 0,\
  i=1,\dots,r\}. \qedhere
\end{displaymath}
\end{proof}

\begin{df}\label{def:3}
  Let $A$ be a graded linear series on $X$ of almost integral type. We 
define 
\begin{displaymath}
  \bdiv(A)=(\bdiv(A)_{\pi })_{\pi \in R(X)}\in \WbDiv(X)
\end{displaymath}
by
\begin{displaymath}
  \bdiv(A)_{\pi }=\inf \{E\in \Div(X_{\pi })\; \big{|}\;
  \ell E+\dv(f)\ge 0,\
  \forall \ell\ge 0\ \forall f\in A_{\ell}\},
\end{displaymath}
where the infimum is defined componentwise. In other words, for every
prime divisor $P$ on $X_{\pi }$ we put
\begin{displaymath}
  \ord_{P}(\bdiv(A)_{\pi })=\sup \left\{\frac{-1}{\ell}\ord_{P}(f) \mid \ell \ge 0,
  f\in A_{\ell}\right\}.
\end{displaymath}  
\end{df}

Note that $\bdiv(A)$ is well defined thanks to Lemma 
\ref{lemm:6}.

\subsection{From psh metrics to b-divisors}
\label{sec:from-psh-metrics}

\begin{df}\label{def:4}
  Let $(L, h)$ be a line bundle on $X$ together with a psh metric such that
  there is a dense open Zariski subset $U \subset X$ on which $h$ is locally
  bounded, 
  and let $s$ be a non-zero rational section of $L$. The Weil
  $\R$-b-divisor $ \D(L,s,h)$ is defined, for every $\pi \in R(X)$
  and prime divisor $P$ on $X_{\pi }$, by
  \begin{displaymath}
    \ord_{P} \D(L,s,h)_{\pi }
    =
    \ord_{P} \dv(s) - \nu(h,P). 
  \end{displaymath}
\end{df}
Here $\nu(h,P)$ denotes the Lelong number from Definition
\ref{def:lelong_number}.

\begin{prop}\label{prop:3}
  The Weil $\R$-b-divisor $ \D(L,s,h)$ is nef.
\end{prop}
\begin{proof}
  This is proved in \cite[Theorem 5.18]{BBHJ}.
\end{proof}

\begin{prop}\label{prop:div_of_toroidal_metric_is_toroidal}
  Let $(X_{1},\pi _{1},D)$ be a toroidal structure on $X$. Assume that the
  psh metric $h$
  is toroidal with respect to $D$ in the sense of Definition~\ref{def:9}.
  \begin{enumerate}
  \item\label{i:1} The Weil $\R$-b-divisor $\D=\D(L,s,h)-\dv(s)$ is toroidal with respect
    to $D$. Note that this divisor only depends on the singularities
    of $h$ and not on the particular section $s$.
  \item\label{i:2} Let $\sigma $ be a cone of $\Pi
    (X_{1},D)$ and $Y$ the corresponding stratum. Let $\varphi
    _{\D}$ be the conical function on $\Pi
    (X_{1 },D)$ corresponding to $\D$ (see \eqref{eq:conical_fun}). Choose a
    generic point $x$ of $Y$ and a small enough coordinate neighborhood $W$  of $x$
    with $W\cap D$ given by the equation $z_{1}\cdots z_{r}=0$. After
    changing $s$ if needed, we can assume
    that $s$ is a generating section in $W$. Therefore, by the
    definition of toroidal psh metric, there is  
    a bounded above convex Lipschitz continuous function $g$ on a cone $\R_{\ge
    c}^{r}$ 
  such that
  \begin{displaymath}
    -\log h(s)(z_{1},\dots,z_{d}) -
     g(-\log|z_{1}|,\dots,-\log|z_{r}|)
  \end{displaymath}
 is bounded on $W \cap D$.  Then
    \begin{displaymath}
      \varphi_{\D}|_{\sigma }=-\rec(g).
    \end{displaymath}
  \end{enumerate}
\end{prop}
\begin{proof}
  Let $\D'$ be the projection of $\D$ onto
  $\WbDiv(X,D)^{\tor}$. In order to prove \eqref{i:1} we have to show
  that 
  \begin{equation}\label{eq:20}
\iota(\D')=\D,
\end{equation}
where $\iota$ is the section \eqref{eq:section}
from toroidal Weil $\R$-b-divisors to general Weil $\R$-b-divisors. And for \eqref{i:2} we
have to show that
  \begin{equation}\label{eq:19}
    \varphi_{\D'}|_{\sigma } = -\rec(g).
  \end{equation}
  In view of the definition of $\D(L,s,h)$ using Lelong numbers,
  equation \eqref{eq:19} is just a reformulation of Lemma
  \ref{lemm:16}. So it only remains to prove equation \eqref{eq:20}.

  Let $\pi \colon X_{\pi }\to X$ be a smooth modification and $P$
  a prime divisor of $X_{\pi }$. We have to check that
  \begin{equation}
    \label{eq:21}
    \ord_{P}\D = \ord_{P}\iota(\D').
  \end{equation}
  Let $(X_{2},\pi _{2},D')$ be a toroidal structure above $D$
  with a map $f\colon X_{2}\to X_{\pi }$ and such that
  $f^{-1}(P)\subset D'$. Let $\widehat P$ be the strict transform of
  $P$ in $X_{2}$. In order to prove \eqref{eq:21} it is enough to
  prove $\ord_{\widehat P}\D = \ord_{\widehat P}\iota(\D')$. Let
  $v_{P}$ be the primitive vector in $\Pi (X_{2},D')$
  corresponding to $\widehat P$, $\sigma $ the minimal cone of $\Pi
  (X_{1},D)$ containing $r(D,D')(v_{P})$, $Y$ the stratum
  of $X_{1}$ corresponding to $\sigma $, $x$ a generic point of
  $Y$, and $W$ a small enough coordinate neighborhood of $x$ with
  coordinates $z_{1},\dots, z_{d}$ such that $z_{1}\cdots z_{r}=0$ is
  an equation of $D\cap W$. Write $D_{i}$ for the divisor $z_{i}=0$,
  $i=1,\dots,r$, and $v_{i}$ for the corresponding primitive vector of $\Pi
  (X_{1},D)$. Then the cone $\sigma $ is generated by
  $v_{1},\dots,v_{r}$.

  Let $\mu \colon X_{2}\to X_{1}$ be the map of
  modifications. 
  Write $a_{i}=\ord_{\widehat P}\mu^{\ast}z_{i}$. Then
  \begin{equation}
    \label{eq:22}
    r_{D,D'}(v_{P})=\sum _{i=1}^{r}a_{i}v_{i}, 
  \end{equation}
and therefore
  \begin{displaymath}
    \ord_{P}\iota(\D')= \ord_{\widehat
      P}\iota(\D')=\sum_{i=1}^{r}a_{i}\varphi_{\D'}(v_{i}). 
  \end{displaymath}
  Let $y$ be a generic point of $\widehat P$ above $x$ and
  choose a small enough coordinate neighborhood $V$ with coordinates
  $(x_{1},\dots,x_{d})$ such that $x_{1}$ is a local equation  for
  $\widehat P$. As in the proof of Lemma \ref{lemm:16}, the map $\mu $
  can be written as
  \begin{displaymath}
    \mu
    (x_{1},\dots,x_{d})=
    (x_{1}^{a_{1}}u_{1},\dots,x_{1}^{a_{r}}u_{r},u_{r+1},\dots,u_{d}),
  \end{displaymath}
  where the $u_{i}$ are holomorphic functions on $V$, and $u_{i}$ does not
  vanish on $V$ for $i=1,\dots,r$. Choose a rational section $s$ of
  $L$ that generates $L$ around $x$, so that
  \begin{displaymath}
    -\log h(s)=\gamma + g(-\log|z_{1}|,\dots, -\log|z_{r}|)
  \end{displaymath}
  and $\varphi_{\D'}|_{\sigma }=-\rec(g)$. Then around $y$ we have
    \begin{displaymath}
    -\log h(s) = \gamma
    +g(-\log|u_{1}|-a_{1}\log|x_{1}|,\dots,-\log|u_{r}|-a_{r}\log|x_{1}|).  
  \end{displaymath}
Hence
\begin{displaymath}
  \ord_{\widehat P}(\D')=-\nu (h,P)=
  \rec(g)\left(\sum_{i=1}^{r}a_{i}v_{i}\right)
  =-\varphi_{\D'}(r_{D,D'}(v_{\widehat P}))=\ord_{P}\iota(\D'),
\end{displaymath}
proving the result. 
\end{proof}

Proposition \ref{prop:div_of_toroidal_metric_is_toroidal} has the
following consequence. 

\begin{cor} \label{cor:4} With the hypothesis of Proposition~\ref{prop:div_of_toroidal_metric_is_toroidal}, if the Weil $\R$-b-divisor $\D(L,s,h)$ is Cartier, then for every cone $\sigma \in \Pi
(X_{1 },D)$ and every decomposition
\begin{displaymath}
  -\log h(s)(z_{1},\dots,z_{d})=\gamma +g(-\log|z_{1}|,\dots,-\log|z_{r}|),
\end{displaymath}
with $\gamma$ locally bounded  and $g$ bounded above, convex and Lipschitz
continuous, the function $\rec(g|_\sigma)$ is piecewise linear. 
\end{cor}
\begin{proof}
If $\D$ is Cartier, we know that $\varphi_{\D}|_{\sigma }$ is
piecewise linear. Hence the corollary is an immediate consequence of Proposition \ref{prop:div_of_toroidal_metric_is_toroidal}.\ref{i:2}.    
\end{proof}

\subsection{From b-divisors to graded linear series}
\label{sec:from-b-divisors}

Let $\D$ be a Weil $\R$-b-divisor (not necessarily toroidal).  We define
\begin{align*}
  \caL(\D)&=\{0\not = f \in F\mid \D_{\pi }+\dv(f) \ge 0,\ \forall \pi \in R(X)
               \}\cup \{0\},\\
  \caR(\D) &= \bigoplus _{\ell \in \mathbb Z_{\ge 0}} \caL(\ell \D)t^{\ell}\subset F[t].
\end{align*}
The latter is a graded linear series.

\begin{lem}\label{lemm:7}
  The graded linear series $\caR(\D)$ is of almost integral type.
\end{lem}
\begin{proof}
  Let $\pi \in R(X)$. From the definition it follows that
  $\caR(\D)\subset \caR(\D_{\pi })$. By Proposition \ref{prop:2} we know 
  $\caR(\D_{\pi })$ is of almost integral type, so it is contained in
  a graded linear series $A$ of integral type. 
\end{proof}

\subsection{Summarizing the relations}
Combining the previous subsections we obtain a diagram
\begin{align}\label{diag:relations}
  \xymatrix{
    &\text{graded linear series}\ar@<.5em>[dd]^{\bdiv}\\
    \text{psh-metrics}\ar^{\caR}[ur]\ar_{\D}[dr]&\\
  &\text{b-divisors}\ar@<.5em>[uu]^{\caR}}
\end{align}
This diagram is in general not commutative. Much of the remainder of this section will be taken up with verifying certain weaker relations in this diagram. 
\begin{lem}\label{lemm:3} Let $\D$ and $\D'$ be Weil $\R$-b-divisors and $A$ and
  $A'$ graded linear series.
  \begin{enumerate}
  \item $\D\le \D' \Longrightarrow \caR(\D)\subset \caR(\D')$.
  \item $A\subset A' \Longrightarrow \bdiv(A)\le \bdiv(A')$.
  \item \label{item:14} $\bdiv(\caR(\D))\le \D$.
  \item $A\subset \caR(\bdiv(A))$. 
  \end{enumerate}
  The same is true if we replace b-divisors by toroidal b-divisors.
\end{lem}
\begin{proof}
  These follow directly from the definitions.
\end{proof}

\begin{lem}\label{lemm:4} Let $(L,h,s)$ be a line bundle on $X$ with a
  psh metric $h$ and a non-zero rational section $s$ of $L$. Then
  \begin{displaymath}
    \caR(L,s,h) \subset
    \caR(\D(L,s,h)).
  \end{displaymath}
\end{lem}
\begin{proof}
  The elements of $\caR(L,s,h)_{\ell}$ are rational functions 
  $f$ with $h(fs^{\otimes \ell})$
  bounded, while the elements of $\caR(\D(L,s,h))_{\ell}$ are rational
  functions $f\in \caL(\ell \dv(s))$ with enough zeroes to
  cancel the Lelong numbers of $h$. 
  The result then follows from the fact that bounded functions
  have zero Lelong numbers; in what follows we give some more details.
 
 Let $f \in \caR(L, s,h)_{\ell}$. From Definition \ref{def:4} we can write $\D
 = \D(L, s,h ) = (\D_{\pi})_{\pi \in R(X)}$ with 
  \[
  \D_{\pi} = \pi^*\dv(s) - \sum_P\nu(h,P)P,
  \]
  where the sum is over all prime divisors $P$ on $X_{\pi}$ and 
  $\nu(h,P)$ denotes the Lelong number of the
  metric $h$ at a very general point of $P$. 
  Let $x$ be a very general point of $P$,  $s_{0}$ a local
  generating section of $L$ at $x$, and $g$ a local equation 
  of $P$ at $x$. By Lemma \ref{lemm:8} we have
  \begin{equation}
  \nu(h,P) = \sup\{\gamma \ge 0 \mid -\log(h(s_{0})| g|^\gamma) \text{
    bounded above near $x$}\}. 
  \end{equation}
Since $h(fs^{\otimes \ell})$ is bounded we know  that
$-\log(h(s_{0})^{\ell}|g|^{\ord_{P}(fs^{\otimes \ell})})$ is bounded
below. 
Therefore, for every $\epsilon >0$,
$-\log(h(s_{0})^{\ell}|g|^{\epsilon +\ord_{P}(fs^{\otimes \ell})})$ is \emph{not}
bounded above. Hence $  \ord_P(fs^{\otimes \ell}) \geq \ell \nu(h,P)$ and therefore
  \[
  \dv(f) + \ell \D_{\pi} = \dv(f) + \ell \pi^*\dv(s) - \ell \sum_P\nu(h,P)P \geq 0,
  \]
so $f \in \caR(\D(L,s,h))_\ell$.  
\end{proof}

In the toroidal case we also have an inclusion in the reverse
direction.

\begin{lem}\label{lemm:14}
  Let $(X_{\pi },D)$ be a toroidal structure on $X$ and let $h$ be a
  psh metric on a line bundle $L$ on $X$ that is toroidal with respect to
  $D$. Then for every real number $\epsilon >0$ there is an inclusion
    \begin{displaymath}
      \caR(\D(L,s,h)-\epsilon D) \subset
       \caR(L,s,h).
    \end{displaymath}
  \end{lem}
  \begin{proof}
    Let $f\in \caR(\D(L,s,h)-\epsilon D)_{\ell}$ and let $p\in X_{\pi }$.
    We need  to show that $\|fs^{\otimes \ell}\|$ is bounded near $p$,
    or equivalently that $-\log \|fs^{\otimes \ell} \|$ is bounded below near $p$. Let $s'$ be
    another rational section of $L$ with $s=vs'$ for some rational
    function $v$ and let $f'=fv^{\ell}$. Then
    \begin{displaymath}
      f\in \caR(\D(L,s,h)-\epsilon D)_{\ell}
      \Longleftrightarrow
      f'\in \caR(\D(L,s',h)-\epsilon D)_{\ell},
    \end{displaymath}
    and $\|fs^{\otimes \ell}\|$ is bounded if and only if
    $\|f's'{}^{\otimes \ell}\|$ is
    bounded. Therefore we can assume
    that $s$ is a generating section around $p$.
    
    Choose a
    small enough coordinate system $W$ around $p$ as in Definition~\ref{def:9} and let $g$ and $\gamma $ be the functions appearing in
    that definition. By Lemma~\ref{lemm:16}, the condition 
    $f\in \caR(\D(L,s,h)-\epsilon D)_{\ell}$ implies
    that, on $W$, we can write
    \begin{displaymath}
      f=z_{1}^{m_{1}}\cdots z_{r}^{m_{r}}f_{0}
    \end{displaymath}
    where $f_{0}$ is holomorphic, not divisible by $z_{1},\dots,z_{r}$
    and the exponents $m_{i}$ are integers 
    with the property that for every primitive vector
    $u=(u_{1},\dots,u_{r})$ corresponding to a prime divisor
    $P_{u}$ in an allowable modification, we have
    \begin{align}\label{eq:15}
      \sum _{i=1}^{r}m_{i}u_{i} &= \ord_{P_{u}}f \nonumber\\
      &\ge 
      -\ell \ord_{P_{u}}(
      \D(L,s,h)-\epsilon D)
    \\ &= -\ell \rec(g)(u)+\ell \epsilon \sum_{i=1}^{r}u_{i}.
    \nonumber \end{align}
    Since $\rec(g)$ is conical and continuous on $\R^{r}_{\ge 0}$ we
    deduce that, for every vector $u\in \R^{r}_{\ge 0}$, the inequality
    \eqref{eq:15} holds. We compute
    \begin{displaymath}
      -\log\|fs^{\otimes \ell}\| =
      -\log|f_{0}| - \sum_{i=1}^{r}m_{i}\log|z_{i}| +\ell \gamma +
      \ell g(-\log|z_{1}|,\dots,-\log|z_{r}|).
    \end{displaymath}
    Using that $f_{0}$ is holomorphic, so $-\log|f_{0}|$ is bounded
    below, and that $\gamma $ is bounded near $p$, the condition
    \eqref{eq:15} implies that  there is a constant $B$ with
    \begin{displaymath}
      -\log\|fs^{\otimes \ell}\| \ge B+\ell g -\ell \rec(g)+\ell
      \epsilon \sum_{i=1}^r (-\log|z_{i}|).
    \end{displaymath}
    By Lemma \ref{lemm:15}.\ref{item:6} the quantity on the right is bounded below,
   hence we obtain the result. 
  \end{proof}

  \begin{cor} \label{cor:3}  Take the assumptions of
    Lemma~\ref{lemm:14}. If for every irreducible component $D_{i}$ of
    $D$ 
    the condition $\ord_{D_{i}}\D(L,s,h)>0$ holds then, for every $\epsilon >0$,
    \begin{displaymath}
      \caR((1-\epsilon )\D(L,s,h)) \subset
      \caR(L,s,h).
    \end{displaymath} 
  \end{cor}

\subsection{The case of a divisor generated by global sections}
\label{sec:case-divis-gener}


\begin{prop}\label{prop:1}
  Let $\D$ be a Weil $\R$-b-divisor such that there is an $e \in
  \mathbb Z_{>0}$ with  $e\D$ a globally
  generated integral Cartier divisor on some proper modification
  $X_{\pi}$ of $X$. Then
  \begin{displaymath}
    \bdiv(\caR(\D))=\D.
  \end{displaymath}
\end{prop}
\begin{proof}
By Lemma \ref{lemm:3} we know that $\bdiv(\caR(\D))\le \D$.  On the other hand
  \begin{displaymath}
    \bigoplus _{\ell \in \mathbb Z_{\ge 0}} \caL(\ell e \D_{\pi })t^{\ell e} \subset \caR(\D);
  \end{displaymath}
  indeed, if $e\in \mathbb Z_{\ge 0}$ and $f \in \caL(\ell e \D_{\pi
  })$ is non-zero, then $\dv(f) + \ell e \D_{\pi} \ge 0$, and the
  same holds on every model $X_{\pi'}$ of $X$ since pullback and
  pushforward preserve effectivity, hence $f \in \caR(\D)$.  
  
  Now let $s_{1},\dots ,s_{n}$ be a set of generating sections of $\caO(e\D_{\pi})$, 
  and $s$ the canonical rational section of $\caO(e\D_{\pi })$ with
  $\dv(s)=e\D_{\pi }$. 
  Write $s_{i}=f_{i}s$. Since $s_{i}$ is a global section, we have
  $e\D_{\pi}+\dv(f_{i})\ge 0$ and 
  $f_{i}\in \caR(\D)_{e}$.

  Since the sections $s_{i}$ generate, for every prime divisor $P$ on
  every modification $\pi' \in R(X)$ above $\pi $, there is an $i$ such that
\begin{equation}\label{eq:asd1}
    \ord_{P}(e\D_{\pi' }+\dv(f_{i}))=0.
\end{equation}
  On the other hand, since $f_{i}\in \caR(\D)_{e}$ we see from the
  definition of the group $\caR(\D)$ that 
  \begin{equation}\label{eq:asd2}
  \ord_{P}(e\bdiv(\caR(\D))+\dv(f_{i}))\ge 0. 
  \end{equation}

Combining \eqref{eq:asd1} and \eqref{eq:asd2} we deduce 
  \begin{displaymath}
    \bdiv(\caR(\D))\ge \D. \qedhere
  \end{displaymath}
\end{proof}

\subsection{The volume of a b-divisor}

\begin{df} \label{def:volume_b_div} Let $\D$ be a Weil $\R$-b-divisor. The \emph{volume} of $\D$ is defined as
\begin{equation}\label{eq:vol-b}
\vol(\D)=  \limsup_{\ell} \frac{\dim \caR(\D)_{\ell}}{\ell^{d}/d!}.
\end{equation} 
A Weil $\R$-b-divisor $\D$ is called \emph{big} if  $\vol(\D)>0$. 
\end{df}

\begin{rmk}\label{rmk:limsup}
Since $\caR(\D)$ is a graded algebra of almost integral type, it follows from \cite[Corollary 3.11]{kk} that the $\limsup$ in \eqref{eq:vol-b} is in fact a $\lim$ for sufficiently divisible $\ell$.
\end{rmk}

In case the
Weil $\R$-b-divisor is Cartier, this definition agrees with the usual notion of
the volume of a divisor (see e.g.~\cite[Definition~2.2.31]{Laz}). 

\begin{lem} \label{lemm:20} Let $\D$ be a Weil $\R$-b-divisor. If there is a big Cartier
  divisor $B$ on some modification $X_{\pi }$ of $X$ such that $m \D\ge B$ for some
 $m>0$, then $\D$ is big.
\end{lem}
\begin{proof}
  Since $B$ is assumed to be big, we have $\vol(B)>0$. On the other hand,
  \begin{displaymath}
  \dim \caR(\D)_{m\ell} \ge \dim \caR(B)_{\ell}
\end{displaymath}
by the assumption that $m \D\ge B$.
Therefore,
\begin{displaymath}
  \vol(\D) \ge \limsup _{\ell} \frac{\dim
    \caR(\D)_{m\ell}}{m^d\ell^d/d!}\ge
  \frac{1}{m^d}\limsup_{\ell}\frac{\dim \caR(B)_{\ell}}{\ell ^d/d!}=
  \frac{1}{m^d}\vol(B)>0.
\end{displaymath}
\end{proof}

\subsection{The case of a toroidal nef and big b-divisor} 
\label{sec:case-toro-appr}

\begin{thm}\label{thm:1}
  Let $(X_1,\pi_1,D)$ be a toroidal structure on $X$.
  Let $\D$ be a nef and big $\R$-b-divisor which is  toroidal with respect to $D$. Then
  \begin{displaymath}
    \bdiv(\caR(\D))=\D.
  \end{displaymath}
\end{thm}
The proof requires an
intermediate lemma. We begin by constructing some auxiliary
objects. By
\cite[Lemma 5.9]{BoteroBurgos} we know that there is a
  sequence $\{\D_{i}\}_{i \in \N}$ of toroidal Cartier $\R$-b-divisors, generated by global sections and converging
  monotonically decreasing to $\D$. 
  Moreover by the proof of \cite[Lemma 5.14]{BoteroBurgos} (see also \cite[Remark 5.12]{BoteroBurgos}), we can pick a
  sequence of big toroidal Cartier $\R$-b-divisors $\{\B_{j}\}_{j \in \N}$ with $\B_{j} \le
  \D$ and $\vol(\B_{j})$ converging to $\vol(\D)$. This is where the
  toroidal condition is used.

  By Fujita's approximation theorem \cite[Theorem 11.4.4]{Laz}, for every
  $j>0$ there exists a Cartier $\R$-b-divisor $\A_j$ generated by
  global sections and satisfying
  \begin{displaymath}
    \vol(\A_{j})\ge \vol(\B_{j})-\frac{1}{j} \;\;\; \text{and} \;\;\;   \A_j \le \B_j. 
  \end{displaymath}
  Thus $\vol(\A_{j})$ also converges to $\vol (\D)$.

The key technical result is
\begin{lem}\label{eq:sup} We have
 \begin{equation}
    \sup_{j}(\A_{j})=\D \, , 
  \end{equation}
  where the supremum is computed componentwise. 
\end{lem}
  \begin{proof}
We need to show that, for every
  $\pi \in R(X)$ and every prime divisor $P$ in $X_{\pi}$, we have
  \begin{displaymath}
    \sup_{j}(\ord_{P}(\A_{j}))=\ord_{P}(\D).
  \end{displaymath}
We proceed by contradiction; suppose this does \emph{not} hold. This means that there exists a
  proper modification $\pi \in R(X)$, a prime divisor $P$ on $X_{\pi }$ and a
  positive number $\epsilon >0$ such 
  that for all $j$, 
  \begin{displaymath}
    \ord_{P}(\A_{j})\le \ord_{P}(\D)-\epsilon.
  \end{displaymath}
  
  In what follows we use the theory of Okounkov bodies, for which we refer to \cite{kk, LM}  for more details. 
  Upgrade $P$ to a complete flag
  \begin{displaymath}
    \caF\colon\quad P=Y_{1}\supset Y_{2}\supset \dots \supset Y_{d}
  \end{displaymath}
  and for a graded linear series $A$ denote by $O_{\caF}(A)$ the Okounkov
  body of $A$ on 
  $X$ associated to this
  flag. We briefly recall its construction. By an iterative procedure, one
  constructs for each $\ell \in \N$ a valuation map 
  
  \[
  \nu_\caF \colon A_\ell\setminus \{0\} \longrightarrow \Z^d
  \]
  by taking the order of vanishing along the given $Y_i$ into
  account. So, if $f\in A_{\ell}$,
  \begin{displaymath}
    \nu_{\caF}(f)=\left(\ord_{P}f,\ast,\dots,\ast\right).
  \end{displaymath}
  That is, the first component of $\nu _{\caF}(f)$ is the order
  of $f$ at $P$. 
  This gives rise to a semigroup 
  \[
  \Gamma(A) = \left\{\left(\nu_\caF(f), \ell\right) \mid f \in A_\ell\setminus
    \{0\}, \ell \in \N\right\} \subset \Z^{d}\times \N. 
  \]
  The Okounkov body of $A$ with respect to $\caF$ is then given by 
  \[
  O_{\caF}(A) = \overline{\text{cone}(\Gamma(A))} \cap \left(\R^d \times \{1\}\right).
  \]
  It is a closed convex set of $\R^d$, and if $A$ is of almost integral type,
  then it is bounded, hence compact \cite[Theorem~2.30]{kk}.  
  
   Let $\omega _{P}\colon
  \R^{d}\to \R$ be the projection onto the first variable. If $0\not=f \in
  A_{\ell}$ is of degree $\ell$ and $x = \nu_\caF(f)/\ell$ the corresponding
  point in the Okounkov body, then by construction one has
  \begin{displaymath}
    \omega _{P}(x)=\ord_{P}(f)/\ell.
  \end{displaymath}
  For a Weil $\R$-b-divisor $\mathbb{E}$ we write $O_\caF(\mathbb{E})$ for the
  Okounkov body $O_\caF(\caR(\mathbb{E}))$. 
  From $\A_{j}\le \D\le \D_{i}$ we have $\caR(\A_j) \subset \caR(\D)
  \subset \caR(\D_i)$ and hence 
  \begin{displaymath}
    O_{\caF}(\A_{j})\subset O_{\caF}(\D) \subset O_{\caF}(\D_{i})
  \end{displaymath}
  for all natural numbers $i,j$.
  Since each $\D_{i}$ is generated by global sections there exist
  $f_{i}\in H^{0}(X_{i},\ell_{i}\D_{i})$ for some $\ell_i$ such that
  \begin{displaymath}
    \ord_{P}(f_{i})/\ell_{i} = -\ord _{P}\D_{i}\le -\ord_{P}(\D).
  \end{displaymath}
  Therefore, there is a point $x_{i}\in O_{\caF}(\D_{i})$ with
  \begin{displaymath}
    \omega _{P}(x_{i})\le -\ord_{P}(\D).
  \end{displaymath}
  Since $O_{\caF}(\D_{1})$ is compact and $O_{\caF}(\D_{i})\subset
  O_{\caF}(\D_{1})$, the sequence $\{x_{i}\}_{i \in \N}$ has at least one
  accumulation point $x$. Moreover we claim that $\bigcap O_{\caF}(\D_{i})=O_{\caF}(\D)$. Indeed, we have $\bigcap O_{\caF}(\D_{i}) \subset O_{\caF}(\D)$. On the other hand $\{O_\caF(\D_i)\}_{i \in \N}$ form a decreasing (under inclusion) sequence of compact convex sets, hence their intersection $O = \bigcap_i O_\caF(\D_i)$ is again a compact convex set. We have that 
  \[
  \vol(O) = \lim_i\vol(O_\caF(\D_i)) = \lim_i \vol(\D_i) = \vol(\D) = \vol(O_\caF(\D)),
  \]
  where the second and last equalities follow from \cite[Theorem
  5.14]{BoteroBurgos}. This proves the claim since two
  full-dimensional compact convex sets with equal volume and such that
  one is contained in the other have to agree. Now, the compactness of
  $O_{\caF}(\D)$ implies that the accumulation point $x$ lies in
  $O_{\caF}(\D)$. In
  particular, there is a point $x\in O_{\caF}(\D)$ with $\omega
  _{P}(x)\le -\ord_{P}(\D)$.

  On the other hand, since $\ord_{P}(\A_{j})\le \ord_{P}(\D)- \epsilon$ we have
  that
  \begin{displaymath}
    \emptyset \not = O_{\caF}(\A_{j})\subset \{x\in  \R^{d}\mid \omega _{P}(x)\ge
    -\ord_{P}(\D)+\epsilon \}.
  \end{displaymath}
  The set $O_{\caF}(\D)$ is convex, has non-zero volume, contains a
  point $x$
  satisfying $\omega _{P}(x)\le -\ord_{P}(\D)$, and also contains a point $y$
  with  $\omega _{P}(y)\ge -\ord_{P}(\D)+\epsilon $ (just choose
  any point of $O_{\caF}(\A_{j})$ for some $j$). Hence
  \begin{displaymath}
    \vol(O_{\caF}(\D)\cap\{-\ord_{P}(\D)\le \omega _{P}\le
    -\ord_{P}(\D)+\epsilon \}) =:\eta >0.
  \end{displaymath}
  This implies that $\vol(O_{\caF}(\A_{j}))\le \vol (O_{\caF}(\D)) - \eta$, 
  contradicting the fact that $\vol(O_{\caF}(\A_{j}))$ converges to $\vol
  (O_{\caF}(\D))$. 
\end{proof}

\begin{proof}[Proof of Theorem {\ref{thm:1}}]
From Lemma \ref{lemm:3} and the inequalities $\A_{j}\le \D\le
  \D_{i}$ for any natural numbers $i,j$, we deduce
  \begin{displaymath}
    \bdiv(\caR(\A_{j}))\le
    \bdiv(\caR(\D))\le
    \bdiv(\caR(\D_{i})).
  \end{displaymath}
  By Lemma~\ref{lemm:3} and Proposition \ref{prop:1}  we get 
  \begin{displaymath}
    \A_{j}\le
    \bdiv(\caR(\D))\le
    \D_{i} \, . 
  \end{displaymath}
Invoking Lemma~\ref{eq:sup} and once more Lemma~\ref{lemm:3}, we get
\begin{displaymath}
  \D=\sup (\A_{j})\le \bdiv(\caR(\D))\le \D. \qedhere
\end{displaymath}
\end{proof}
\begin{rmk}\label{rem:exa-non-tor}
The toroidal condition in Theorem \ref{thm:1} is necessary. Indeed, consider the Weil $\R$-b-divisor $\D$ from \cite[Appendix A]{BBHJ}. Then $\D$ is a nef and big Weil $\R$-b-divisor on $\mathbb{P}^2$ which is not toroidal. It satisfies 
\[
\caR(\D) = \caR(2H-L),
\]
where $L$ and $H$ are two lines in $\mathbb{P}^2$ as defined in \emph{loc.\ cit.}
Hence, by the construction we get
\[
\D \neq 2H-L = \bdiv(\caR(2H-L))= \bdiv(\caR(\D)).
\]
\end{rmk}

From Lemma \ref{lemm:14} and Theorem \ref{thm:1} we obtain also the
following compatibility in the case of toroidal psh metrics. 
  \begin{cor}\label{cor:2}
    Let $(L,h)$ be a line bundle with a toroidal psh metric with
    singularity divisor $D$, and
    $s$ a non-zero rational section of $L$. If $\D(L,s,h)$ is nef and big and for every
    irreducible component $D_{i}$ of $D$ the condition
    $\ord_{D_{i}}(\D(L,s,h))>0$ holds, then  
    \begin{displaymath}
      \bdiv(\caR(L,s,h)) = \D(L,s,h).
    \end{displaymath}
  \end{cor}
  \begin{proof}
By Lemma \ref{lemm:4} and  Corollary \ref{cor:3}, for every $j>1$ 
    we have
    \begin{displaymath}
      \caR(((1-1/j)\D(L,s,h)) \subset
      \caR(L,s,h)
      \subset \caR(\D(L,s,h)).
    \end{displaymath}
    Since $h$ is a toroidal psh metric, by Proposition
    \ref{prop:div_of_toroidal_metric_is_toroidal} the Weil $\R$-b-divisor $\D(L,s,h)$ is
    toroidal with respect to a toroidal structure $D'$ above $D$. Note
    that we need a toroidal structure above $D$ in order to make $\dv(s)$
    toroidal. 
    By Lemma \ref{lemm:3} and Theorem \ref{thm:1}, we
    deduce
    \begin{displaymath}
      (1-1/j)\D(L,s,h)\le 
      \bdiv(\caR(L,s,h)) \le 
      \D(L,s,h).
    \end{displaymath}
    Since $(1-1/j)\D(L,s,h)$ converges to $\D(L,s,h)$ when $j\to
    \infty$, we obtain the corollary.  
  \end{proof}

\subsection{Criterion for not being finitely generated}
\label{sec:crit-non-finit}

\begin{lem}\label{lemm:2} If $A$ is a finitely generated graded linear
  series on $X$, then $\bdiv(A)$ is a Cartier $\R$-b-divisor.
\end{lem}
\begin{proof}
  Assume  that $A$ is generated by $f_{1}t^{d_1},\dots,f_{r}t^{d_r}$,
  with $f_{i} \in F=K(X)$.

  \smallskip
  \noindent
  \emph{Claim.} If a Weil $\R$-b-divisor $\D$ satisfies $d_{i} \D+\dv(f_{i})\ge 0$, for
  $i=1,\dots,r$, then $\ell \D+\dv(f)\ge 0$ for all $\ell \ge 0$ and all
  $f\in A_{\ell} $.

  The claim follows from the compatibility of multiplicity with
  products and the ultrametric inequality for sums. Namely, if $f\in
  A_{\ell}$ and $f'\in A_{\ell'}$, then the compatibility of the
  valuation with products yields,  
  \begin{displaymath}
    \left.
    \begin{aligned}
      \ell  \D +\dv(f)&\ge 0\\
      \ell'  \D +\dv(f')&\ge 0
    \end{aligned}
  \right\} \Longrightarrow
  (\ell+\ell')  \D +\dv(ff')\ge 0,
\end{displaymath}
and if $f,f'\in A_{\ell}$ the ultrametric inequality implies 
  \begin{displaymath}
    \left.
    \begin{aligned}
      \ell  \D +\dv(f)&\ge 0\\
      \ell  \D +\dv(f')&\ge 0
    \end{aligned}
  \right\} \Longrightarrow
  \ell  \D +\dv(f+f')\ge 0.
\end{displaymath}

From the claim we deduce
\begin{displaymath}
  \bdiv(A)_{\pi }=\inf \{D\in \Div(X_{\pi })\mid d_{i}D+\dv(f_{i})\ge
  0,\ i=1,\dots,r\}. 
\end{displaymath}

Choose $D_{0}$ such that $d_{i}D_{0}+\dv(f_{i})\ge 0$. Let
$e=\operatorname{lcm}(d_{1},\dots,d_{r})$ 
and let $\mathfrak{a}$ be the fractional ideal generated by
$f_{i}^{e/d_{i}}$, $i=1,\dots,r$. Since $f_{i}^{e/d_{i}}\in \caO(eD_{0})$, we obtain
$\mathfrak{a}\subset \caO(eD_{0})$, so
\begin{displaymath}
  \caI\coloneqq\mathfrak{a}\caO_X(-eD_{0})\subset \caO_{X}
\end{displaymath}
is a coherent ideal sheaf. 

Let $\pi _{0}\colon X_{\pi _{0}}\to X$ be a proper modification such
that $\pi ^{-1}(\caI)\caO_{X_{\pi _{0}}}$ is principal and equal to
$\caO(-D')$ for some effective integral divisor $D'$. Then
\begin{displaymath}
  e(\bdiv(A))=eD_{0}-D',
\end{displaymath}
and in particular $\bdiv(A)$ is Cartier.
\end{proof}

\section{Siegel--Jacobi forms}
\label{sec:jacobi-forms}

\subsection{Basic definitions}
\label{sec:basic-definitions-1}

We recall the definition of Siegel--Jacobi forms. For more details
we refer to \cite{Kramer_Crelle} and \cite{Ziegler-J}.

Let $g \geq 1$ be an integer. The real symplectic group $\Sp(2g,\R)$ is the group of real $2g\times 2g$
matrices of the form
\begin{equation}\label{eq:1}
  \begin{pmatrix}
    A & B \\
    C & D
  \end{pmatrix}
\end{equation}
such that
\begin{displaymath}
  A^{t}C=C^{t}A,\quad D^{t}B=B^{t}D, \quad A^{t}D=\Id_{g}+C^{t}B,
\end{displaymath}
where $\Id_{g}$ is the identity matrix of dimension $g$. There is an
inclusion 
\[
\Sp(2g,\R)\hookrightarrow  \Sp(2g+2,\R)
\]
sending a matrix
of the form \eqref{eq:1} to
\begin{displaymath}
  \begin{pmatrix}
    A &0& B&0 \\
    0&1&0&0\\
    C &0& D&0\\
    0&0&0&1
  \end{pmatrix}.
\end{displaymath}
For any commutative ring $R$ let $H_{R}^{(g,1)}$ be the Heisenberg group
\begin{displaymath}
  H_{R}^{(g,1)}=\left\{\left[(\lambda ,\mu ),x\right]\mid \lambda ,\mu \in
R^{(1,g)},\, x\in R\right\},
\end{displaymath}
where $R^{(1,g)}$ denotes the set of row vectors of size $g$ and coefficients in $R$, with the composition law given by
\begin{displaymath}
  [(\lambda ,\mu ),x]\circ [(\lambda ',\mu '),x']=
  [\lambda +\lambda ',\mu +\mu ',x+x'+\lambda {\mu '}^{t}-\mu
  {\lambda '}^{t}]. 
\end{displaymath}
These are the same definitions as in \cite{Ziegler-J} for the case $g=1$. The real Heisenberg group $H_{\R}^{(g,1)}$ can be realized as the
subgroup of $\Sp(2g+2,\R)$
consisting of matrices of the form
\begin{displaymath}
  \begin{pmatrix}
    \Id_{g} & 0 & 0 & \mu^{t}\\
    \lambda  & 1 & \mu  & x\\
    0 & 0 & \Id_{g} & -\lambda^{t} \\
    0 & 0 & 0 & 1
  \end{pmatrix}.
\end{displaymath}
 The full Jacobi group $G^{(g,1)}_{\R}=\Sp(2g,\R) \ltimes
 H^{(g,1)}_{\R}$ is the subgroup of $\Sp(2g+2,\R)$ generated by 
 $\Sp(2g,\R)$ and $H^{(g,1)}_{\R}$.

 Let
 \begin{displaymath}
   \caH_{g}=\{ Z=X+iY \mid X,Y\in \mathrm{Mat}_{g\times g}(\R),\, Z^{t}=Z,\, Y>0\} 
 \end{displaymath}
 be the Siegel upper half space. The group $\Sp(2g,\R)$ acts
 transitively on $\caH_g$, where for $M = \begin{pmatrix}
A&B \\ C&D
\end{pmatrix} \in \Sp(2g, \R)$ and $Z \in \caH_g$ the action is given by 
\[
  Z\longmapsto 
M\langle Z \rangle = (AZ + B)(CZ + D)^{-1} \, . 
\]
On the other hand, the group $G^{(g,1)}_{\R}$ acts transitively on $\caH_{g}\times
\C^{(1,g)}$ by the action
\[
\left(M, (\lambda, \mu, x)\right) \cdot (Z,W) = \left(M\langle Z
  \rangle, (W+\lambda Z + \mu)(CZ + D)^{-1}\right).
\]
Let $\Gamma \subset \Sp(2g,\Z)$ be a subgroup of finite
index.  We write $\widetilde \Gamma =\Gamma \ltimes
H^{(g,1)}_{\Z}\subset G^{(g,1)}_{\Z}$.

Recall that a subgroup $\Gamma  \subset \Sp(2g,\Z)$ is called \emph{neat} when for
every $M\in \Gamma $, the 
subgroup of  $\C^{\times}$ generated by the eigenvalues of $M$ is
torsion free. 
If $\Gamma $ is neat, then the quotient $\caA(\Gamma )=\Gamma
\backslash \caH_{g}$ is a smooth
complex manifold and the quotient $\caB(\Gamma
)=\widetilde \Gamma \backslash \caH_{g}\times \C^{(1,g)}$ is a
fibration over $\caA(\Gamma )$ by principally polarized abelian
varieties. 

Following \cite{Ziegler-J} we now introduce automorphy factors for the
group  $G^{(g,1)}_{\R}$ in order to obtain interesting line bundles on
the quotient $\caB(\Gamma )$.

Let $\phi \colon \caH_g \times \C^{(1,g)} \to \C$ be a holomorphic map.
Let $\rho_{k} \colon \GL(g, \C) \to \C^\times$ be the representation given by $N
\mapsto (\det N)^k$. For $M \in \Sp(2g, \R)$, $\zeta =
(\lambda, \mu, x) \in H_{\R}^{(g,1)}$ and $m \in \N_{\geq 0}$ define
\begin{multline}\label{eq:25}
  \left(\phi|_{k,m}M\right)(Z,W) \coloneqq \\
  \rho_{k}(CZ + D)^{-1} e^{-2\pi i
    mW(CZ+D)^{-1}CW^{t}} \phi\left(M\langle
    Z\rangle, W(CZ+ D)^{-1}\right)
\end{multline}
and 
  \begin{equation}
\left(\phi|_{k,m}\zeta\right)(Z,W) \coloneqq e^{2\pi i
  m\left(\lambda Z \lambda^{t}+2\lambda W^{t}
      + (x + \mu \lambda^{t}) \right)} \phi(Z,
W+\lambda Z+ \mu). \label{eq:26}
\end{equation}

A matrix $T$ is called \emph{half integral} if $2T$ has integral entries and
the diagonal entries of $T$ are integral. Note that if $T$ is
symmetric then $T$ is half integral if and only if the associated quadratic form is
integral.

\begin{df} \label{def:SJ_forms}
  A holomorphic map $\phi \colon \caH_g \times \C^{(1,g)} \to \C$ is
  called a \emph{Siegel--Jacobi form of weight $k$ and index $m$} for a
  subgroup $\Gamma \subset \Sp(2g, \Z)$ of finite index if the following
  conditions are satisfied:
  \begin{enumerate}
  \item \label{item:17} $\phi|_{k,m}M = \phi$ for all $M \in \Gamma$.
  \item \label{item:18} $\phi|_{k,m}\zeta = \phi$ for all $\zeta \in H_{\Z}^{(g,1)}$.
  \item \label{item:16} For each $M \in \Sp(2g, \Z)$ the function $\phi|_{k,m}M$ has a
    Fourier expansion of the form
      \begin{equation}
      \left(\phi|_{k,m}M\right)(Z,W) = \sum_{\stackrel{T =
          T^{t}\geq 0}{T \text{ half integral}}} \sum_{R \in
        \Z^{(g,1)}} c(T,R) e^{2 \pi i/\nu \tr(TZ)}e^{2 \pi i
        WR}\label{eq:29}
    \end{equation}
    for some  suitable integer $0 < \nu \in \Z$ depending only on
    $\Gamma$, and such that 
    $c(T,R) \neq 0$ implies 
      \begin{equation}
        \label{eq:35}
        \begin{pmatrix}
          \frac{1}{\nu}T & \frac{1}{2}R \\
          \frac{1}{2}R^{t} & m
        \end{pmatrix} \geq 0.
      \end{equation}
  \end{enumerate}
  A Siegel--Jacobi form $\phi$ is said to be a \emph{cusp form} if
  \[
    \begin{pmatrix}
      \frac{1}{\nu }T & \frac{1}{2}R \\
      \frac{1}{2}R^{t} & m
    \end{pmatrix} > 0.
  \]
  for any $T,R$ with $c(T,R) \neq 0$.
\end{df}

  We note that when $g\geq 2$, condition \ref{item:16} is a
  consequence of conditions \ref{item:17} and \ref{item:18} due to the Koecher
  principle \cite[Lemma 1.6]{Ziegler-J}.

  \begin{df}
The vector space of all Siegel--Jacobi forms of weight $k$ and index
$m$ for  $\Gamma$ is denoted by $J_{k,m}(\Gamma )$ and the
space of cusp forms is denoted by $J_{k,m}^{\cusp}(\Gamma )$.
\end{df}

We next give an interpretation of condition \ref{item:16} in the
definition of Siegel--Jacobi forms. This interpretation is mentioned in
\cite{dulinski}, but since we have not found a proof we include one
for the convenience of the reader.

Let $\phi \colon \caH_g \times \C^{(1,g)} \to \C$ be a holomorphic map satisfying
that there exists an integer $\nu >0$ such that for any integral symmetric
$g\times g$ matrix
$X$ and integer vector $\mu $, the periodicity condition
\begin{displaymath}
  \phi (Z+\nu X,W+\mu  )=\phi (Z,W)
\end{displaymath}
is satisfied. This is the case for any function satisfying conditions
\ref{item:17} and \ref{item:18} of Definition \ref{def:SJ_forms}.
Such a function admits a Fourier expansion of the form \eqref{eq:29}
(with $M=\Id$).
Since the function $\phi $ is continuous, this Fourier expansion is
absolutely convergent and uniformly convergent on compact subsets. 

We introduce two kind of subsets of $\caH_{g}\times \C^{(1,g)}$ that will be useful for stating the
next lemma. 
For any $Z_{0}\in \caH_{g}$ and compact subset $K\subset \R^{(1,g)}$, we
will denote by $U_{Z_{0},K}$ the domain 
\begin{displaymath}
  U_{Z_{0},K}=\left\{(Z,W)\in \caH_{g}\times \C^{(1,g)}\mid \im(Z)\ge
  \im(Z_{0}),\ \im(W)\in K \im(Z)\right\}.
\end{displaymath}
(Here, $\im(Z)\ge  \im(Z_{0})$ means that $\im(Z - Z_0)$  is positive semi-definite.) On the other hand, if $K\subset \caH^{g}\times \C^{(1,g)}$ is a compact subset, $Y$ is
a symmetric semipositive real matrix and $\zeta$ is a real $g$-dimensional row vector, we define the strip
\begin{displaymath}
  S_{Y,\zeta,K}=K+\R_{\ge 0}i(Y,\zeta Y) \subseteq \caH_{g}\times \C^{(1,g)}.
\end{displaymath}

\begin{prop}\label{prop:9} Let $\phi $ be a periodic holomorphic function as
  before. If the Fourier expansion \eqref{eq:29} satisfies that
  $c(T,R)=0$ unless condition \eqref{eq:35} holds, then 
  the function
  \begin{displaymath}
    \psi(Z,W)=\phi (Z,W)e^{-2\pi m\im(W) (\im Z)^{-1}\im(W)^{t}}
  \end{displaymath}
  is bounded in each domain of the form $U_{Z_{0},K}$. Conversely, if
  the function $\psi(Z,W)$ is bounded in each strip of the form
  $S_{Y,\zeta,K}$, then the Fourier
  expansion of $\phi $
  satisfies that $c(T,R)\not =0$ implies condition \eqref{eq:35}.  
\end{prop}
\begin{proof}
  Assume that the Fourier expansion satisfies that $c(T,R)=0$ unless
  condition \eqref{eq:35} holds. Let $(Z,W)\in U_{Z_{0},K}$. Write
  $Y=\im(Z)$ and $\beta =\im(W)$. Consider the product
  \begin{displaymath}
    A_{T,R}=|e^{2 \pi i \tr(TZ)/\nu }e^{2 \pi i
      WR} e^{-2\pi m\beta Y^{-1}\beta ^{t}}|
    =e^{-2\pi (\tr(TY)/\nu +\beta R+m\beta Y^{-1}\beta ^{t})}.
  \end{displaymath}
  Since $(Z,W)\in U_{Z_{0},K}$ we can write $Y=Y_{0}+Y_{1}$, with
  $Y_{0}=\im(Z_{0}) $ and $Y_{1}\ge 0$, and $\beta =\zeta Y$ with
  $\zeta \in K$. Then
  \begin{displaymath}
    A_{T,R}=e^{-2\pi (\tr(TY_{0})/\nu+\zeta Y_{0}R+m\zeta Y_{0}\zeta ^{t})}
    e^{-2\pi (\tr(TY_{1})/\nu+\zeta Y_{1}R+m\zeta Y_{1}\zeta ^{t})}
  \end{displaymath}
  We define the matrices
  \begin{displaymath}
    M_{1}=
    \begin{pmatrix}
      Y_{1} & Y_{1} \zeta ^{t}\\
      \zeta  Y_{1} & \zeta Y_{1} \zeta ^{t}
    \end{pmatrix},\qquad
    M_{2}=
    \begin{pmatrix}
      T/\nu  & R/2\\
      R/2 & m
    \end{pmatrix}.
  \end{displaymath}
   Since $M_{1}$ and $M_{2}$ are symmetric and positive semidefinite,
   $\tr(M_{2}M_{1})\ge 0$. Moreover, 
   \begin{displaymath}
     e^{-2\pi (\tr(TY_{1})/\nu+\zeta Y_{1}R+m\zeta Y_{1}\zeta ^{t})}=
     e^{-2\pi \tr(M_{2}M_{1})}\le 1.
   \end{displaymath}
   We deduce that, for $(Z,W)\in U_{Z_{0},K}$,
   \begin{displaymath}
     |\psi(Z,W)|\le
     \sup_{\zeta \in K}\left( e^{-2\pi m\zeta Y_{0}\zeta
       ^{t}}\sum_{T,R}|c(T,R)| e^{-2\pi \tr(TY_{0})/\nu}
     e^{-2\pi \zeta Y_{0}R }\right). 
   \end{displaymath}
   Since the Fourier expansion is absolutely  convergent
   uniformly on compacts, the right hand side of the previous inequality is
   finite. Thus $\psi$ is bounded in $U_{Z_{0},K}$.

   To prove the converse we show that, if in the Fourier expansion
   \eqref{eq:29} there is a coefficient $c(T,R)\not =0$ that does not
   satisfy the condition \eqref{eq:35} then we can find a strip as in
   the proposition where the function $\psi$ is not bounded. Indeed,
   since we are assuming that condition \eqref{eq:35} is not
   satisfied, there exists a real vector $\zeta _{0}$ and a number $q$
   such that
   \begin{equation}\label{eq:34}
     \frac{1}{\nu }\zeta_{0} T\zeta_{0} ^{t} +q\zeta _{0}R+ q^{2}m <0. 
     \end{equation}
   Since $m\ge 0$ and $T\ge 0$ necessarily $\zeta _{0}\not =0$ and
   $q\not = 0$. We can assume furthermore that $\zeta _{0}\zeta
   _{0}^{t}=1$. Consider the subset $K\subset \caH_{g}\times \C^{(1,g)}$
   consisting of the elements $\big((x_{j,k})_{1\le j,k\le
     g}+i\Id,(y_{\ell})_{1\le \ell\le g}\big)$ with
   \begin{displaymath}
     x_{j,k}=x_{k,j}\in \R,\quad 0\le x_{j,k}\le \nu ,\quad  y_{\ell}\in
     \R,\quad 0\le y_{\ell}\le 1.
   \end{displaymath}
   Put $Y=\zeta_{0}^{t}\zeta_{0}$ and $\zeta=q\zeta_{0}$. The formula
   for the coefficients of the Fourier expansion tells us that, for any
   real number $t>0$, 
   \begin{displaymath}
     c(T,R)=
     \frac{1}{\nu ^{g(g-1)/2}}
     \int_{K}\psi e^{-2\pi i(\tr(T\cdot(x_{j,k}))/\nu +(y_{\ell})\cdot R))}
     e^{2\pi t(\tr(TY)/\nu +\zeta R)}dx_{j,k}dy_{\ell}.
   \end{displaymath}
   In particular this integral is independent of $t$.
Since
\begin{displaymath}
  \frac{1}{\nu }\tr(TY) +\zeta R=\frac{1}{\nu
  }\zeta_{0}T\zeta_{0}^{t}+q\zeta_{0}R<0, 
\end{displaymath}
   if the function $\psi$ is bounded in the strip $S_{Y,\zeta,K}$,
   then, taking the limit $t\to \infty$, we deduce that 
   $c(T,R)=0$ concluding the proof of the proposition.
 \end{proof}

  The following lemma follows easily from the definitions.

  \begin{lem} \label{lem:multiplicativity}
    If $\phi $ is a Siegel--Jacobi form of weight $k$ and index $m$ and $\psi
    $ is a Siegel--Jacobi form of  weight $k'$ and index $m'$,
    then $\phi \psi $ is a Siegel--Jacobi form of  weight $k+k'$ and
    index $m+m'$. If one of them is a cusp form, the same is true for
    the product.
  \end{lem}

  The following result is \cite[Theorem 1.5]{Ziegler-J}
\begin{lem}\label{lemm:21}
        Let $\phi $ be a Siegel--Jacobi form of weight $k$ and index
        $m$  for a subgroup $\Gamma \subset \Sp(2g, \Z)$ of finite index
       and let $\lambda,\mu  \in \Q^{(1,g)}$ be rational
        vectors. Then there is a finite index subgroup $\Gamma '\subset
        \Sp(2g, \Z)$ that depends only on $\Gamma$, $\lambda $ and $\mu
        $ such that the function
        \begin{displaymath}
          f(Z)=e^{2\pi i m\lambda  Z\lambda ^{t}}\phi (Z,\lambda Z+\mu )
        \end{displaymath}
        is a Siegel modular form of weight $k$ for $\Gamma '$.
  \end{lem}
  
  \begin{ex}\label{exm:1} Here we list some examples of Siegel--Jacobi
    forms which turn out to be useful for our purposes.
    \begin{enumerate}
    \item \label{item:8} The only Siegel--Jacobi forms of
      weight $0$ and index
      $0$ are the constants (i.e. $J_{0,0}(\Gamma )=\C$). Moreover
      $J_{k,m}(\Gamma )=0$ whenever $k <0$ or $m<0$. The first case follows from
      Lemma \ref{lemm:21} using the 
      Corollary to
      Proposition 1 in \cite[Section 4]{klingen} or the remark after
      Theorem 1 in \cite[Section 8]{klingen}, while the second follows
      from the geometric description  \eqref{eq:18} of the line bundle
      of Siegel--Jacobi forms; indeed, if $m<0$ the restriction to any fibre is strictly anti-effective. 
    \item \label{item:9} Any Siegel modular form of weight $k$  for
      $\Gamma $ defines a Siegel--Jacobi form of weight $k$ and index
      $0$. A classical example is as follows (for details we refer to
      \cite[Section 2]{Ziegler-J}). For $k>g+1$ even, the
      \emph{Eisenstein series}
      \begin{equation}\label{eq:23}
        E_{g,k}(Z)=\sum_{M \in P_{g}\backslash
          \Sp(2g,\Z)}\rho _{k} (C Z +D)^{-1}
      \end{equation}
 is convergent and defines a Siegel modular form of weight
      $k$ for the group $\Sp(2g,\Z)$, in particular for any subgroup
      $\Gamma \subset \Sp(2g,\Z)$. Hence it defines a Siegel--Jacobi form of weight $k$ and index $0$ for $\Gamma$.
      In equation \eqref{eq:23}, $P_{g}$ is the subgroup
      \begin{displaymath}
        P_{g}=\left\{
          \begin{pmatrix}
            A & B\\
            C & D
          \end{pmatrix}
          \in \Sp(2g,\Z)\,\middle | \,C=0\right\}
        \end{displaymath}
        and $M$ is written as in \eqref{eq:1}.
      \item \label{item:13} Similarly, \emph{Poincar\'e series} can be used to produce non-zero
        Siegel cusp forms of any weight $k>2g$ such
          that $kg$ is even (see \cite[Proposition 2 and its
        corollary]{klingen}). By pullback they produce
        Siegel--Jacobi cusp forms of index zero. 
      \item \label{item:10} Eisenstein series can be generalized to
        produce Siegel--Jacobi modular forms of arbitrary index; given
        $m>0$ an integer and $k>g+2$ an even integer, Ziegler constructs in \cite[Theorem 2.1]{Ziegler-J} a Siegel--Jacobi modular form of weight $k$ and index $m$. 
%
    \end{enumerate}
  \end{ex}

  \begin{lem}\label{lemm:22}
    If $\phi $ is a Siegel--Jacobi form of weight $0$ and index $m \neq 0$,
    then $\phi =0$.
  \end{lem}
  \begin{proof}
    Let $\phi $ be such a Siegel--Jacobi form. For $Z\in \caH_{g}$ and
    $\lambda, \mu \in \R^{(1,g)}$, write
    \begin{displaymath}
          f(Z,\lambda ,\mu )=e^{2\pi i m\lambda  Z\lambda ^{t}}\phi (Z,\lambda Z+\mu ).
        \end{displaymath}
        By Lemma \ref{lemm:21}, whenever $\lambda, \mu \in \Q^{(1,g)}$
        are fixed we obtain a Siegel modular form of weight zero for a
        certain finite index subgroup  of $\Sp(2g,\Z) $, which is necessarily constant. By
    continuity, for every $\lambda, \mu \in \R^{(1,g)}$ the
        function $f(Z,\lambda ,\mu )$ is constant. Therefore, there is
        a function $g\colon \R^{(1,g)}\times \R^{(1,g)} \to \C$ such
        that
        \begin{equation}\label{eq:24}
          \phi (Z,\lambda Z+\mu )=e^{-2\pi i m\lambda  Z\lambda ^{t}}g(\lambda ,\mu ).
        \end{equation}
        We now see that this is incompatible with the modular condition
        $\phi|_{k,m}M = \phi$ for all $M \in \Gamma$. Let
        \begin{displaymath}
          M=
          \begin{pmatrix}
            A & B\\ C&D
          \end{pmatrix}\in \Gamma, 
        \end{displaymath}
        then the modular condition and equation \eqref{eq:24} imply that
        \begin{displaymath}
          e^{-2\pi i m \left( (\lambda Z+\mu)(CZ+D)^{-1}C(\lambda
              Z+\mu )^{t}+\lambda 'M\langle Z\rangle \lambda
              '^{t}-\lambda Z\lambda ^{t} \right)}
          g(\lambda ',\mu ')=g(\lambda ,\mu ),
        \end{displaymath}
        where
        \begin{displaymath}
          (\lambda ',\mu ')=(\lambda ,\mu )M^{-1}.
        \end{displaymath}
        When $m\not =0$, $M\not = \Id$ and $\lambda \not = 0$, the
        exponential term actually depends on $Z$. As the function $g$ does not depend on $Z$, we conclude that $g=0$. 
  \end{proof}

  A standard consequence of the previous examples and Lemma
  \ref{lemm:22} is the following (see also \cite{EZ}). 

  \begin{prop}\label{prop:bi-gr-alg}
    The ring $J_{\ast,\ast}(\Gamma )\coloneqq \bigoplus _{k,m}J_{k,m}(\Gamma )$
    is not finitely generated. 
  \end{prop}
  \begin{proof}
    Assume that the ring $J_{\ast,\ast}(\Gamma )$ is finitely
    generated. Let $1,f_{i}$, $i=1,\dots,r$ be a set of homogeneous generators
    with $f_{i}$ non-constant. Let $k_{i}$ and $m_{i}$ be the weight and
    index of $f_{i}$. Then by Example \ref{exm:1}.\ref{item:8} and Lemma~\ref{lemm:22} we have
    $k_{i}>0$ for $i=1,\dots,r$. Thus the possible ratios
    $m_{i}/k_{i}$ are bounded above. Therefore if $f $ is a non-constant Siegel--Jacobi modular form of weight $k$ and index $m$, then
    the ratio $m/k$ is bounded. But in Example
    \ref{exm:1}.\ref{item:10} we have a construction of non-zero Siegel--Jacobi
    modular forms of fixed weight and arbitrary index.         
  \end{proof}
  
  \begin{rmk}
Fix $r\in \mathbb Q_{>0}$. Theorem \ref{th:intro-not-fin-gen} states that the subring $$\bigoplus_{m\le rk}J_{k, m}(\Gamma)\subseteq J_{\ast,\ast}(\Gamma )$$ is not finitely generated. It does not seem possible to give a simple proof of this fact by the methods of the proof of Proposition \ref{prop:bi-gr-alg}; indeed, that argument relies on there being homogeneous elements in $J_{\ast,\ast}(\Gamma )$ for which the slope $m/k$ gets arbitrarily large, and homogeneous elements in $\bigoplus_{m\le rk}J_{k, m}(\Gamma)$ have slope bounded by $r$. 
  \end{rmk}
  
\subsection{The line bundle of Siegel--Jacobi forms}\label{sec:geom-interpr-line}

As mentioned in the introduction, to the neat arithmetic group $\Gamma $ we
associate a fibration of principally polarized abelian varieties over a
complex manifold
\begin{displaymath}
 \pi \colon  \caB(\Gamma )\longrightarrow \caA(\Gamma ).
\end{displaymath}
The transformations \eqref{eq:25} and \eqref{eq:26} define a cocycle
for the group $\widetilde \Gamma $. Therefore they determine a line
bundle $L_{k,m}$ such that the Siegel--Jacobi forms of weight $k$ and
index $m$ can be seen as global sections of this line bundle.  We
recall the geometric interpretation of the line bundle $L_{k,m}$.  

Let $e\colon \caA(\Gamma )\to \caB(\Gamma )$ be the zero section. Let
$M$ be the line bundle on $\caA(\Gamma )$ defined as 
\begin{displaymath}
  M=\det\left(e^{\ast} \Omega ^{1}_{\caB(\Gamma )/\caA(\Gamma )}\right).
\end{displaymath}
Let $(u_{1},\dots,u_{g})$ be the canonical 
holomorphic coordinates  on $\C^{g}$. Then the form 
$du_{1}\wedge \dots \wedge du_{g}$ is a multivalued section of $M$,
well defined up to a global constant. Note that this is a multivalued
section because $du_{1}\wedge \dots \wedge du_{g}$ is not invariant
under the action of $\Gamma $. Let $f$ be a Siegel modular form of
weight $k$. Then one can verify that the symbol
\begin{displaymath}
  f(du_{1}\wedge \dots \wedge du_{g})^{\otimes k}
\end{displaymath}
is invariant under the action of $\Gamma $ and therefore determines a
section of the line bundle $M^{\otimes k}$. In this way we obtain an
identification of $L_{k,0}$ with $\pi ^{\ast}M^{\otimes k}$. 

Since $\caB(\Gamma )$ is a family of principally polarized abelian
varieties it comes equipped with a biextension line bundle $B$. This
line bundle is defined as follows.
Let $\caB(\Gamma )^{\vee}$ be the dual family of abelian varieties and
$P$ the Poincar\'e line bundle on $\caB(\Gamma )\times_{\caA(\Gamma)}
\caB(\Gamma)^{\vee}$. Let $\lambda \colon \caB(\Gamma )\to \caB(\Gamma
)^{\vee}$ be the isomorphism defined by the polarization. Then
\begin{displaymath}
  B=(\Id,\lambda )^{\ast}P.
\end{displaymath}
The line bundle $B$ gives on each fibre twice the principal
polarization. There is also an identification $L_{0,m}$ with
$B^{\otimes m}$. Therefore, the line bundle of Siegel--Jacobi forms of weight $k$ and index $m$ is given by
\begin{equation}\label{eq:18}
  L_{k,m}=\pi ^{\ast}M^{\otimes k}\otimes B^{\otimes m}.
\end{equation}

\begin{lem} We have
\[ J_{k,m}(\Gamma) \subseteq H^0(\caB(\Gamma),L_{k,m}) \, , \]
with equality if $g \geq 2$. 

\end{lem}
\begin{proof} The first statement follows directly from Definition
  \ref{def:SJ_forms} and the second statement is the Koecher
  principle  \cite[Lemma 1.6]{Ziegler-J}.  
\end{proof}

\subsection{The invariant metric}

The aim of this subsection is to discuss the canonical invariant metric on the line bundle of Siegel--Jacobi forms.

For $Z\in \caH_{g}$ we write $Z=X+iY$ with $X, Y \in \mathrm{Mat}_{g \times g}(\R)$ and for $W\in \C^{(1,g)}$ we write
$W=\alpha +i\beta $ with $\alpha, \beta \in \R^{(1,g)}$. 
\begin{df}\label{def:6}
  Let $\phi \in J_{k,m}(\Gamma )$ be a Siegel--Jacobi form. Then the
  \emph{standard invariant norm} $h^{\inv}(\phi)$ of $\phi $ is defined by
  \begin{displaymath}
    h^{\inv}(\phi (Z,W))^{2}=|\phi (Z,W)|^{2}\rho _{k}(Y)e^{-4\pi m \beta
      Y^{-1}\beta^{t}}.
  \end{displaymath}
  This quantity is readily checked to be $\widetilde \Gamma $-invariant. When $\Gamma
  $ is neat it induces a smooth hermitian metric on $L_{k,m}(\Gamma )$.
\end{df}

The standard invariant metric of Definition \ref{def:6} gives in particular
metrics on $M$ and $B$. These metrics agree with the classical
Hodge metric on $M$ and the canonical biextension metric on $B$. We refer to \cite[Section~5]{MR4221000}
for a brief discussion of these classical metrics and a verification of this agreement.

\begin{lem}\label{lem:psh} Assume $k, m \ge 0$.
  The standard invariant metric $h^{\inv}$ on $L_{k,m}(\Gamma )$ is a psh (i.e.\ semipositive) metric. 
\end{lem}
\begin{proof}
  It is enough to show that $-\log \det (Y)$ and $\beta
      Y^{-1}\beta^{t}$ are psh functions on $\caH_{g}\times
      \C^{(1,g)}$. A smooth function $f\colon \C^{n}\to \R$ of the form
      \begin{displaymath}
        f(z_{1},\dots, z_{n})=\varphi(\im(z_{1}),\dots,\im(z_{n})) 
      \end{displaymath}
      is psh if and only if $\varphi$ is convex
      (see \cite[I~(5.13)]{demailly12:cadg}). Then the lemma follows from the
    Example ``log-determinant'' in \cite[Section~3.1.5]{Boyd:convex} and
    by \cite[Example~3.4,
    Section~3.1.7]{Boyd:convex}.
\end{proof}

A useful result by Ziegler is the following  (see \cite[Proposition~1 and its Corollary]{dulinski}).

\begin{prop}\label{prop:6}
  A Siegel--Jacobi form $\phi $ is a cusp form if and only if
  $h^{\inv}(\phi)$ is
  bounded. 
\end{prop}

\section{Toroidal compactifications of the universal abelian variety}
\label{sec:toro-comp}

\subsection{Toroidal compactifications}
\label{subsec:toro-comp}

We give a brief account of the theory of toroidal compactifications of
$\caA(\Gamma )$ and of $\caB(\Gamma )$. For more details we refer to the book
\cite{fc} by Faltings and Chai, and Namikawa's work \cite{Nam1977, Nam1979}.  Note that in \cite{fc}
everything is worked out for the full modular group $\Sp(2g,\Z)$ and the
principal congruence subgroups $\Gamma (N)$. This is because the authors are mainly interested in
integral models. If one is only interested in the theory over the
complex numbers everything carries over to the general case of a
commensurable subgroup $\Gamma $. On the other hand, to avoid having
to deal with algebraic stacks we will be mainly interested in the case when
$\Gamma $ is neat.   

Let $\Lambda =\Z^{2g}$ be a lattice of rank $2g$ endowed
with the standard symplectic form. The group $\Sp(2g,\Z)$ acts on
$\Lambda $ by isometries and hence also the group $\Gamma $. A
\emph{rational boundary component} is a subspace $\eta \subset
\Lambda _{\Q}$ such that $\eta ^{\perp}\subset \eta $. With a rational
boundary component we associate the torsion free abelian group
\begin{equation*}
  X_{\eta} = \Lambda /(\Lambda \cap \eta).
\end{equation*}
Given a lattice $X$ of rank $r$ we denote by $C(X)$ the space of
symmetric semipositive bilinear forms in $X_{\R}$ with rational
radical and by $\widetilde C(X)$ the set of pairs $(\Omega ,\beta )$,
where $\Omega \in C(X)$ and $\beta $ is a linear form on $X_{\R}$ such
that $\Omega +2\beta $ is bounded below. After choosing an isomorphism
$X\simeq \Z^{r}$ we can identify  $C(X)$ with be the cone of symmetric
semipositive real matrices with 
rational kernel. Moreover, thanks to Lemma \ref{lemm:sym-cone}, we can identify $\widetilde C(X)$ with the cone
\begin{displaymath}
  \widetilde C(X)=\{(\Omega ,\beta )\in C(X)\times
\R^{(1,g)} \mid \exists \zeta \in \R^{(1,g)}, \, \beta =\zeta \Omega \}.
\end{displaymath}
\begin{lem}\label{lemm:sym-cone}
Let $\Omega $ be a real symmetric positive semidefinite matrix with
rational kernel and $\beta $ a row vector. Then the following
statements are equivalent.
\begin{enumerate}
\item \label{item:11} There is a $\zeta  \in \R^{(1,g)}$ such that $\beta =\zeta 
  \Omega $. 
\item \label{item:12} The set $x^{t} \Omega x+2\beta x$, $x\in
  \R^{(g,1)}$ is bounded below. 
\end{enumerate}
\end{lem}
\begin{proof}
  Since $\Omega $ is positive semidefinite, the
  function $x\mapsto x^{t} \Omega x+2\beta x$ is convex. Therefore it
  is bounded below if and only if it has a stationary point. Taking
  derivatives and dividing by 2, we deduce that the function is
  bounded below if and only 
  if the equation $x^{t}\Omega +\beta =0$ has a solution. 
\end{proof}

We will denote by $C(X)_{\Z}$ the subset of half-integral
matrices of $C(X)$, that is, the set of bilinear forms whose
associated quadratic form has integral coefficients. Furthermore we
set $\widetilde C(X)_{\Z}=\widetilde C(X)\cap (C(X)_{\Z}\times
X^{\vee})$.

\begin{rmk}\label{rem:6}
  If $\eta\subset \eta'$, then $X_{\eta'}$ is a quotient of
  $X_{\eta}$ and we can identify $C(X_{\eta'})$ with a subspace of
  $C(X_{\eta})$, namely the set of bilinear forms whose radical
  contains $\eta'/\eta$. Similarly $\widetilde C(X_{\eta'})$ can be
  identified with a subspace of $\widetilde C(X_{\eta})$. 
\end{rmk}

\begin{df}
  We will denote by $\caC(\Lambda )$ the topological space
  \begin{displaymath}
    \caC(\Lambda )=\left. \coprod_{\eta}C(X_{\eta})\right/\sim,
  \end{displaymath}
  where $\sim$ is the equivalence relation generated by the
  identifications described in Remark \ref{rem:6}. Similarly we write
  \begin{displaymath}
    \widetilde{\caC}(\Lambda
    )=\left. \coprod_{\eta}\widetilde{C}(X_{\eta})\right/\sim.
  \end{displaymath}
\end{df}

\begin{notation}\label{def:15}
  The standard rational boundary component $\eta_{0}$ is the subspace
  generated by the first $g$ canonical vectors of $\Z^{2g}$. The
  corresponding lattice $X_{0}\coloneqq X_{\eta_{0}}$ comes with an
  isomorphism $X_{0}\simeq \Z^{(g,1)}$. We identify $C(X_{0})$ with the
  space of symmetric positive semidefinite matrices with rational kernel and
  $\widetilde C(X_{0})$ with the set of pairs  $(\Omega ,\beta )$,
  $\beta \in \R^{(1,g)}$ such that there is a $\zeta$ with $\beta =\zeta \Omega $. 

  It will be convenient to write
  down the elements of $\widetilde C(X_{0})$ as pairs of the form
  $\left(\Omega ,\zeta \Omega \right)$ with $\zeta \in
  \R^{(1,g)}$. Such an element will belong to
  $\widetilde C(X_{0})_{\Z}$ if and only if
  $\Omega \in  C(X_{0})_{\Z}$ and $\zeta $ has rational
  coefficients, but $\beta =\zeta \Omega$ has integral coefficients.
  This is just to be sure that $\beta $ belongs to the image of
  $\Omega $. But one has to be careful that $(\Omega,\zeta)$ are not
  \emph{affine} coordinates in the interior of $\widetilde
  C(X_{0})$. In particular any reference to convexity is with respect
  to the affine structure given by the coordinates $(\Omega ,\beta )$.
\end{notation}
Let $P_{g}\subset \Sp(2g,\Z)$ be the subgroup that fixes
$\eta_{0}$. It consists of symplectic matrices $\begin{pmatrix}
A&B \\ C&D
\end{pmatrix}$ satisfying $C=0$. This group has already appeared in
Example \ref{exm:1}~\ref{item:9}.
There is a group homomorphism
\begin{displaymath}
  \begin{matrix}
    P_{g} & \longrightarrow & \GL(X_{0})\\
    \begin{pmatrix}
      A & B\\
      0 & D
    \end{pmatrix}
    &\longmapsto & A.
  \end{matrix}
\end{displaymath}
We denote by $\overline \Gamma $ the image of $\Gamma \cap P_{g}$ in
$\GL(X_{0})=\GL(g,\Z)$. Similarly we denote by $\overline{\widetilde \Gamma }$
the image of $\widetilde \Gamma \cap P_{g+1}$ in $\GL(g+1,\Z)$. Then
$\overline{\widetilde \Gamma }$ is contained in the group
$\GL(X_{0})\ltimes \Z^{(1,g)}$ of
matrices of the form $
\begin{pmatrix}
  A& 0\\
  \lambda  &1
\end{pmatrix}
$
and is a semidirect product $\overline{\widetilde \Gamma }=\overline \Gamma
\ltimes \Z^{(1,g)}$.
The group $\GL(X_{0})$ acts on $C(X_{0})$ by the action
\begin{displaymath}
  A\cdot \Omega =A\Omega A^{t}
\end{displaymath}
and the group 
$\GL(X_{0})\ltimes \Z^{(1,g)}$ acts on $\widetilde C(X_{0})$ by the
action
\begin{displaymath}
  (A,\lambda )\cdot (\Omega ,\beta )=
  (A\Omega A^{t},(\beta  + \lambda\Omega  ) A^{t}). 
\end{displaymath}
This action can also be written as
\begin{equation}\label{eq:2}
  (A,\lambda )\cdot (\Omega ,\zeta \Omega )=
  (A\Omega A^{t},(\zeta + \lambda )A^{-1} A\Omega A^{t}). 
\end{equation}
The previous actions can be restricted to $\widetilde \Gamma $ and 
$\overline{\widetilde \Gamma }$ respectively. 

\begin{df}\label{def:7}
  An \emph{admissible cone decomposition of $C(X_{0})$} is a
  set $\Sigma$ of cones in $C(X_{0})$ such that
  \begin{enumerate}
  \item\label{item:19} each $\sigma \in \Sigma $ is generated by a finite set of
    elements of $C(X_{0})_{\Z}$ and contains no lines. In other
    words it is a rational polyhedral strictly convex cone;
  \item\label{item:20} if $\sigma $ belongs to $\Sigma $ each face of $\sigma $
    belongs to $\Sigma $;
  \item\label{item:21} if $\sigma $ and $\tau $ belong to $\Sigma $ their
    intersection is a common face;
  \item \label{item:22} the union of all the cones of $\Sigma $ is $C(X_{0})$;
  \item the group $\GL(X_{0})$ leaves $\Sigma$ invariant with
    finitely many orbits.
  \end{enumerate}
\end{df}
Recall that any polyhedral decomposition satisfying conditions
\ref{item:19}--~\ref{item:22} is called a complete rational fan.    
\begin{rmk}\label{rem:7}
  If $\Sigma $ is an admissible cone decomposition of $C(X_{0})$, it induces a
  $\Sp(2g,\Z)$-invariant conical complex structure in $\caC(\Lambda )$,
  that we denote $\Sigma (\Lambda )$.  The set of $\Sp(2g,\Z)$-orbits of
  $\Sigma (\Lambda )$ is also finite. Since $\Gamma \subset
  \Sp(2g,\Z)$ is a subgroup of finite index, the set of cones $\Sigma
  (\Lambda )$ is also $\Gamma $-invariant with finitely many orbits. 
\end{rmk}

\begin{rmk}\label{rem:9} For the purpose of defining a toroidal
  compactification it is enough to have a complete rational fan on
  $\caC(\Lambda )$ invariant under the action of $\Gamma $ with
  finitely many orbits. To choose one that is invariant under the whole group
  $\Sp(2g,\Z)$ has two main advantages. First, the structures around
  the different 
  \emph{cusps} are equivalent, so it is enough to discuss what happens
  around the standard boundary component. Second, it gives a uniform
  description for all subgroups $\Gamma $, which allows for some nice
  functorial properties between the compactifications associated to
  different groups. See \cite[IV~\S6]{fc}. 
\end{rmk}

\begin{df}\label{def:8}
  Let $\Sigma $ be an admissible cone decomposition of
  $C(X_{0})$. An \emph{admissible cone decomposition of
    $\widetilde C(X_{0})$ over $\Sigma $} is a set of cones $\Pi $ in
  $\widetilde C(X_{0})$ such that

  \begin{enumerate}
  \item each $\tau \in \Pi $ is generated by a finite set of elements
    of $\widetilde C_{\Z}(X_{0})$ and contains no lines;
  \item if $\sigma $ belongs to $\Pi $ each face of $\sigma $ belongs
    to $\Pi $;
  \item if $\sigma $ and $\tau $ belong to $\Pi $ their intersection
    is a common face;
  \item the union of all the cones of $\Pi $ is $\widetilde C(X_{0})$;
  \item the group $\GL(X_{0})\ltimes \Z^{(1,g)}$ leaves $\Pi $
    invariant with finitely many orbits;
  \item for each $\tau \in \Pi $, the projection of $\tau $ to
    $C(X)$ is contained in a cone $\sigma \in \Sigma $.
  \end{enumerate}
\end{df}

\begin{rmk}\label{rem:8}
  If $\Pi $ is an admissible cone decomposition of $\widetilde C(X_{0})$, it induces a
  $G_{\Z}^{(1,g)}$-invariant conical complex structure in
  $\widetilde{\caC}(\Lambda )$, 
  that we denote $\Pi  (\Lambda )$. Since $\widetilde{\Gamma} \subset
  G_{\Z}^{(1,g)}$, the set of cones $\Pi (\Lambda )$ is also  $\widetilde{\Gamma}
  $-invariant. Moreover, the set of $\widetilde{\Gamma} $-orbits of
  $\Pi  (\Lambda )$ is finite. 
\end{rmk}

We say that $\Sigma$ (resp.  $\Pi $) is \emph{smooth} if every cone of $\Sigma
$ (resp.\ $\Pi$) is generated by part of a $\Z$-basis of  the abelian
group generated by $C_{\Z}(X_{0})$ (resp.\   $\widetilde
C_{\Z}(X_{0})$). We say that $\Pi $ is \emph{equidimensional} (over $\Sigma
$) if for every cone $\tau$
the projection of $\tau $ to 
$C(X_{0})$  is a cone $\sigma \in \Sigma $.

\begin{df}\label{def:11}
    Let $\Sigma $ be an admissible cone decomposition of $
  C(X_{0})$. An \emph{admissible divisorial function} on  
  $\Sigma $ is a continuous
  $\GL(X_{0})$-invariant function $\phi \colon
  C(X_{0})\to \R$ satisfying the following properties:
  \begin{enumerate}
  \item it is conical, in the sense that $\phi (\lambda x)=\lambda
    \phi (x)$ for all $x\in C(X_{0})$ and
    $\lambda \in \R_{\ge 0}$;
  \item it is linear on each cone $\sigma $ of $\Sigma $;
  \item takes integral values on $ C(X_{0})_{\Z}$.
  \end{enumerate}
  An admissible divisorial function is called
    \emph{strictly anti-effective} if furthermore
    \begin{enumerate}[resume]
    \item $\phi (x)>0$ for $x\not = 0$.
    \end{enumerate}
    A strictly anti-effective divisorial function is called an
    \emph{admissible polarization function} if, in addition, it
    satisfies
    \begin{enumerate}[resume]
    \item $\phi$ is concave;
    \item it is strictly concave on $\Sigma $ in the sense that, if
      $\tau $ is a cone on $C(X_{0})$ such that the restriction
      of $\phi $ is linear on $\tau $ then $\tau $ is contained in a
      cone $\sigma $ of $\Sigma $. In other words, the maximal cones
      of $\Sigma $ are the maximal cones of linearity of $\phi $.  
    \end{enumerate}
  \end{df}

  \begin{rmk}
    For a given admissible cone decomposition $\Sigma $ it may happen
    that there are no admissible polarization functions on $\Sigma$. A
    cone decomposition that admits an admissible polarization function
    is called \emph{projective}. As explained in \cite[\S V.5]{fc}, every
    admissible cone decomposition admits a smooth projective
    refinement. 
  \end{rmk}

\begin{df}\label{def:10}
  Let $\Sigma $ be an admissible cone decomposition of $C(X_{0})$ and
  $\Pi $ an admissible cone decomposition of $\widetilde 
  C(X_{0})$ over $\Sigma $. An \emph{admissible divisorial function} on  
  $\Pi $ is a continuous
  $\GL(X_{0})\ltimes \Z^{(1,g)}$-invariant function $\phi \colon
  \widetilde
  C(X_{0})\to \R$ satisfying the following properties:
  \begin{enumerate}
  \item  it is conical: $\phi(tx)=t\phi (x)$ for all $t\in \R_{\ge 0}$ and
    $x\in \widetilde C(X_{0})$;
  \item it takes rational values on $\widetilde C(X_{0})_{\Z}$ with bounded
    denominators; 
  \item it is linear on each $\tau \in \Pi $;
  \item \label{item:1} for each $\lambda  \in \Z^{(1,g)}$  and
    $q=(\Omega ,\zeta \Omega )\in
    \widetilde C(X_{0})$, the condition
    \begin{equation}\label{eq:3}
      \phi (q)-\phi (\lambda  \cdot q) = \lambda  \Omega \lambda
      ^{t}+2 \zeta \Omega \lambda  ^{t}
    \end{equation}
    holds.
    Recall that the action \eqref{eq:2} gives $\lambda \cdot q=(\Omega
    ,(\zeta +\lambda )\Omega )$.  
  \end{enumerate}
  An admissible divisorial function $\phi $ on $\Pi $ is called an \emph{admissible
    polarization function} if it also satisfies the conditions
  \begin{enumerate}[resume]
  \item $\phi $ is concave;
  \item $\phi $ is strictly concave over each cone $\sigma $ of
    $\Sigma $. That is, for each maximal cone $\tau $ over $\sigma $,
    there is a linear function $\varphi_{\tau }$ such that
    $\varphi_{\tau }(q)=\phi (q)$ for each $q\in \tau $, but
    $\varphi_{\tau }(q)>\phi (q)$ for each $q=(\Omega ,\zeta \Omega
    )\not \in \tau $, with $\Omega \in \sigma $.  
  \end{enumerate}
\end{df}
\begin{rmk}
For a given admissible cone decomposition $\Pi $, even smooth,  it may
be possible that there are no admissible 
polarization functions on $\Pi $.  Nevertheless there is always a
refinement $\Pi '$ of $\Pi $ such that there exists an admissible
polarization  function on $\Pi '$. As for $\Sigma$, the admissible
cone decompositions that
admit an admissible polarization function are called
\emph{projective}.
\end{rmk}

The theory of toroidal embeddings allows us to compactify the moduli
spaces $\caA(\Gamma )$ and $\caB(\Gamma )$.
The precise statement is as follows.
\begin{thm}
   Let $\Sigma $ be a projective admissible
  cone decomposition of $
  C(X_{0})$ and $\Pi $ a projective admissible cone decomposition of $\widetilde
  C(X_{0})$ over $\Sigma $.
  \begin{enumerate}
  \item To the cone decomposition $\Sigma $ there is attached a
    projective scheme $\overline \caA(\Gamma )_{\Sigma }$ that contains
    $\caA(\Gamma )$ as an open dense subset.
  \item To the cone decomposition $\Pi $ there is attached a
    projective scheme $\overline \caB(\Gamma )_{\Pi }$ that contains
    $\caB(\Gamma )$ as an open dense subset, and a projective morphism
    \begin{displaymath}
      \pi _{\Sigma ,\Pi }\colon 
      \overline \caB(\Gamma )_{\Pi } \to \overline \caA(\Gamma )_{\Sigma }
    \end{displaymath}
    that extends the canonical projection $\caB(\Gamma )\to \caA(\Gamma )$.
  \item If $\Sigma $ or $\Pi $ are smooth, then the corresponding
    schemes $\overline \caA(\Gamma )_{\Sigma }$ or $\overline
    \caB(\Gamma )_{\Pi}$ are smooth. If $\Pi $ is equidimensional over
    $\Sigma $, then $\pi _{\Sigma ,\Pi }$ is equidimensional. 
  \item The space $\overline \caA(\Gamma )_{\Sigma }$
    admits a stratification by locally closed subschemes, indexed by
    the $\Gamma $-orbits of $\Sigma (\Lambda )$
    \begin{displaymath}
      \overline \caA(\Gamma )_{\Sigma }=\bigcup_{\overline \sigma \in \Sigma(\Lambda )
        / \Gamma } \overline \caA(\Gamma )_{\overline \sigma }.
    \end{displaymath}
    The correspondence between cones and strata reverses dimensions
    and a stratum $\overline \caA(\Gamma )_{\overline \sigma }$ lies in the
    closure of another stratum $\overline \caA(\Gamma )_{\overline \tau }$ if
    and only if there are representatives $\sigma $ and $\tau $ of
    $\overline \sigma $ and $\overline \tau $ such that $\tau$ is a
    face of $\sigma $.
  \item There is an analogous stratification of $\overline \caB(\Gamma
    )_{\Pi }$ indexed by the $\widetilde \Gamma $-orbits of $\Pi
    (\Lambda )$.
  \end{enumerate}
\end{thm}

\begin{rmk}\label{rem:5}
The projectivity condition is used to show that formal versions of these moduli spaces are algebraizable. 
\end{rmk}

\subsection{Local coordinates}\label{sec:local_coordinates}
Assume now that  $\Sigma$ and $\Pi$ are smooth.  
We want to describe local coordinates of $\overline \caA(\Gamma
)_{\Sigma }$ and $\overline \caB(\Gamma
)_{\Pi }$.

Let $\overline \sigma $ and 
$\overline \tau $ be orbits of cones in $\Sigma(\Lambda ) $ and $\Pi (\Lambda )$
respectively. To $\overline \sigma $
corresponds a stratum $\overline \caA(\Gamma )_{\overline \sigma }$
and to $\overline \tau
$ a stratum $\overline \caB(\Gamma )_{\overline \tau }$. We choose a point
$x_{\overline \sigma }\in \overline \caA(\Gamma )_{\overline \sigma }$ and a point
$y_{\overline \tau }\in \overline \caB(\Gamma )_{\overline \tau }$.
We now describe local
coordinates in both spaces around $x_{\overline \sigma}$ and
$y_{\overline \tau }$. These
coordinates and the uniformization map depend on the choice of
representatives $\sigma $ and $\tau $ of the orbits.   For simplicity
we will assume that $\sigma \in \Sigma $ and that $\tau \in \Pi
$. Note that, in general $\sigma \subset C(X_{\eta})$ and $\tau
\subset \widetilde C(X_{\eta})$, for some rational boundary component
$\eta$ of rank $g$, so one can find a $\gamma \in
\Sp(2g,\Z)$ such that $\eta_{0}\subset \gamma \eta$. Then it is enough
to translate everything by $\gamma $ to be in the case  $\sigma \in
\Sigma $ and that $\tau \in \Pi$.  

We start with $x_{\overline \sigma }$.
Write $G=g(g+1)/2$ for the dimension of $\caA(\Gamma )$ and let $n=
\dim \sigma $.  For $r>0$, we denote
by $\Delta _{r}\subset \C$ the  disk of radius $r$ centered at $0$.
The chosen cone $\sigma $ is generated by a set of
symmetric half-integral positive semi-definite 
matrices $\Omega _{1}, \dots, \Omega _{n}$ that are part of an
integral basis of the lattice of symmetric positive semidefinite half-integral
matrices. Then there exist
\begin{enumerate}
\item a
symmetric positive definite matrix $\Omega _{0}$ and symmetric matrices
$\Omega _{n+1},\dots,\Omega _{G}$, such that the set $\Omega
_{1},\dots,\Omega _{G}$ is a basis of the same lattice;
\item positive real numbers 
$0<r_{1},\dots, r_{G}< 1$ such that 
\begin{displaymath}
  U\coloneqq \left\{ i\Omega _{0}+\sum_j t_{j} \Omega _{j}\, \middle
    |\,
    \begin{alignedat}{2}
      \im(t_{j})&> - \log r_{j},\ &j&\le n,\\
      |t_{j}| &< r_{j},&j &> n.
    \end{alignedat}
\right\}\subset \caH_{g};
\end{displaymath}
\item a coordinate neighborhood $V \subset \overline
\caA(\Gamma )_{\Sigma }$ centered at $x_{\overline \sigma }$ of the form
$\Delta _{r_{1}}\times \dots \times \Delta _{r_{G}}$, 
\end{enumerate}
 such that the
uniformization map $\caH_{g}\to 
\caA(\Gamma )$ sends $U$ to $V$ through the map 
\begin{displaymath}
  (t_{1},\dots,t_{G})\mapsto \left(z_{1},\dots,z_{G})=(e^{2\pi i
    t_{1}},\dots,e^{2\pi i t_{n}},t_{n+1},\dots,t_{G}\right). 
\end{displaymath}

The situation for  $y_{\overline \tau }$ is similar. We write $d=G+g$
for the dimension of $\caB(\Gamma )$. Let $n'=\dim \tau
$.  The chosen cone $\tau $ is
generated by part of an integral basis $\{(\Omega' _{1},\beta
_{1}),\dots ,(\Omega' _{n'},\beta _{n'})\}$, where the $\Omega '_{j}$
are half-integral positive semi-definite 
symmetric matrices, and the vectors $\beta _{j}$ are integral and of
the form
$\beta _{j}= \zeta _{j}\Omega' _{j}$. Then there exist
\begin{enumerate}
\item  a pair $(\Omega '_{0},\beta _{0})$ with $\Omega '_{0}$
  symmetric and positive definite and pairs $(\Omega' _{n'+1},\beta
_{n'+1}),\dots,(\Omega' _{d},\beta _{d}) $, such that
\begin{displaymath}
  \{(\Omega' _{1},\beta_{1}),\dots,(\Omega' _{d},\beta _{d})\}
\end{displaymath}
is an integral basis of the lattice $\widetilde C(X_{0})_{\Z}$;
\item real numbers
$0<r_{j}<1$, $j=1,\dots,d$ such that
\begin{displaymath}
    U'\coloneqq \left\{ i(\Omega' _{0},\beta _{0})+\sum_j s_{j}( \Omega' _{j},\beta _{j})\, \middle
    |\,
    \begin{alignedat}{2}
      \im(s_{j})&> - \log r_{j},\ &j&\le n',\\
      |s_{j}| &< r_{j},&j &> n'.
    \end{alignedat}
\right\}\subset \caH_{g}\times \C^{(1,g)};
\end{displaymath}
\item  a coordinate neighborhood $V'$
centered at $y_{\overline \tau }$  of the form
$\Delta _{r_{1}}\times \dots \times \Delta _{r_{d}}$,
\end{enumerate}
such that
the uniformization map sends $U'$ to $V'$ via the map 
\begin{equation}\label{eq:31}
  \left(s_{1},\dots,s_{d}\right)\mapsto \left(w_{1},\dots,w_{d}\right)=\left(e^{2\pi i
    s_{1}},\dots,e^{2\pi i s_{n'}},s_{n'+1},\dots,s_{d}\right). 
\end{equation}

\begin{df}\label{def:16}
  Given data $(\Omega '_{j},\beta _{j})$, $j=0,\dots, d$ and
  $r_{j}$, $j=1,\dots,d$ satisfying the previous conditions, for any
  choice $0<r_{j}'<r_{j}$, the open set
  \begin{displaymath}\left\{ i(\Omega' _{0},\beta _{0})+\sum_j s_{j}(
      \Omega' _{j},\beta _{j})\, \middle 
    |\,
    \begin{alignedat}{2}
      \im(s_{j})&> - \log r'_{j},\ &j&\le n',\\
      |s_{j}| &< r'_{j},&j &> n'.
    \end{alignedat}
  \right\}\subset U'
\end{displaymath}
will be called a \emph{standard} open set.
\end{df}

For future use we need the following technical lemma.
\begin{lem}\label{lemm:1} With notations as in definition
  \ref{def:16}, the set
  \begin{displaymath}
    K^{\circ}\coloneqq \left\{ \left(\beta _{0}+\sum_{j} x_{j}\beta  _{j}\right)\left(\Omega'
      _{0}+\sum_j x_{j} \Omega' _{j}\right)^{-1}\, \middle
    |\,
    \begin{alignedat}{2}
      x_{j}&\ge - \log r'_{j},\ &j&\le n',\\
      |x_{j}| &\le r_{j}',&j &> n'.
    \end{alignedat}
\right\}
     \end{displaymath}
is bounded in $\R^{(1,g)}$. Hence its closure $K$ is compact. 
\end{lem}
\begin{proof}
  Arguing as in \cite[Section 3.3]{bghdj_sing}, in particular
\cite[Lemma 3.5]{bghdj_sing}, we can reduce to the case when $\ker
\Omega' _{j}\subset \ker \Omega '_{j+1}$ for $j=1,\dots,n'-1$ and the
matrix $\Omega '_{j}$ 
can be written in block notation as
\begin{displaymath}
  \begin{pmatrix}
    \Omega ''_{j} & 0\\
    0 & 0
  \end{pmatrix}
\end{displaymath}
with $\Omega''_{j}$ positive definite. Let $\rho _{j}$ be the rank of $\Omega '_{j}$.
If we fix $(x_{n'+1},\dots,x_{d})$, by \cite[Lemma 3.6]{bghdj_sing}
there is a real number $C_{1}(x_{n'+1},\dots,x_{d})>0$ such that for
all $1\le p, q \le g$ we 
have
\begin{displaymath}
  \left | \left(\Omega'
    _{0}+\sum_j x_{j} \Omega' _{j}\right)^{-1}_{p,q}\right |\le
  \frac{C_{1}\left(x_{n'+1},\dots,x_{d}\right)}{ \sum_{j\colon \rho _{j}\ge \min(p,q)}x_{j}}.
\end{displaymath}
Since the set
\begin{displaymath}
  \left\{ \Omega'
      _{0}+\sum_{j>n'} x_{j} \Omega' _{j}\, \middle
    |\, |x_{j}| \le r_{j}'
\right\}
\end{displaymath}
is a compact subset of the set of positive definite matrices,
following the proof of  \cite[Lemma 3.6]{bghdj_sing}, one sees that
the constant $C_{1}$ can be chosen uniformly for all
$(x_{n'+1},\dots,x_{d})$.   
Clearly, there is a second constant $C_{2}$ such that for all $1\le p\le g$ we have
\begin{displaymath}
  \left|\left(\beta _{0}+\sum_{j} x_{j}\beta_{j}\right)_{p }\right|\le C_{2}
  \sum_{j\colon \rho _{j}\ge p} x_{j}.
\end{displaymath}
These two bounds imply the lemma. 
\end{proof}

Assume now that $\sigma $ and $\tau $ are maximal, so $n=G$ and
$n'=d$, that $\pi (\tau )\subset \sigma $ and that $\pi _{\Sigma ,\Pi
}(y_{\overline \tau })=x_{\overline \sigma }$.  
We describe the map $\pi _{\Sigma ,\Pi }$ in the previously introduced 
coordinates.
Since each $\Omega '_{j}$, $j=1,\dotsc ,n'$ is contained in the cone generated by the
$\Omega_k$'s, $k = 1, \dotsc, n$ and we are assuming that the cone
decomposition is smooth, 
there are integer vectors $\underline
a_{j}=(a_{j,1},\dots,a_{j,n})\in \Z^{n}_{\ge 0} $ such that
\begin{displaymath}
  \Omega '_{j}=\sum_{k}a_{j,k}\Omega _{k} \, , \quad j=1,\dotsc, n'.
\end{displaymath}
Then the map $\pi _{\Sigma ,\Pi }$ is given in the $w$- and $z$-coordinates
 by
\begin{displaymath}
  \left(w_{1},\dots ,w_{d}\right)\mapsto
  \left(\prod_{j=1}^{G+g}w_{j}^{a_{j,k}}\right)_{k=1,\dots,G}. 
\end{displaymath}

\subsection{Extending the line bundle of Siegel--Jacobi forms} 
\label{sec:extend-line-bundle}

We next discuss the extension of the line bundles $L_{k,m}$ to the
toroidal compactifications.

We start by recalling how to extend the line bundle of modular forms
on the Siegel modular variety. This is related with the construction of
the Satake--Baily--Borel (or minimal) compactification (which is in general not toroidal).

Let $\Sigma $ be a projective admissible cone decomposition of $
C(X_{0})$. The universal abelian variety $\caB(\Gamma )$ over
$\caA(\Gamma )$ can be uniquely extended to a semi-abelian variety
over $\overline \caA(\Gamma )_{\Sigma }$ that we denote
$\overline \caB(\Gamma )_{\Sigma }^{0}$. The zero section $e\colon \caA(\Gamma
)\to \caB(\Gamma )$ extends to a section $\overline e\colon \overline \caA(\Gamma
)_{\Sigma }\to \overline \caB(\Gamma )_{\Sigma }^{0}$. Therefore 
the line bundle $M= \det\left(e^*\Omega ^{1}_{\caB(\Gamma
  )/\caA(\Gamma )}\right)$ on $\caA(\Gamma )$ can be extended
canonically to the  line bundle
\begin{displaymath}
\overline M=\det\left(\overline
  e^*\Omega ^{1}_{\overline \caB(\Gamma
  )^{0}_{\Sigma }/\overline \caA(\Gamma )_{\Sigma }}\right)  
\end{displaymath}
on  $\overline \caA(\Gamma
)_{\Sigma }$ by \cite[\S V.1]{fc}. Moreover there is a
$k >0$ such that the
line bundle $\overline M^{\otimes
  k}$ is globally generated \cite[Chapter V, Proposition
2.1]{fc}. Here we are using  
that $\Gamma $ is neat and in particular torsion-free.

Let $R_{\Gamma }$ be the graded ring
\begin{displaymath}
  R_{\Gamma }=\bigoplus _{k\ge 0} H^{0}\left(\overline \caA(\Gamma
)_{\Sigma }, \overline M^{\otimes k}\right).
\end{displaymath}
By \cite[Chapter V, Theorem 2.3]{fc}, the ring $R_\Gamma$ is
finitely generated and does not depend on the choice of $\Sigma $. Then
\begin{equation}\label{eq:28}
  \caA(\Gamma )^{\ast}=\Proj(R_{\Gamma })
\end{equation}
is the Satake--Baily--Borel compactification of $\caA(\Gamma )$, a
projective complex variety. In particular this implies that
$\caA(\Gamma )$ is quasi-projective and $M$ is an algebraic line
bundle.  Finally we have a canonical projection map
  \begin{displaymath}
    \overline \caA(\Gamma )_{\Sigma }\longrightarrow
    \caA(\Gamma )^{\ast}.
  \end{displaymath}

\begin{lem}\label{lemm:10}
  The standard invariant metric on $M$ extends to a  psh metric on
  $\overline M$. 
\end{lem}
\begin{proof}
  It follows from Lemma \ref{lem:psh} that the standard invariant metric is psh on $M$.
  Then, by the theory of psh functions, we only  have to show that for every
  point $y\in \overline \caA(\Gamma 
  )_{\Sigma }$ there exists $k \in \Z_{>0}$ and a local section $s$ of $\overline M^{\otimes k}$ that
  generates $\overline M^{\otimes k}$ around $y$, such that the function $-\log \|s\|$
  is bounded above locally around $y$.
  For simplicity we discuss only the case of a point corresponding to
  a maximal cone of $\Sigma $ as the general case is only notationally
  more complex.   
  Let $\sigma \in \Sigma $ be a cone of maximal dimension. Since
  $\sigma $ is maximal, the stratum $\overline \caB(\Gamma
  )_{\overline \sigma }$ contains a single point $y_{\overline
    \sigma}$. Let
  $U\subset \caH_{g}$ and
  $y_{\overline  \sigma }\in V\subset \overline \caA(\Gamma )_{\Sigma }$ be open
  sets as in the previous section.

  Let $(u_{1},\dots,u_{g})$ be linear coordinates on $\C^{(1,g)}$
  defined by an integral basis. 
  If $V$ is small enough, the symbol 
  $du_{1}\wedge \dots\wedge du_{g}$ defines a generating section of
  $M$ over $V\cap \caA(\Gamma)$. As explained in  \cite[V
  \S1]{fc} this symbol extends to a generating section of $\overline
  M$ on $V$. Let $k \in \Z_{>0}$ and let  $s=f(du_{1}\wedge \dots\wedge du_{g})^{\otimes k}$
  be a section of $\overline M^{\otimes k}$ over $V$. Note that $f$ lifts to $U \subset \caH_g$ and can be interpreted there as a 
  meromorphic Siegel modular form on
  $\caH_{g}$ of weight $k$. Then the
  section $s$ can be
  extended to a local generating 
  section of $\overline M^{\otimes k}$ around $y_{\overline \sigma }$ if and
  only if $-\log |f|$ is bounded in $U$ (for $U$ small enough). Assume therefore that $-\log |f|$ is bounded on $U$.
   Then
  \begin{displaymath}
    -\log \|s\| = -\log |f|-k \log \det\left(\sum \frac{-1}{2\pi }\log |z_{j}|\Omega _{j}\right) 
  \end{displaymath}
  is bounded above. Indeed, after shrinking $U$
  if necessary, we can assume that $2\pi \le -\log|z_{j}|$ for all $j$. Since the
  $\Omega _{j}$ are positive semidefinite  and since for two semidefinite
  matrices $A,B$ the inequality $\det(A+B)\ge \det(A)$ holds, we deduce that
  \begin{displaymath}
    \det\left(\sum \frac{-1}{2\pi }\log |z_{j}|\Omega _{j}\right)
    \ge \det \left(\sum \Omega _{j}\right)
  \end{displaymath}
  from which the boundedness above follows. 
\end{proof}

We next recall how the theory of toroidal compactifications allows us
to extend line bundles from $\caA(\Gamma )$ and $\caB(\Gamma )$ to
$\overline \caA(\Gamma )_{\Sigma }$ and $\overline \caB(\Gamma )_{\Pi }$.

\begin{df}\label{def:12} To every divisorial function $\phi $ we
  associate a divisor $D_{\phi }$ as follows. Each irreducible
  component $D_{\alpha }$ of the boundary $\overline \caA(\Gamma
  )_{\Sigma }\setminus \caA(\Gamma )$ corresponds to a $\overline
  \Gamma$-orbit of one-dimensional cones of $\Sigma $. Choose one
  element of this orbit $\sigma _{\alpha }$. Let $v_{\alpha }$ be the
  primitive generator of $\sigma _{\alpha }\cap C(X_{0})_{\Z}$. Then we write
  \begin{displaymath}
    D_{\phi }=\sum_{\alpha }-\phi (v_{\alpha })D_{\alpha }.
  \end{displaymath}
  Since $\phi $ is $\overline \Gamma $-invariant, this divisor is
  independent of the choice of $\sigma _{\alpha }$. 
\end{df}

\begin{prop}\label{prop:8} The divisor $D_{\phi }$ has support
  contained in the boundary $\overline \caA(\Gamma
  )_{\Sigma }\setminus \caA(\Gamma )$. If the divisorial function is
  strictly anti-effective then $-D_{\phi }$ is effective and the support of
  $D_{\phi }$ agrees with the whole boundary. If $\phi $ is an
  admissible polarization function, then $\caO(D_{\phi })$ is
  relatively ample with respect to the canonical projection
  \begin{displaymath}
    \overline \caA(\Gamma )_{\Sigma }\longrightarrow
    \caA(\Gamma )^{\ast}.
  \end{displaymath}
\end{prop}
\begin{proof}
  The first two  statements follow directly from the definition of
  $D_{\phi }$ and the last statement follows from \cite[Theorem
  V.5.8]{fc} and the proof of \cite[Proposition II.7.13]{Hartshorne:ag}. 
\end{proof}
Also the next result follows from the theory of toroidal compactifications (see \cite[Chapter VI, Theorem 1.13]{fc}).

\begin{prop} \label{prop:4}   Let $\Sigma $ be an admissible
  cone
  decomposition of $
  C(X_{0})$ and $\Pi $ an admissible cone decomposition of $\widetilde
  C(X_{0})$ over $\Sigma $. Let $\phi $ be an
  admissible divisorial
  function on $\Pi $ and let $m>0$ be an integer such that
  $m\phi $ has 
  integral values on $\widetilde C(X_{0})_{\Z}$.   Associated with $m\phi $
  there is a line bundle $\overline{B}_{m\phi}$ on $\overline
  \caB(\Gamma )_{\Pi }$ such 
  that its restriction to $\caB(\Gamma )$ agrees with
  $B^{\otimes m}$. Moreover, if $\phi$ is an
  admissible polarization
  function  then $\overline{B}_{m \phi}$ is relatively ample with
  respect to the projection map
  \begin{displaymath}
   \pi_{\Sigma,\Pi}\colon  \overline {\caB}(\Gamma )_{\Pi }\to
    \overline {\caA}(\Gamma )_{\Sigma  }.
  \end{displaymath}
  In particular, if $\Pi $ is projective,  the latter map is projective. 
\end{prop}

\begin{rmk}\label{rem:3} Condition \ref{item:1} in Definition
  \ref{def:10} ensures that the restriction of
  $\overline{B}_{m\phi}$ to $\caB(\Gamma )$ agrees with 
  $B^{\otimes m}$. In fact this condition implies that the restriction of
  $\overline{B}_{m\phi}$ to $\caB(\Gamma )$ satisfies
  the cocycle condition determining $B^{\otimes m}$.
\end{rmk}

\begin{rmk}\label{rem:4}
   Let $\Sigma $ be an admissible cone
  decomposition of $
  C(X_{0})$ and $\Pi $ an admissible cone decomposition of $\widetilde
  C(X_{0})$ over $\Sigma $.  We can combine admissible divisorial
  functions on $\Sigma $ and $\Pi $. Let $\phi $ be an
  admissible divisorial
  function on $\Pi $ and $\psi $ an admissible divisorial function on
  $\Sigma $. By composing with the projection $\widetilde
  C(X_{0})\to C(X_{0})$, the function $\psi $ defines a function on $\widetilde
  C(X_{0})$
  which (by abuse of notation) we also denote $\psi $. Let $m>0$ be an integer such that
  $m\phi $ has  
  integral values on $\widetilde C(X_{0})_{\Z}$.
  Then $m\phi +\psi $ is again an
  admissible divisorial function with integral values on $\widetilde
  C(X_{0})_{\Z}$ and, by construction,  
  \begin{displaymath}
    \overline B_{m\phi +\psi}=
    \overline B_{m\phi }\otimes (\pi _{\Sigma ,\Pi })^{\ast}\caO(D_{\psi }).
  \end{displaymath}
\end{rmk}

Combining the previous results we see that the line bundle of
Siegel--Jacobi forms $L_{k,m}$ can be extended to a toroidal
compactification with the help of an admissible divisorial
function. 

\begin{df}\label{def:13}
  Let $\Sigma $, $\Pi $ and $\phi $ as in Proposition
  \ref{prop:4}.  Let $\ell>0$ be an integer such that $\ell \phi $ has
  integral values on $\widetilde C(X_0)_\Z$. Then for every integer $m$
  divisible by $\ell$ and for every integer $k$ we define the line
  bundle
  \begin{displaymath}
    \overline L_{k,m,\phi } = (\pi _{\Sigma ,\Pi })^{\ast}\overline
    M^{k}\otimes \overline B_{m\phi }
  \end{displaymath}
on $\overline \caB(\Gamma )_{\Pi }$.
\end{df}
Note that $m\phi$ is integer-valued, so that $ \overline B_{m\phi }$ is well-defined. 
Clearly the restriction of $\overline L_{k,m,\phi }$ to $\caA(\Gamma
)$ agrees with $L_{k,m}$. 

\begin{rmk}\label{rem:ext}
  The extension depends on the choice of $\phi $. Moreover, for $m$ 
  not divisible by $\ell$ we can only extend $L_{k,m}$ as a $\Q$-line
  bundle. 
\end{rmk}

We next see a criterion for when a rational section of $\overline
L_{k,m,\phi }$ is holomorphic at a point of the boundary.

\begin{lem}\label{lemm:23}   Let $\Sigma $ and $\Pi $ be as in Proposition
  \ref{prop:4}. Let $\phi $ be an  admissible divisorial
  function on $\Pi $ and let $k,m>0$ be integers such that
  $m\phi $ has 
  integral values on $\widetilde C(X_{0})_{\Z}$. Let $f$ be a rational section of $L_{k,m}$,
  that we view as a meromorphic Siegel--Jacobi form of weight $k$ and
  index $m$. Let
  $\tau \in \Pi $ and $x\in \overline \caB(\Gamma 
  )_{\overline \tau  }$. Let $V'$ be a sufficiently small open coordinate neighborhood of $x$
  and $U\subset \caH_{g}\times 
  \C^{(1,g)}$ a standard open set as in Definition \ref{def:16}.    
  Then $f$ extends to a holomorphic section of $\overline L_{k,m,\phi
  }$ around $x$ (resp. a 
  non-vanishing holomorphic section) if and only if the function
  \begin{displaymath}
    -\log |f(Z,W)|-2\pi m\phi (\im Z,\im W)
  \end{displaymath}
  is bounded below (resp. bounded) in $U$.
\end{lem}
\begin{proof}
  We start by recalling a few steps in the construction of 
  $\overline \caB(\Gamma )_{\Pi }$ and 
  $\overline{B}_{m \phi}$ (see the proof of \cite[Chapter VI, Theorem
  1.13]{fc}). We describe the situation analytically over the complex
  numbers as that is enough for our purposes.

  Over an open set $V$ of $\overline \caA(\Gamma )_{\Sigma}$ that
  contains the image of $V'$ there is a complex manifold $P$ together with a free discontinuous action of the lattice $\Z^{(1,g)}$ such that
  \begin{displaymath}
    P/\Z^{(1,g)}=(\pi _{\Sigma ,\Pi })^{-1}(V)\subset \overline
    \caB(\Gamma )_{\Pi }. 
  \end{displaymath}
  The map $p\colon P \to (\pi _{\Sigma ,\Pi })^{-1}(V)$ is \'etale and
  we can find an open subset $V''\subset P$, that depends on the
  representative $\tau $ of $\overline \tau $, and that maps isomorphically
  to $V'$. Thus the holomorphic coordinates of $V'$ also give us
  holomorphic coordinates of $V''$. Moreover, the uniformization $U\to
  V'$ factors through $V''$ as the map $U\to V''$ is also given by
  formula \eqref{eq:31}. Then the preimage of $\overline{B}_{m \phi}$
  in $V''$ is the sheaf $\caO(D)$, where
  \begin{displaymath}
    D=\sum_{j}-m\phi (\Omega '_{j},\beta _{j}) D_{j},\qquad
    D_{j}=\{w_{j}=0\}.
  \end{displaymath}
  As in the proof of Lemma \ref{lemm:10} the symbol $du_{1}\wedge
  \dots \wedge du_{g}$ defines a generating section of
  $(\pi _{\Sigma ,\Pi })^{\ast}\overline M$ over $V'$. 
  Therefore $f(du_{1}\wedge
  \dots \wedge du_{g})^{\otimes k}$ extends to a holomorphic section of $\overline
  B_{m\phi }\otimes (\pi _{\Sigma ,\Pi })\overline M^{\otimes k}$ if and only if the function 
  \begin{displaymath}\label{eq:27}
    g=f\prod _{j}w_{j}^{-m\phi (\Omega '_{j},\beta _{j})}
  \end{displaymath}
  is holomorphic in $V''$. Since $g$ is meromorphic, it is
  holomorphic if and only if the function 
  \begin{align*}
    -\log |g| &= -\log |f|+\sum_{j} m\phi (\Omega '_{j},\beta
                _{j})\log|w_{j}|\\
    &=
      -\log |f|+\sum_{j} -2\pi m\phi (\Omega '_{j},\beta _{j})\im(s_{j})\\
    &= -\log |f| -2\pi m\phi \left(\sum_{j}\im(s_{j}(\Omega '_{j},\beta _{j}))\right)
  \end{align*}
  is bounded below. This proves the lemma for holomorphic
  sections. Similarly $g$ is holomorphic and non-vanishing if and
  only if $-\log|g|$ is bounded; this proves the second case. 
\end{proof}

\begin{cor}\label{cor:5}
       Let $f\in H^0\left(\caB(\Gamma),L_{k,m}\right)$ be a section that we see as
     a holomorphic function on $\caH_{g}\times \C^{(1,g)}$ satisfying
     conditions \ref{item:17} and \ref{item:18} of Definition~     \ref{def:SJ_forms}. Then $f\in H^{0}(\overline \caB(\Gamma 
  )_{\Pi }, \overline L_{k,m,\phi })$ if and only if,
     for every $M \in \Sp(2g, \Z)$, the function
     \begin{displaymath}
       \psi (Z,W)=-\log\left|f|_{k,m}M(Z,W)\right|)-2\pi m\phi (\im Z,\im W)
     \end{displaymath}
     is bounded below in each standard open set as in Definition \ref{def:16}. 
   \end{cor}
   \begin{proof}
     This is just a reformulation of Lemma \ref{lemm:23} taking into
     account that $\Sp(2g,\Z)$ acts transitively on the set of
     rational boundary components of a given rank.
   \end{proof}

   Once we have extended the line bundle of Siegel--Jacobi forms to a toroidal
   compactification, we study when the standard invariant metric
   extends as a psh metric. For this we need a definition. 

\begin{df}\label{def:14}
  An admissible divisorial function $\phi$ on $\Pi$ is called
  \emph{sufficiently negative} if
  \begin{displaymath}
    \phi (\Omega ,\zeta \Omega )\le -\zeta \Omega \zeta ^{t}.
  \end{displaymath}
\end{df}

\begin{rmk}\label{rem:exists_suff_neg}
For each given smooth $\Pi $ we can always find a sufficiently
negative admissible  divisorial function. Indeed, for every cone $\tau $ of $\Pi $ take the linear
function that agrees with $-\zeta \Omega \zeta ^{t}$ on the one
dimensional faces of  $\tau $. This function has
rational values on $\widetilde C(X_{0})_{\Z}$ with bounded
denominators. Nevertheless, even if the function $(\Omega ,\beta
)\mapsto -\beta \Omega ^{-1}\beta ^{t}$ is concave, the function we
have constructed is not necessarily concave. In particular it might
not be a polarization. 
\end{rmk}

The interest of the definition of \emph{sufficiently negative} admissible divisorial
functions lies in the following result.

\begin{lem}\label{lemm:9}
  Let $\Sigma $ and $\Pi $ be as in Proposition
  \ref{prop:4}. Let $\phi $ be an  admissible divisorial
  function on $\Pi $ and let $m>0$ be an integer such that
  $m\phi $ has 
  integral values on $\widetilde C(X_{0})_{\Z}$.  If $\phi $ is sufficiently
  negative then the
  metric of Definition \ref{def:6} on $B^{\otimes m}$ 
  extends to a singular psh metric 
  on $\overline{B}_{m\phi }$.  
\end{lem}
\begin{proof}
  As in the proof of Lemma \ref{lemm:10}, since we already know that the
  standard invariant metric is psh on $B^{\otimes m}$, we only  have
  to show that,
  for every 
  point $x\in \overline \caB(\Gamma 
  )_{\Pi  }$ and every local section $s$ of $\overline{B}_{m\phi }$ that
  generates $\overline{B}_{m\phi }$ around $x$, the function $-\log \|s\|$
  is bounded above locally around $x$. Again we discuss only the case
  when $x=x_{\tau }$ corresponds to a maximal cone $\tau $ of $\Pi
  $. Let $U'\subset \caH_{g}\times \C^{(1,g)}$ and $x_{\tau }\in V'$ be 
  open coordinate sets as in Section
  \ref{sec:local_coordinates}. After shrinking $V'$ we can assume that
  $U'$ is a standard open set as in Definition \ref{def:16}. A
  rational section of $B^{\otimes m}$ determines a 
  meromorphic Siegel--Jacobi form $f $ of index $m$. By Lemma
  \ref{lemm:23} it extends to a
  generating section $s_{f}$ of  
  $\overline{B}_{m \phi}$  around $x_{\tau }$ if and only if the function
  \begin{displaymath}
    -\log |f|-2\pi m \phi 
  \end{displaymath}
  is bounded in $U'$. By the sufficient
  negativity of $\phi$ and the definition of the standard invariant metric
  in Definition \ref{def:6} we deduce that
  \begin{align*}
    -\log \|s_{f}\| &= -\log |f(\Omega , \zeta \Omega )| +2\pi m
    \zeta \Omega \zeta ^{t}\\ &=
    -\log |f(\Omega , \zeta \Omega )|-2\pi m \phi+2\pi m \phi+
    2\pi m
    \zeta \Omega \zeta ^{t}
  \end{align*}
  is bounded above, hence extends to a psh function on $V'$.
  \end{proof}

\begin{rmk}\label{rem:1} When $\Pi $ is smooth, we can choose the
  sufficiently negative function
  $\phi $ that agrees with $\zeta \Omega \zeta ^{t}$ on the rays of
  $\Pi $ and is linear on every cone. Then the extension obtained is the one considered by Lear
  \cite{lear}. In this case the standard invariant metric extends to a psh metric
  that is continuous up to a set of codimension at least $2$.
\end{rmk}

Combining Lemmas \ref{lemm:10} and \ref{lemm:9} we obtain a criterion
for when the invariant metric on $L_{k,m}$ extends to a psh metric on
$\overline L_{k,m,\phi }$.

\begin{prop}\label{prop:5} Let $\Sigma $, $\Pi $, $m$, $\phi $ and $k$
  be as in Definition \ref{def:13}. 
  Assume that $\phi $ is a sufficiently
  negative admissible divisorial
  function on $\Pi $.
  Then the standard invariant metric $h^{\inv}$ of
  $L_{k,m}$ of Definition~\ref{def:6} extends to a singular psh metric
  on $\overline L_{k,m,\phi }$
  that we denote by~$\overline{h}^{\inv}$.
  \end{prop}

In a different direction one may ask when the extension $\overline L_{k,m,\phi
}$ of Definition \ref{def:13} is ample. 

\begin{lem}\label{lemm:19}
   Let $\Sigma $ be a projective admissible
  cone
  decomposition of $
  C(X_{0})$ and $\Pi $ a projective admissible cone decomposition of $\widetilde
  C(X_{0})$ over $\Sigma $. Let $\phi $ be an
  admissible polarization function on $\Pi $ and $\ell_{0}\ge 1$ an integer
  such that $\ell_{0}\phi$ has integral values on $\widetilde
  C(X_{0})_{\Z}$. Then there
     exists an admissible polarization function $\psi $ on $\Sigma$
     and for every $m>0$ a number $k_{0}\gg 0$ such that, for every
     $k\ge k_{0}$, the line  bundle
     $\overline L_{\ell_{0}k,\ell_{0}m,\phi+\psi }$ is ample on $\overline \caB(\Gamma
     )_{\Pi}$. 
   \end{lem}
\begin{proof}
  Let $\psi _{0}$ be a polarization function
  on $\Sigma $. By  Proposition \ref{prop:8} the line
  bundle $\caO(D_{\psi_{0} })$ is relatively ample for the map
     $\overline \caA(\Gamma )_{\Sigma }\to \caA(\Gamma )^{\ast}$. By
     Proposition \ref{prop:4} the line bundle $\overline B_{\ell_{0}\phi }$
     is relatively ample with respect to the map
     \begin{displaymath}
      \pi_{\Sigma,\Pi}\colon  \overline \caB(\Gamma )_{\Pi }\longrightarrow \overline
       \caA(\Gamma )_{\Sigma }. 
     \end{displaymath}
     Therefore we can find an integer $a>0$ such that the line bundle
     $\overline B_{\ell_{0}\phi }\otimes (\pi _{\Sigma ,\Pi })^{\ast}\caO(a\ell_{0}D_{\psi _{0}})$
     is relatively ample with respect to the map 
     \begin{displaymath}
       \overline \caB(\Gamma )_{\Pi }\longrightarrow 
       \caA(\Gamma )^{\ast}. 
     \end{displaymath}
     Writing $\psi =a\psi _{0}$, following Remark \ref{rem:4} we have
     $\overline B_{\ell_{0}\phi }\otimes (\pi _{\Sigma ,\Pi
       })^{\ast}\caO(a\ell_{0}D_{\psi _{0}})
     = \overline B_{\ell_{0}(\phi +\psi )}$.

     By the definition of $\caA(\Gamma )^{\ast}$ in
     \eqref{eq:28} the line bundle $\overline M$ is ample on
     $\caA(\Gamma )^{\ast}$. Therefore there exists $k_{1}$ such
     that
     \begin{displaymath}
       \overline L_{k_{1},\ell_{0},(\phi +\psi )}=\overline
       B_{\ell_{0}(\phi +\psi )}\otimes (\pi _{\Sigma ,\Pi
       })^{\ast}\overline M^{\otimes k_{1}}
     \end{displaymath}
     is ample in $\overline \caB(\Gamma )_{\Pi }$. Therefore it is
     enough to choose $k_{0}=mk_{1}$.
   \end{proof}

   The next objective is to give a more precise comparison between the
   group of sections of $\overline L_{k,m,\phi}$ and the group $J_{k,m}(\Gamma
   )$ whenever $m\phi $ has 
   integral values on $\widetilde C(X_{0})_{\Z}$. The main tool for
   this comparison is the following consequence of Proposition
   \ref{prop:9}.

   \begin{lem}\label{lemm:25}
     Let $f\in H^0\left(\caB(\Gamma),L_{k,m}\right)$ be a section that we see as
     a holomorphic function on $\caH_{g}\times \C^{(1,g)}$ satisfying
     conditions \ref{item:17} and \ref{item:18} of Definition~\ref{def:SJ_forms}. Then $f\in J_{k,m}(\Gamma )$ if and only if,
     for every $M \in \Sp(2g, \Z)$, the function
     \begin{displaymath}
       \psi (Z,W)=-\log\left|f|_{k,m}M(Z,W)\right|)+2\pi  m\im(W)\im(Z)^{-1}\im(W)^{t}
     \end{displaymath}
     is bounded below in each standard open set (in the sense of Definition \ref{def:16}). 
   \end{lem}
   \begin{proof}
     We start by showing that, if $f\in J_{k,m}(\Gamma )$, then the
     function $\psi (Z,W)$ is bounded 
     below in each standard open set as in Definition \ref{def:16}. Let $U$ be
     a standard open set. We use the notations in Section
     \ref{sec:local_coordinates}. 
     In view of
     Proposition \ref{prop:9}, it is enough to show that there
     is a  compact $K\subset \C^{(1,g)}$ and a matrix $Z_{0}\in \caH_{g}$ such that
     $U\subset U_{Z_{0}.K}$.
     Put
     \begin{displaymath}
       Z_{0}=i\Omega _{0}'-i\sum_{j\le n'}\log (r_{j}') \Omega '_{j}-
       i\sum_{j>n'} r'_{j} \Omega '_{j}.
     \end{displaymath}
     By hypothesis $\im(Z_{0})$ is a symmetric positive definite
     matrix. Let $K$ be the compact set of Lemma
     \ref{lemm:1}. By construction,
     $U\subset U_{Z_{0},K}$, so $\psi $ is bounded in $U'$.
     For the converse, it is enough to show that any strip of the form
     $S_{Y,\zeta,K'}$ is contained in a standard open set. But this is
     clear by the shape of both spaces.  
   \end{proof}

   Putting together Lemma \ref{lemm:25} and Corollary \ref{cor:5} we
   obtain the next result.

   \begin{thm}\label{thm:3}
     Let $\Sigma $ and $\Pi $ be as in Proposition
  \ref{prop:4}. Let $\phi $ be an  admissible divisorial
  function on $\Pi $ and let $k,m>0$ be integers such that
  $m\phi $ has 
  integral values on $\widetilde C(X_{0})_{\Z}$.
  \begin{enumerate}
  \item \label{item:23} If $\phi (\Omega ,\beta )\le -\beta
    \Omega^{-1}\beta ^{t}$ (i.e. the function $\phi $ is sufficiently
    negative) then $J_{k,m}(\Gamma )\subset 
    H^{0}\left(\overline \caB(\Gamma)_{\Pi},\overline  L_{k,
        m,\phi}\right)$.
  \item \label{item:24} If $\phi (\Omega ,\beta )\ge -\beta
    \Omega^{-1}\beta ^{t}$ then $
    H^{0}\left(\overline \caB(\Gamma)_{\Pi},\overline  L_{k,
        m,\phi}\right)\subset J_{k,m}(\Gamma )$.
  \end{enumerate}
   \end{thm}

   The Koecher
  principle \cite[Lemma 1.6]{Ziegler-J} implies that, for $g\ge 2$,
  the inclusion $
    H^{0}\left(\overline \caB(\Gamma)_{\Pi},\overline  L_{k,
        m,\phi}\right)\subset J_{k,m}(\Gamma )$ is always satisfied. 
   Combining the Koecher principle with Theorem
   \ref{thm:3}~(\ref{item:23}) we obtain immediately the complex version
   of a conjecture by Kramer, \cite[Remark 2.19]{Kramer_Crelle}.

   \begin{cor}\label{cor:6} With the hypothesis of Theorem
     \ref{thm:3}, if $g\ge 2$ and $\phi $  is sufficiently negative,
     then
     \begin{equation}\label{eq:33}
       H^{0}\left(\overline \caB(\Gamma)_{\Pi},\overline  L_{k,
        m,\phi}\right)= J_{k,m}(\Gamma).
     \end{equation}
     In particular
     \begin{displaymath}
       J_{k,m}(\Gamma)= J_{k,m}^{\can}(\Gamma)=
       \varinjlim _{\Pi ,\phi } H^{0}\left(\overline \caB(\Gamma)_{\Pi},\overline  L_{k,
        m,\phi}\right) 
  \end{displaymath}
  where the inductive limit runs over all admissible  cone
  decompositions of $\widetilde C(X_{0})$ and admissible divisorial
  functions $\phi $ such that $m\phi $ has integral values on
  $\widetilde C(X_{0})_{\Z}$ with the order $(\Pi ,\phi )\le (\Pi ',\phi ')$ if
  and only if $\Pi '$  is a refinement of $\Pi $ and $\phi '\le \phi$.
\end{cor}

\begin{rmk}
  The order we are using is different from the order defined
  in \cite[2.16]{Kramer_Crelle}. The reason is that the ordered set
  defined in  \cite[2.16]{Kramer_Crelle} is codirected, while the one
  used here is directed. 
\end{rmk}

\begin{rmk}
  We have only stated the complex case of Kramer's conjecture because
  this is the framework we are working in. Nevertheless both Lemma
 ~\ref{lemm:25} and Corollary~\ref{cor:5} can be easily extended to
  forms with coefficients in a torsion-free $\Z$-module, using
  Kramer's integral $q$-expansion principle to check whether a Jacobi
  form gives a section over the integers. This proves Kramer's
  Conjecture \cite[Remark 2.19]{Kramer_Crelle} in the case where (in
  his notation) the module $M$ is torsion free. 
\end{rmk}

\begin{rmk} One may wonder if there exists a choice $k,m,\phi $, with
  $k,m>0$ such that equality \eqref{eq:33} is true and the line bundle
  $\overline L_{k,m,\phi }$ is ample. A quick glance at Lemma \ref{lemm:19} may
  suggest so. Nevertheless, if we start with a sufficiently negative
  function $\phi $, in order to obtain an ample line bundle,  we will
  need to add a positive function $\psi $ that may destroy the
  sufficient negativity. In fact a consequence of Theorem
  \ref{thm:2} is that, if $\overline L_{k,m,\phi }$ is ample, then
  $\phi $ can not be sufficiently negative. 
\end{rmk}

\subsection{The standard invariant metric is toroidal}
\label{sec:stand-invar-metr}

\begin{prop}\label{prop:7}
   Let $\Sigma $, $\Pi $ and $\phi $ be as in Proposition
  \ref{prop:4}. Assume that $\phi $ is sufficiently
  negative. Let $\ell>0$ be an integer such that $\ell\phi$ is integral. Then, for every $k$ and every $m$ divisible by $\ell$, the
  singular psh metric $\overline h^{\inv}$  on $\overline L_{k,m,\phi
  }$ is toroidal.
\end{prop}
\begin{proof}
  Again, we prove this only in a neighborhood of a point $x_{\tau }\in
  \overline \caB(\Gamma )_{\Pi }$ corresponding to a maximal cone
  $\tau $ of $\Pi$. Let $U\subset U'$ be a standard open subset (Definition
  \ref{def:16}) and $V'$ the corresponding coordinate neighborhood of
  $x_{\tau }$ as described in Section
  \ref{sec:local_coordinates}. Let $G=g(g+1)/2$ and $d=G+g$. For
  $i=1,\dots,d$ we write 
  \begin{displaymath}
    u_{i}=\im(s_{i})=\frac{-1}{2\pi }\log|w_{i}|,
  \end{displaymath}
  where $(w_{1},\dots,w_{d})$ are the coordinates of $V'$. Recall
  also that the cone $\tau $ is generated by the points $(\Omega'
  _{i},\zeta _{i}\Omega _{i}')$, $i=1,\dots,d$.

  Let $f$ be a meromorphic Siegel--Jacobi form
  that defines a rational section $s_{f}$ of $\overline L_{k,m,\phi
  }$ that is a generating local section on $V'$. This means that on
  the set $U$ the function
  \begin{displaymath}
    -\log|f|-2\pi m\phi 
  \end{displaymath}
  is bounded (see Lemma~\ref{lemm:23}). Therefore
  \begin{align*}
    -\log\|s_{f}\|
    &= -\log|f|-\frac{k}{2}\log\det \sum_{i=1}^{d}
      u_{i}\Omega' _{i}\\
    &\quad
    +2\pi m \left(\sum_{i=1}^{d} u_{i}\zeta _{i}\Omega' _{i}\right)
    \left(\sum_{i=1}^{d} u_{i}\Omega' _{i}\right)^{-1}
      \left(\sum_{i=1}^{d} u_{i}\zeta _{i}\Omega' _{i}\right)^{t}\\
    &=\gamma -\frac{k}{2}\log\det \left(\sum_{i=1}^{d}
      u_{i}\Omega' _{i}\right) +2\pi m\phi (u_{1},\dots,u_{d})\\
    &\quad
    +2\pi m \left(\sum_{i=1}^{d} u_{i}\zeta _{i}\Omega' _{i}\right)
    \left(\sum_{i=1}^{d} u_{i}\Omega' _{i}\right)^{-1}
      \left(\sum_{i=1}^{d} u_{i}\zeta _{i}\Omega' _{i}\right)^{t},
  \end{align*}
  where $\gamma $ is bounded. We already know that $\phi $ is linear
  on the cone $\tau $ and that the other two functions appearing in
  the last equation are convex, as seen in Lemma
  \ref{lem:psh}. Moreover, since $\phi $ is assumed to be sufficiently
  negative, we know that
  \begin{displaymath}
     \phi (u_{1},\dots,u_{d})
    +\left(\sum_{i=1}^{d} u_{i}\zeta _{i}\Omega' _{i}\right)
    \left(\sum_{i=1}^{d} u_{i}\Omega' _{i}\right)^{-1}
      \left(\sum_{i=1}^{d} u_{i}\zeta _{i}\Omega' _{i}\right)^{t}
    \end{displaymath}
    is bounded above. Thus it only remains to show that the function
    \begin{displaymath}
     \varphi_{1}(u_{1},\dots,u_{d})=-\log\det \left(\sum_{i=1}^{d}
      u_{i}\Omega' _{i}\right)
    \end{displaymath}
    is bounded above and Lipschitz continuous in $U$, and that the
    function   
    \begin{displaymath}
      \varphi_{2}(u_{1},\dots,u_{d})=
      \left(\sum_{i=1}^{d} u_{i}\zeta _{i}\Omega' _{i}\right)
    \left(\sum_{i=1}^{d} u_{i}\Omega' _{i}\right)^{-1}
      \left(\sum_{i=1}^{d} u_{i}\zeta _{i}\Omega' _{i}\right)^{t}
    \end{displaymath}
    is Lipschitz continuous on $U$.

    With respect to $\varphi_{1}$, we can assume that the $u_{i}\ge M$
    for some constant $M$. If $A$ is a positive definite 
    matrix and $B$ is a positive semidefinite matrix then
    $\det(A+B)\ge \det(A)$, hence
    \begin{displaymath}
      \varphi_{1}(u_{1},\dots,u_{d})\le \varphi_{1}(M,\dots,M), 
    \end{displaymath}
   and $\varphi_{1}(u_{1},\dots,u_{d})$ is bounded above. To prove that it is Lipschitz continuous, following \cite[A~4.1]{Boyd:convex}
    we first compute 
    \begin{displaymath}
      \frac{\partial \varphi_{1}}{\partial u_{i}} = -\tr \left(\left(\sum_{j=1}^{d}
          u_{j}\Omega' _{j}\right)^{-1}\Omega _{i}\right) \, . 
    \end{displaymath}
    Then, applying Lemma \ref{lemm:17} below to
    \begin{displaymath}
      A=\sum_{j} M\Omega _{j},\quad B= \sum_{j} (u_{j}-M)\Omega
      _{j},\text{ and }C=\Omega _{i},
    \end{displaymath}
    we deduce that
    \begin{displaymath}
      0\ge \frac{\partial \varphi_{1}}{\partial u_{i}}(u_{1},\dots,u_{d})\ge
      \frac{\partial \varphi_{1}}{\partial u_{i}}(M,\dots,M).
    \end{displaymath}
    Hence the function $\varphi_{1}$ is Lipschitz continuous in $U$.

    Regarding $\varphi_{2}$, by \cite[Theorem
    3.2.2]{bghdj_sing} we know that $\varphi_{2}$ is continuous and
    bounded on the open simplex $\sum_{i}u_{i}=1$, $u_{i}>0$. Since it
    is homogeneous of degree 1, it is Lipschitz continuous on the open
    quadrant $\R^{d}_{>0}$ that contains $U$. Therefore it is
    Lipschitz continuous on $U$.    
  \end{proof}

      \begin{lem}\label{lemm:17}
      Let $A$ be a real positive definite symmetric matrix of
      dimension $r$ and $B,C$ real positive semi-definite symmetric
      matrices of the same dimension. Then
      \begin{displaymath}
        0\le \tr\left((A+B)^{-1}C\right)\le \tr\left(A^{-1}C\right).
      \end{displaymath}
    \end{lem}
    \begin{proof}
    This is an immediate consequence of the Spectral Theorem. 
    \end{proof}

    We now assume the hypotheses of Proposition \ref{prop:5}. Fix a
    rational section $s$ of $\overline L_{k,m,\phi }$, and let
    $h=\overline{h}^{\inv}$ be the psh
   metric on $\overline L_{k,m,\phi }$ induced by the standard invariant
   metric. As in Section~\ref{sec:from-psh-metrics} we denote by $\D(\overline L_{k,m,\phi },s,h)$ the Weil $\R$-b-divisor associated to $s$ and $h$.  
   
 Let $(\overline\caB(\Gamma )_{\Pi })_\pi$ be a model of $\overline\caB(\Gamma )_{\Pi }$ on which the pullback $E$ of the union
   of $\operatorname{div}(s)$ and the boundary of $\overline 
   \caB(\Gamma )_{\Pi }$ has simple normal crossings. Then we can view
   $\D(\overline L_{k,m,\phi },s,h)$ as a Weil $\R$-b-divisor on $(\overline\caB(\Gamma )_{\Pi })_\pi$. 
      
\begin{cor}\label{cor:toroidal}
The Weil $\R$-b-divisor $\D(\overline L_{k,m,\phi },s,h)$ is toroidal with respect to $E$. 
\end{cor}
\begin{proof}
 By Proposition \ref{prop:7} the metric $h$ is toroidal with respect to the boundary of $\overline
   \caB(\Gamma )_{\Pi }$. Hence, by Proposition
   \ref{prop:div_of_toroidal_metric_is_toroidal}, it follows that the divisor
   $\D(\overline L_{k,m,\phi },s,h) - \operatorname{div}(s)$ is toroidal
   with respect to the boundary of $\overline \caB(\Gamma )_{\Pi
   }$, and hence with respect to $E$. Since $\dv(s)$ is
   also toroidal with respect to $E$ we deduce the result.  
\end{proof}

\section{Proof of the main result}
\label{sec:proofs_main}

We continue with the notation and assumptions of the previous
section. We assume from now on that $\Sigma $ is a smooth and
projective admissible cone decomposition of $ C(X_{0})$, and
$\Pi $ is a smooth and projective admissible cone decomposition of
$\widetilde C(X_{0})$ over $\Sigma $.   

\subsection{A b-divisor which is not Cartier}

In view of Corollary \ref{cor:toroidal}, in order to determine 
$\D(\overline L_{k,m,\phi },s,h)-\dv(s)$  it is enough to
compute Lelong numbers along toroidal divisors. Let $\tau
$ be a maximal cone of $\Pi $ and $x_{\tau }$ the corresponding
point. As before, let $G=g(g+1)/2$, $d=G+g$ and let $\zeta _{j}$, $\Omega '_{j}$, 
$j=1,\dots, d$  be local coordinates as in Section \ref{sec:local_coordinates}.
Let $u=(\Omega _{0},\zeta \Omega _{0})\in \widetilde C(X_{0})_{\Z}$ be a
primitive vector in the interior of 
the cone $\tau $. Let $\Pi '$ be an admissible cone decomposition
subdividing $\Pi$, such
 that the ray generated by $u$ is a ray of $\Pi '$. Let $P_{u}$ be the
 irreducible divisor of $\overline \caB(\Gamma )_{\Pi '}$ corresponding 
 to the ray generated by $u$.

 \begin{lem}\label{lemm:13} The Lelong number of $h$
   at $P_{u}$ is 
   given by
   \begin{displaymath}
     \nu (h,P_{u}) = -m\phi(\Omega_{0} ,\zeta \Omega_{0} )-m\zeta
     \Omega_{0} \zeta ^{t}. 
   \end{displaymath}
   In particular it does not depend on the subdivision $\Pi '$. 
 \end{lem}
 \begin{proof}
   Since the metric of $\overline M$ is good in the sense of
   \cite{hi}, it has zero Lelong numbers everywhere (see \cite[Example~2.34]{BBHJ}). So it is sufficient
   to treat the case $k=0$. Let $s$ be a rational section of
   $\overline{B}_{m\phi}$ that is generating on a neighborhood of
   $x_{\tau }$.
   Let $U\subset U'$ be a standard open subset (Definition
  \ref{def:16}) and $V'$ the corresponding coordinate neighborhood of
  $x_{\tau }$.  
   As in the proof 
   of Lemma \ref{lemm:9}, $s$ corresponds to a meromorphic Siegel--Jacobi
   form $f$ of index $m$  such that $-\log |f|-2\pi m\phi $ is bounded in
   $U$.  Since $u$ belongs to the interior
   of $\tau $ and is integral, there are positive integers
   $a_{1},\dots, a_{d} $ such that
   \begin{displaymath}
     u=(\Omega _{0},\zeta \Omega _{0})=\sum _{j=1}^{d}a_{j}(\Omega
     '_{j},\zeta _{j}\Omega '_{j}).
   \end{displaymath}
   Let $z=(z_{1},\dots,z_{d})$ be a point of $V'$ and consider the
   curve
   \begin{displaymath}
     \beta (t)=(z_{1}t^{a_{1}},\dots,z_{d}t^{a_{d}}),\quad |t|\le 1.
   \end{displaymath}

   For a general point $z$, the strict transform of the curve $\beta $ in
   $\overline \caB(\Gamma )_{\Pi '}$ goes through a general point of
   $P_{u}$.

   By the explicit description of the standard invariant metric we have
   \begin{align*}
     -\log \|s(\beta (t))\|
     &= -\log|f(\beta(t)|+2\pi m\frac{-1}{2\pi }\log|t| \zeta \Omega_{0} \zeta ^{t}+C\\
     &= \eta(t) + 2\pi m\phi \left(\frac{-1}{2\pi }\log |t| u\right)
       +2\pi m \frac{-1}{2\pi }\log|t|\zeta \Omega_{0} \zeta ^{t}+C\\
     &= \eta(t) - m\left(\phi (u)+\zeta \Omega_{0} \zeta
       ^{t}\right)\log |t|+C,
   \end{align*}
   where $C$ is a constant depending on the point $z$ and where the function
   \begin{displaymath}
     \eta(t)\coloneqq-\log|f(\beta (t))|-2\pi m \phi \left(\frac{-1}{2\pi }\log |t| u\right)
   \end{displaymath}
   is bounded. Therefore 
   \begin{equation}\label{eq:16}
     \lim _{t\to 0}\frac{|\eta(t)+C|}{-\log|t|}=0.
   \end{equation}
   Thus by Lemma \ref{lemm:8} the Lelong number $\nu
   (h,P_{u})$ is given by
   \begin{displaymath}
     \nu(h,P_{u})=- m\left(\phi (u)+\zeta \Omega_{0} \zeta
       ^{t}\right),
     \end{displaymath}
     proving the lemma.
     \end{proof}
     \begin{rmk}
   We note that in the proof above, instead of relying on the fact that good metrics
     have zero Lelong numbers to reduce to the case $k=0$, we can prove
     directly the result for arbitrary $k\ge 0$. Namely, in this case the
     singularities of the metric of  $\overline M^{\otimes k}$ can be
     absorbed into the function $\eta$. Then the function $\eta $ will
     no longer be bounded but will have a growth of the shape
     $K\log(-\log|t|))$ when $t$ goes to zero. The estimate
     \eqref{eq:16} would still be true and we could finish the proof
     in the same way.
 \end{rmk}
   Since $(\Omega_{0},\zeta \Omega_{0} ) $ belongs to the interior of the
   cone, the matrix $\Omega_{0} $ is positive definite, in particular
   invertible. Writing $\beta  =  \zeta \Omega $ (recall that the
   affine coordinates of $\widetilde C(X_{0})$ are $(\Omega ,\beta )$) we have
   \begin{displaymath}
     \zeta \Omega \zeta ^{t}=\beta  \Omega ^{-1}\beta  ^{t}.
   \end{displaymath}
   The function 
   \begin{displaymath}
     (\Omega ,\beta  )\mapsto \beta  \Omega ^{-1}\beta  ^{t}
   \end{displaymath}
   is a smooth convex function with non-zero Hessian, hence is
   not piecewise linear. On the other hand, the function $\phi $ is
   linear. Therefore the function that to each primitive vector $v$ in the maximal cone 
   $\tau $ associates the value $\nu (h,P_{v})$  is not the restriction
   of a piecewise linear function on $\tau $.

 \begin{cor}\label{cor:1}
The Weil $\R$-b-divisor $\D(\overline
   L_{k,m, \phi},s,h$) is not
   Cartier. 
 \end{cor}
 
 \begin{proof}
 Assume that it is a Cartier $\R$-b-divisor.
Replacing $\pi \colon X_{\pi }\to \overline
   \caB(\Gamma )_{\Pi }$ by a finer toroidal modification we may assume
   $D=\D(\overline L_{k,m,\phi },s,h)$ is realized as a $\R$-divisor on
   $X_{\pi }$. Moreover we can assume that the union of the support of $D$ and the
   preimage of the boundary divisor of $\overline
   \caB(\Gamma )_{\Pi }$ is contained in a simple normal crossings divisor $E$. 
   
Again, the
   theory of toroidal compactifications \cite{toroidal} assigns to
   $(X_{\pi },E)$ and $D$ a conical rational complex $\Delta $ and a piecewise
   linear function $\varphi_{D}$ such that each ray $\rho $ of $\Delta $
   corresponds to an irreducible component $E_{\rho }$  of $E$ and the value of
   $\varphi_D$ at the primitive generator of $\rho $ is
   $-\ord_{E_{\rho}}D$. Moreover, each primitive vector $v$ contained in
   a cone of $\Delta $ corresponds  to an irreducible exceptional
   divisor in some toroidal modification of $(X_{\pi
   },E)$. Similarly,
   associated to the toroidal embedding  $\caB(\Gamma ) \subset
   \overline \caB(\Gamma )_{\Pi }$ there is a rational conical complex
   $\Xi $ such that
   \begin{displaymath}
     \Xi =\Pi /\overline{\widetilde \Gamma }.
   \end{displaymath}
   Note that, by the conditions on the function $\phi$, the
   function $-m\phi (\Omega ,\zeta \Omega )-m\zeta \Omega \zeta ^{t}$
   is invariant under the action of $\overline{\widetilde \Gamma
   }$. Hence it descends to a continuous conical function $\psi $ on $\Xi
   $. As discussed above the function $\psi $ is not piecewise linear
   in any maximal cone of $\Xi $.

   Since the preimage of the boundary component is contained in $E$,
   there is a retraction map $r \colon \Delta \to \Xi $ such that
   for each cone $\sigma $ of  $\Delta$ there is a cone $\tau $ of
   $\Xi $ with $r(\sigma )\subset \tau $. Choose  a cone $\sigma $
   of $\Delta $ of maximal dimension such that $r(\sigma )$ is also
   maximal. The line bundle $\overline L_{k,m,\phi }$ and the section $s$
   define a Cartier divisor $D_{0}$ that we may assume has support on
   $E$ because we can always enlarge $E$. To the divisor $D_{0}$
   corresponds a piecewise linear function $\varphi_{D_{0}}$. Since
   $\D(\overline L_{k,m,\phi },s,h)$ is defined using Lelong numbers, Lemma \ref{lemm:13} shows that on the cone $\sigma $,
   \begin{displaymath}
     \varphi_{D}=\varphi_{D_{0}}+\psi .
   \end{displaymath}
 Since $\varphi_{D_{0}}$ is piecewise linear but $\psi $ is not, this contradicts
 the piecewise linearity of $\varphi_D$. We conclude that $\D(\overline
 L_{k,m},s,h)$ is not  Cartier.
 \end{proof}
\subsection{Graded linear series and b-divisors  associated to Siegel--Jacobi
  forms}\label{sec:graded-linear-series-1}

Let $\caB(\Gamma )$ be the universal family of abelian varieties with
level $\Gamma $. 
Let $F=K(\caB(\Gamma ))$ denote the field of rational
functions on $\caB(\Gamma )$.
Let $k,m\ge 0$ be integers and 
fix a rational section of $L_{k,m}$, that is, a meromorphic Siegel--Jacobi form $s$ of weight $k$,
index $m$ and level $\Gamma$. 
We denote
\begin{align*}
  \caJ_{k,m}(\Gamma, s)_{\ell}
  &=\left\{f\in F^{\times}\mid f s^{\ell}\in J_{\ell k,\ell m}(\Gamma) \right\}\cup \{0\},\\
  \caJ_{k,m}(\Gamma, s) &=\bigoplus _{\ell} \caJ_{k,m}(\Gamma, s)_{\ell}t^{\ell}\subset F[t],\\
  \caJ^{\cusp}_{k,m}(\Gamma, s)_{\ell}
  &=\left\{f\in F^{\times}\mid f s^{\ell}\in J^{\cusp}_{\ell k,\ell m}(\Gamma) \right\}\cup \{0\},\\
  \caJ^{\cusp}_{k,m}(\Gamma, s) &=\bigoplus _{\ell}
  \caJ^{\cusp}_{k,m}(\Gamma, s)_{\ell}t^{\ell}\subset F[t].
\end{align*}
Both $\caJ^{\cusp}_{k,m}(\Gamma, s)$ and $\caJ_{k,m}(\Gamma, s)$ are graded
linear series.

The first observation to make is that the above graded linear series
are non-trivial as long as $k>0$.

\begin{lem}\label{lemm:18}
  Let $k>0$ and $m\ge 0$. There is an integer $\ell >0$ such that
  $\caJ^{\cusp}_{k,m}(\Gamma, s)_{\ell}\not =\{0\}$.  
\end{lem}
\begin{proof}
  Let $\ell>0$ be big
  enough so that we can write
  \begin{displaymath}
    \ell k = k_{1}+k_{2}
  \end{displaymath}
  with $k_{1}>2g$, $k_{2}>g+2$ both even.   As recalled in Example
  \ref{exm:1}.\ref{item:13}, using Poincar\'e 
  series one can produce a non-zero cusp form $\varphi_{1}$ of weight $k_{1}$ and index
  $0$. By Example   \ref{exm:1}.\ref{item:10} there is a
  non-zero Siegel--Jacobi form $\varphi_{2}$ of weight $k_{2}$ and
  index $\ell m$. Then by Lemma~\ref{lem:multiplicativity} we have $\varphi_{1}\varphi_{2}$ a non-zero cusp form of
  weight $\ell k$ and index $\ell m$ and the non-zero function
  $\varphi_{1}\varphi_{2}/s^{\ell}$ belongs to
  $\caJ^{\cusp}_{k,m}(\Gamma, s)_{\ell}$. 
\end{proof}

We next choose $\Sigma $ and $\Pi $ admissible cone decompositions as in
Proposition~\ref{prop:4}. Let $\phi $ be a  sufficiently negative admissible
divisorial function on $\Pi $ (which exists by Remark
\ref{rem:exists_suff_neg}). Assume that $m$ is divisible enough so
that 
$m\phi $ has integral values on $\widetilde C(X_{0})_{\Z}$.
To ease notation we write $X=\overline \caB(\Gamma )_{\Pi }$. 

Let $\overline L_{k,m,\phi }$ be the extension of $L_{k,m}$ on $X$
determined by the divisorial function $\phi $. We can view the
meromorphic Siegel--Jacobi form $s$ as a rational section of $\overline
L_{k,m, \phi}$.

\begin{prop}\label{prop:10}
  The graded linear series $\caJ_{k,m}(\Gamma, s)$ and
  $\caJ^{\cusp}_{k,m}(\Gamma, s)$ are of almost integral type.
\end{prop}
\begin{proof}
  We can view $s$ as a rational section of $\overline L_{k,m,\phi }$.
  Let $D=\dv(s)$ as a section of that line bundle. Since $\phi $ is
  sufficiently negative, by Theorem \ref{thm:3} we have
  \begin{displaymath}
    \caJ^{\cusp}_{k,m}(\Gamma, s)\subset
    \caJ_{k,m}(\Gamma, s) \subset \caR(D).
  \end{displaymath}
  By Proposition \ref{prop:2} the graded linear series $\caR(D)$ is of
  almost integral type, hence the result. 
\end{proof}
\begin{rmk}
Since $\caJ_{k,m}(\Gamma, s)$ and
  $\caJ^{\cusp}_{k,m}(\Gamma, s)$ are of almost integral type, one can
  associate convex Okounkov bodies to them (see \cite{kk,LM}). This
  gives another way of studying these rings by means of convex
  analysis. We expect this to be the case for other rings of
  automorphic forms in mixed Shimura varieties as well.
  \end{rmk}

Let $h= \overline{h}^{\inv}$ be the  psh metric of $\overline L_{k,m,\phi }$ obtained
by Proposition~\ref{prop:5}.  Following Section~\ref{sec:psh-metrics-b-1} we
have an associated graded linear series $\caR(\overline L_{k,m,\phi },s,h)$. 

\begin{lem}\label{lemm:11}
  The graded linear series $\caR(\overline L_{k,m,,\phi },s,h)$ and
  $\caJ^{\cusp}_{k,m}(\Gamma,s)$ are equal.
\end{lem}
\begin{proof} Let $\ell \in \Z_{\ge 0}$.
Then $\caR(\overline L_{k,m,\phi },s,h)_{\ell}$ is the set of $f
  \in F$ such that $fs^{\ell}$ is a
  meromorphic Siegel--Jacobi form, holomorphic on $\caH_{g}\times
  \C^{(1,g)}$ and such that $\|fs^{\ell}\|$ is bounded. By Proposition
  \ref{prop:6} the latter set is exactly the set of $f\in F$ such that
  $fs^{\ell}$ is a Siegel--Jacobi cusp form. 
\end{proof}

To the graded linear series  $\caJ^{\cusp}_{k,m}(\Gamma, s)$ and
$\caJ_{k,m}(\Gamma,s)$ we can associate the Weil $\R$-b-divisors
$\bdiv\Big(\caJ^{\cusp}_{k,m}(\Gamma, s)\Big)$ and
$\bdiv\Big(\caJ_{k,m}(\Gamma, s)\Big)$.

\begin{lem}\label{lemm:12}
  The equality
  \begin{displaymath}
    \bdiv\left(\caJ^{\cusp}_{k,m}(\Gamma, s)\right) = \bdiv\left(\caJ_{k,m}(\Gamma, s)\right)
  \end{displaymath}
  holds.
\end{lem}
\begin{proof}
  Since $\caJ^{\cusp}_{k,m}(\Gamma, s)\subset \caJ_{k,m}(\Gamma, s)$, by Lemma
  \ref{lemm:3} we have that
  \begin{displaymath}
\bdiv\left(\caJ^{\cusp}_{k,m}(\Gamma, s)\right) \le
  \bdiv\left(\caJ_{k,m}(\Gamma, s)\right).
\end{displaymath}
 To prove the converse inequality let $\pi \in
  R(X)$ and let $P$ be a prime divisor of $X_{\pi }$. Set
  \begin{displaymath}
    r=\ord_{P} \bdiv\left(\caJ_{k,m}(\Gamma, s)\right)\quad \text{and} \quad r_{0}=\ord_{P}
    \bdiv\left(\caJ^{\cusp}_{k,m}(\Gamma, s)\right).
  \end{displaymath}
Also choose a non-zero $f_{0}\in \caJ^{\cusp}_{k,m}(\Gamma, s)_{\ell_{0}}$ for some
  $\ell_0$ (this exists by Lemma \ref{lemm:18}), and let $\epsilon
  >0$. The number $r$ can be characterized as 
  \begin{displaymath}
    r=\sup\{(1/\ell)\ord_{P}f \mid \ell \ge 0,\, f\in \caJ_{k,m}(\Gamma, s)_{\ell}\}. 
  \end{displaymath}
Hence we can find $\ell\gg 0$ and $f\in \caJ_{k,m}(\Gamma, s)_{\ell}$
  satisfying the conditions
  \[
   \frac{\ord_{P}f}{\ell} \ge r-\epsilon, \quad
    \ell_{0}/\ell \le \epsilon \quad \text{and} \quad 
    \frac{\ord_{P}f_{0}}{\ell} \ge -\epsilon. 
    \]
  Note that to achieve the second and third condition we only need to
  make $\ell $ big enough. Since $ff_{0}$ is a cusp form by
  Lemma~\ref{lem:multiplicativity}, we have
  \begin{displaymath}
    \frac{\ord_{P}f+\ord_{P}f_{0}}{\ell+\ell_{0}}\le r_{0}.
  \end{displaymath}
  Together with the above conditions, this implies
  \begin{displaymath}
    \ell(r-\epsilon )\le \ord_{P}f \le r_{0}(\ell +\ell_{0})-\ord_{P}f_{0}
  \end{displaymath}
  and hence
  \begin{displaymath}
    r\le r_{0}(1+\epsilon )+2\epsilon. 
  \end{displaymath}
  As $\epsilon>0 $ can be chosen arbitrarily small we deduce that
  $r\le r_{0}$. This completes 
  the proof. 
\end{proof}

   \begin{lem}\label{lem:nef} For $k,m\ge 0$, 
   the Weil $\R$-b-divisor $\D(\overline L_{k,m,\phi },s,h)$ is nef. 
   \end{lem}
   \begin{proof}
By Proposition \ref{prop:5} the metric $h$ is psh. The result then
follows from Proposition~\ref{prop:3}. 
\end{proof}
   
   \begin{lem}\label{lem:big} For $k,m>0$,
     the Weil $\R$-b-divisor $\D(\overline L_{k,m,\phi },s,h)$ is big. 
   \end{lem}
   \begin{proof}
     By Lemma \ref{lemm:19}, after choosing  projective refinements
     $\Sigma _{0}$ and $\Pi_{0}$ of $\Sigma $ and $\Pi $ we
     can find numbers $m_{0}$, $k_{0}$ and polarization functions $\phi _{0}$
     and $\psi _{0}$ satisfying that $\overline L_{k_{0},m_{0},\phi_{0}
       +\psi_{0} }$ is ample and
     \begin{displaymath}
       H^{0}(\overline \caB(\Gamma )_{\Pi _{0}},\overline L_{\ell
         k_{0},\ell m_{0},\phi_{0}
       +\psi_{0} }) \subset J_{\ell k_{0},\ell m_{0}}(\Gamma ).
     \end{displaymath}
     After taking a multiple we can
     also assume that $\overline
     L_{m_{0},k_{0},\phi_{0}+\psi _{0}}$ is very ample, hence generated by
     global sections. Let $r>0$ such that
     \begin{displaymath}
       r k > k_{0},\qquad r m> m_{0}, \qquad J^{\cusp}_{r
         k-k_{0},r m-m_{0}}(\Gamma )\not = \{0\}.
     \end{displaymath}
     Such an $r>0$ exists by Lemma \ref{lemm:18}.
     Let  $0\not = \varphi\in J^{\cusp}_{r
         k-k_{0},r m-m_{0}}(\Gamma )$ be a cusp form and let $s_{0}$ be the
       rational section of $\overline 
       L_{k_{0},m_{0},\phi_{0}+\psi _{0}}$ such that $s^{r}=s_{0}\varphi$.

       If $f$ is a rational function
       such that $fs_{0}^{\ell}\in H^{0} \left(\overline
         \caB(\Gamma)_{\Pi_{0} },\overline 
     L_{\ell m_{0},\ell k_{0},\phi_{0}+\psi _{0}}\right)$, then by Lemmas
   \ref{lemm:19}  and \ref{lem:multiplicativity} we obtain that
   $f(s_0\varphi)^{\ell}\in J^{\cusp}_{\ell rk,\ell rm}(\Gamma
   )$. 
   Therefore there is an inclusion of graded linear series
   \begin{displaymath}
     \caR(\dv(s_{0})) \subset \caJ^{\cusp}_{rk,rm}(\Gamma
     ,s^r) \, .
   \end{displaymath}
Lemmas \ref{lemm:4} and
   \ref{lemm:11} yield the inclusions of graded linear
   series
   \begin{displaymath}
     \caR(\dv(s_{0})) \subset \caR(\overline L_{rk,rm,\phi },s^{r},h^r)
     \subset \caR(\D(\overline L_{rk,rm,\phi },s^r,h^r)).
   \end{displaymath}
   Since $\dv(s_0)$ is very ample, it is also big and generated by
   global sections. Applying now
   Proposition \ref{prop:1} and Lemma \ref{lemm:3}.\ref{item:14} we
   get
   \begin{displaymath}
     \dv(s_0) \le r\D(\overline L_{k,m,\phi },s,h).
   \end{displaymath}
   By Lemma \ref{lemm:20} the Weil $\R$-b-divisor $\D(\overline L_{k,m,\phi
   },s,h)$ is big. 
   \end{proof}

\begin{thm} \label{thm:2} 
  Let $k,m >0$. Then the graded algebra  $\bigoplus _{\ell} J_{\ell
    k,\ell m}(\Gamma )$ is not finitely generated. 
\end{thm}

\begin{proof}
  Assume that $\bigoplus _{\ell} J_{\ell
    k,\ell m}(\Gamma )$ is finitely generated. Choose $\Sigma$, $\Pi $ and $\phi$ as before. Choose $s$ a
  meromorphic Siegel--Jacobi form of weight $k$ and index $m$. Then
  the graded
  algebra $\bigoplus _{\ell} J_{\ell k,\ell m}(\Gamma )$ is
  isomorphic to the graded linear series $\caJ_{k,m}(\Gamma,
  s)$. Hence  this
  graded linear series is also finitely generated as an algebra. By Lemma
  \ref{lemm:2} the $\R$-b-divisor $\bdiv(\caJ_{k,m}(\Gamma, s))$ is a Cartier
  b-divisor. Since $\bdiv(\caJ_{\ell k,\ell m}(\Gamma, s))=\ell
  \bdiv(\caJ_{k,m}(\Gamma, s))$, the first one is Cartier if and only if the
  second one is Cartier. Therefore to achieve a contradiction we may
  assume that $m$ is divisible enough so that $m\phi $ has integral
  values on $\widetilde C(X_0)_\Z$. By Lemmas \ref{lemm:11} and
  \ref{lemm:12} we know that
  \begin{displaymath}
    \bdiv\left(\caJ_{k,m}(\Gamma, s)\right)= \bdiv\left(\caR(\overline L_{k,m,\phi },s,h)\right).
  \end{displaymath}
  We will now show that $\bdiv\left(\caR(\overline L_{k,m,\phi },s,h\right)$ is
  not Cartier. To this end we are allowed to replace the pair $(k,m)$
  by a suitable multiple and to change the section $s$.

  Let $D$ be a singularity divisor of $h$.  As seen in the proof of
  Lemma~\ref{lem:big},  the linear series $\caR(\overline L_{k,m,\phi },s,h)$ contains an ample linear
 series. This implies that after replacing $k$ and $m$ by appropriate
 multiples  we can change the section $s$ so that the condition
 $\ord_{D_{i}}\D(\overline L_{k,m,\phi },s,h)>0$ holds for all
 irreducible components of $D$. 

 By Lemmas \ref{lem:big} and \ref{lem:nef} and Corollary \ref{cor:toroidal} the Weil $\R$-b-divisor $\D (\overline L_{k,m,\phi },s,h)$ is big, nef, and
 toroidal.  
 By
 Corollary \ref{cor:2} we deduce that
  \begin{displaymath}
\bdiv\left(\caR(\overline L_{k,m,\phi },s,h)\right) = \D(\overline L_{k,m,\phi },s,h).
  \end{displaymath}
  By Corollary \ref{cor:1}, we know that $\D(\overline L_{k,m,\phi },s,h)$
  is not a Cartier $\R$-b-divisor. This proves the result. 
\end{proof}
Since Runge works with bounded (rather than fixed) ratio between the
weight and the index, we must work slightly more to show that Theorem
\ref{thm:2} disproves Runge's claim in \cite[Theorem 5.5]{runge}.

Let $J$ be a bigraded ring, graded over $\mathbb Z_{\ge 0}^2$. If $f \in J_{k,m}$ is (bi)homogeneous with $k \neq 0$ then we define the \emph{relative index} of $f$ as $r(f) = m/k$. Given $n \in \mathbb Q_{\ge 0}$, we define 
\begin{equation}
J_n = \bigoplus_{(k,m): m=  kn} J_{k,m} \;\;\; \text{and} \;\;\;J_{\le n} = \bigoplus_{(k,m): m \le  kn} J_{k,m}. 
\end{equation}
\begin{prop}\label{prop:disprove-runge}
Fix $n \in \mathbb Q_{>0}$, and suppose that $J_{\le n}$ is finitely generated as an algebra over $J_{0,0}$. Then $J_n$ is finitely generated as an algebra over $J_{0,0}$. 
\end{prop}
\begin{proof}
Given a finite set of elements $a_i \in J_{m_i, k_i}$, note that 
\begin{equation}\label{eq:face_trick}
\min_i r(a_i) \le r(\prod_i a_i) \le \max_i  r(a_i), 
\end{equation}
with equalities if and only if all $r(a_i)$ are equal. 

Now let $a_1, \dots, a_l$  be generators for $J_{\le n}$ as $J_{0,0}$-algebra; we may assume each $a_i$ is bihomogeneous, say of degree $(m_i, k_i)$. Then we claim that $J_n$ is generated by exactly those $a_i$ with $r(a_i) = n$; in particular, $J_n$ is finitely generated. 

To prove our claim, let $f \in J_n$, say $f \in J_{k,m}$, and write $f = \sum_I \lambda_I a^I$ with $I \in \mathbb Z_{\ge 0}^l$, $\lambda_I \in J_{0,0}$, and $a^I \coloneqq \prod_i a_i^{I_i}$. Then each $a^I$ is bihomogeneous, hence if $\lambda_I \neq 0$ then $a^I \in J_{k,m}$, in particular $r(a^I) = r(f)$. But then by \eqref{eq:face_trick} we know each of the $a_i$ with $I_i \neq 0$ has $r(a_i) = r(f) = n$.
\end{proof}

\begin{rmk}\label{rem:runge-mistake}
As was mentioned in the introduction, Theorem \ref{thm:2} disproves
\cite[Theorem 5.5]{runge}. We can trace back the oversight in Runge's
proof to the proof of Theorem~5.1 in \emph{loc.~cit.} Here, the
author states that the space of Siegel--Jacobi forms can be seen as
the set of global sections of a natural invertible sheaf on a
compactification of the universal abelian scheme living over the
Satake--Baily--Borel compactification of the moduli space of
principally polarized abelian varieties of level $\Gamma$. Let us be
more precise. Consider the Satake--Baily--Borel compactification
$\caA(\Gamma)^{*}$ of $\caA(\Gamma)$. Then
the canonical fibration morphism $\pi \colon \caB(\Gamma) \to
\caA(\Gamma)$ extends to a compactification $\overline{\pi} \colon
 \caB(\Gamma)^{*} \to
\caA(\Gamma)^{*}$, where
$\caB(\Gamma)^{*}$ is constructed as the
BiProj of a bigraded ring $R$ contained in the ring of Siegel--Jacobi
forms. 

The core of Runge's argument  leading to Theorem~5.1 of \emph{loc.~cit.} is that for every Siegel--Jacobi form in
$R$, the limit when one approaches a point on the boundary of
$ \caB(\Gamma)^{*}$ only depends on  $g$ parameters. Then it is 
claimed that this implies that all the fibers of the projection map $ \caB(\Gamma)^* \to
 \caA(\Gamma)^*$  have dimension $g$. 

The problem with this argument is that Siegel--Jacobi forms are
sections of a line bundle and we are looking at the completion of the
projective embedding
defined by this line bundle. In this situation, if several independent
Siegel--Jacobi forms converge to zero simultaneously when approaching
a point at the 
boundary, the dimension of the fibre may be bigger than~$g$.  
\end{rmk}

\section{The asymptotic dimension of spaces of Siegel--Jacobi forms} \label{sec:asymptotic}

The volume of a graded linear series $\caR$ on a variety $X$ of dimension~$d$ is the non-negative real number given by 
 \[
 \vol(\caR) = \limsup_{k \to \infty} \frac{\dim(\caR_k)}{k^d/d!}.
 \]
 In particular, given a Weil $\R$-b-divisor $\D$ on $X$, the volume of $\D$ can be expressed as $\vol(\D) = \vol(\caR(\D))$, see Definition~\ref{def:volume_b_div}. 
 
 We recall that  by \cite[Theorem~3.2]{Da-Fa20} any nef Weil $\R$-b-divisor $\D$ on $X$ has a well-defined degree $\D^d$ in $\R_{\ge 0}$. 
 \begin{lem} \label{lemm:HilbSam} Assume $\D$ is a  big and nef toroidal Weil $\R$-b-divisor on $X$. Then we have the Hilbert--Samuel formula
 \[ \vol \D = \D^d \, . \]
 \end{lem}
 \begin{proof} See \cite[Theorem 5.14]{BoteroBurgos}.
 \end{proof}
 \begin{lem} \label{lemm:continuity_volume} The function $\vol$ is continuous on the space of big and nef toroidal Weil $\R$-b-divisors on $X$.
 \end{lem}
 \begin{proof} See \cite[Corollary 5.16]{BoteroBurgos}.
 \end{proof}

As before  let $h= \overline{h}^{\inv}$ be the  psh metric of $\overline L_{k,m,\phi }$ obtained
by Proposition~\ref{prop:5}. Let $s$ be a non-zero rational section of $L_{k,m}$. 

 \begin{thm}\label{th:vol-jac} Let $k,m>0$ and let
$\D_{k,m}=\D(\overline{L}_{k,m,\phi}, s,h), \caJ_{k,m}(\Gamma, s)$ and
$\caJ_{k,m}^{\operatorname{cusp}}(\Gamma, s)$ be as in
Section~\ref{sec:graded-linear-series-1}. Let $G = g(g+1)/2$ and $d=
G+g= \dim \caB(\Gamma )$. Then the following sequence 
of equalities is satisfied. 
\[
\vol\left(\caJ_{k,m}(\Gamma, s)\right) = \vol\left(\caJ_{k,m}^{\operatorname{cusp}}(\Gamma, s)\right) = \vol\left(\D_{k,m}\right) = \D_{k,m}^{d}.
\]
 \end{thm}
 \begin{proof}
We know by Lemmas \ref{lem:big} and \ref{lem:nef} and Corollary \ref{cor:toroidal} that the Weil $\R$-b-divisor $\D_{k,m}$ is nef, big, and toroidal. Let $D$ be a singularity divisor of $h$. As in the proof of Theorem \ref{thm:2} we may assume that the condition $\operatorname{ord}_{D_i}\D_{k,m} >0$ holds  for all irreducible components $D_i$ of $D$. We let $\D_j =(1-1/j)\D_{k,m}$. The sequence $\{\D_j\}_j$ forms a sequence of nef and big toroidal Weil $\R$-b-divisors converging to $\D_{k,m}$. Again by Lemma \ref{lemm:4} and Corollary \ref{cor:3} we have 
\[
 \caR(\D_j)\subset \caR(\overline{L}_{k,m,\phi}, s,h) \subset \caR(\D_{k,m}).
 \]
 Taking limits, using the continuity statement in Lemma~\ref{lemm:continuity_volume}, and using Lemma~\ref{lemm:11} we obtain
 \begin{align}\label{eq:as1}
 \vol(\D_{k,m}) = \lim_{j \to \infty}\vol(\D_j) =
   \vol\left(\caR(\overline{L}_{k,m,\phi}, s,h)\right) =
   \vol\left(\caJ_{k,m}^{\operatorname{cusp}}(\Gamma, s)\right).
 \end{align}
 
 On the other hand, we have 
 \begin{align}\label{eq:cusp}
 \caJ_{k,m}^{\operatorname{cusp}}(\Gamma, s) \subset
   \caJ_{k,m}(\Gamma, s) \subset \caR(\bdiv\left(\caJ_{k,m}(\Gamma,
   s)\right),
 \end{align}
 where the last inclusion follows from Lemma~\ref{lemm:3}. Moreover, as in the proof of Theorem \ref{thm:2}
 we have 
 \[
 \caR(\bdiv\left(\caJ_{k,m}(\Gamma, s)\right) =
 \caR(\bdiv\left(\caJ_{k,m}^{\operatorname{cusp}}(\Gamma, s)\right) =
 \caR(\overline{L}_{k,m,\phi}, s,h) = \caR(\D_{k,m}).
 \]
Taking volumes in \eqref{eq:cusp} and using \eqref{eq:as1} we get 
 \begin{align}\label{eq:as2}
\vol\left(\caJ_{k,m}(\Gamma, s)\right) = \vol(\D_{k,m}).
\end{align}
Finally, since $\D_{k,m}$ is toroidal, by Lemma~\ref{lemm:HilbSam} we have 
\begin{align}\label{eq:as3}
\vol(\D_{k,m}) = \D_{k,m}^{d} \, . 
\end{align}
Combining \eqref{eq:as1}, \eqref{eq:as2} and \eqref{eq:as3} we obtain the result. 
 \end{proof}
\begin{cor}\label{cor:asy}
We have 
\begin{align}\label{eqn:limsup}
\limsup_{\ell\to \infty}\frac{\dim J_{\ell k, \ell m}(\Gamma)}{\ell^{d}/d!} = \limsup_{\ell\to \infty}\frac{\dim J^{\cusp}_{\ell k, \ell m}(\Gamma)}{\ell^{d}/d!} = \D_{k,m}^{d}.
\end{align}
\end{cor}

Therefore, to obtain an asymptotic estimate of the growth of
$\dim J_{\ell k, \ell m}(\Gamma)$ and $\dim J_{\ell k, \ell
  m}^{\cusp}(\Gamma)$ we are reduced to computing $\D_{k,m}^{d}$.
  \begin{rmk}\label{rmk:limsup2} The $\limsup$ in \eqref{eqn:limsup} is actually a $\lim$ for sufficiently divisible $\ell$ as it is the volume of a graded linear series of almost integral type (see Remark~\ref{rmk:limsup}).
  \end{rmk}

  Our next task is to compute the degree $\D_{k,m}^{d}$. 
 We recall that
$\D_{k,m}=\D(\overline{L}_{k,m,\phi}, s,h)$  and that (see  Proposition~\ref{prop:7}) the metric $h$ is toroidal with
respect to the boundary divisor $\overline  \caB(\Gamma )_{\Pi }\setminus \caB(\Gamma )$
for any admissible cone decomposition $\Pi $. Then by
\cite[Theorem 5.20]{BBHJ} we have that
\begin{displaymath}
  \D_{k,m}^{\udim}=\int_{\overline  \caB(\Gamma )_{\Pi }}\langle
  c_{1}(\overline L_{k,m,\phi},h)^{\udim}\rangle =
  \int _{\caB(\Gamma )} c_{1}(L_{k,m},h)^{\udim}. 
\end{displaymath}
Here the integral in the middle is the so-called \emph{non-pluripolar
volume}, which agrees with the integral on the right hand side because
the metric $h$ is smooth on $\caB(\Gamma)$. 

We let $h_B$ and $h_M$ denote the canonical metrics on the line
bundles $B$ and $M$ respectively,  
see Definition~\ref{def:6} and the remarks immediately thereafter.
As equation \eqref{eq:18} also gives the canonical invariant metric
$h$ on $L_{k,m}$, by transport of structure we
may use multi-linearity to deduce that
\begin{displaymath}
  \int _{\caB(\Gamma )} c_{1}(L_{k,m},h)^{\udim}=
  \sum _{r=1}^{\udim} \binom{\udim}{r}\int _{\caB(\Gamma )}  m^{r}k^{\udim-r}c_{1}(B,h_B)^{r} \pi ^{\ast} c_{1}(M,h_M)^{\udim-r}. 
\end{displaymath}
All the integrals in the sum are finite by \cite[Remark 2.23]{BBHJ}.

 Let
$[2]\colon \caB(\Gamma )\to \caB(\Gamma)$ be the map ``multiplication
by $2$'' fiber by fiber. Then $[2]$ is a finite map of degree
$2^{2g}$. Moreover, since $B$ is symmetric and is rigidified along the origin, we have a canonical isomorphism 
\begin{displaymath}
  [2]^{\ast}B\simeq B^{\otimes 4}.
\end{displaymath}
We have
\begin{displaymath}
  [2]^{\ast}c_{1}(B,h_B)=4c_{1}(B,h_B) \, , 
\end{displaymath}
i.e., the isomorphism is compatible with the metric (see \cite[Lemma~2.6]{DGH21}, for instance).

Therefore
\begin{align*}
  \int _{\caB(\Gamma )}  c_{1}(B,h_B)^{r} \pi ^{\ast} &c_{1}(M,h_M)^{\udim-r}\\&=
  \frac{1}{2^{2g}}\int _{\caB(\Gamma )} [2]^{\ast} c_{1}(B,h_B)^{r}
  [2]^{\ast}\pi ^{\ast} c_{1}(M,h_M)^{\udim-r}\\
  &=
  \frac{1}{2^{2g}}\int _{\caB(\Gamma )} 2^{2r}c_{1}(B,h_B)^{r}
  \pi ^{\ast} c_{1}(M,h_M)^{\udim-r}\\
  &=
  \frac{2^{2r}}{2^{2g}}\int _{\caB(\Gamma )} c_{1}(B,h_B)^{r}
  \pi ^{\ast} c_{1}(M,h_M)^{\udim-r}.
\end{align*}
Hence this integral is zero unless $r=g$. So 
\begin{displaymath}
  \int _{\caB(\Gamma )} c_{1}(L_{k,m},h)^{\udim}=
  \binom{\udim}{g}m^{g}k^{G}\int _{\caB(\Gamma )}
  c_{1}(B,h_B)^{g} \pi ^{\ast}
  c_{1}(M,h_M)^{G}.
\end{displaymath}
Let $A$ be a fibre of the map $\caB(\Gamma )\to \caA(\Gamma )$. It is
an abelian variety of dimension~$g$. By the
projection formula
\begin{displaymath}
  \int _{\caB(\Gamma )}
  c_{1}(B,h_B)^{g} \pi ^{\ast}
  c_{1}(M,h_M)^{G}=
  \deg(B|_{A})\int _{\caA(\Gamma)}c_{1}(M,h_M)^{G}.
\end{displaymath}
Since $B|_{A}$ is twice the principal polarization,
\begin{displaymath}
  \deg(B|_{A})=2^{g}g!. 
\end{displaymath}
We write $\Gamma _{0}=\Sp(2g,\Z)$. Then
\begin{displaymath}
  \int _{\caA(\Gamma)}c_{1}(M,h_M)^{G}=[\Gamma _{0}\colon
  \Gamma ] \int _{\caA_{g}}c_{1}(M,h_M)^{G},
\end{displaymath}
where the second integral is an orbifold integral. The formula
after \cite[Conjecture 8.3]{Geer:moduli} gives
\begin{displaymath}
  \int
  _{\caA_{g}}c_{1}(M,h_M)^{G}=
  (-1)^{G}\frac{G!}{2^{g}}\prod_{k=1}^{g}\frac{\zeta(1-2k)}{(2k-1)!!}.
\end{displaymath}
Summing up we obtain
\begin{displaymath}
  \D^{G + g}_{k,m}=(-1)^{G}m^{g}k^{G}(G+g)![\Gamma _{0}\colon
  \Gamma ] \prod_{k=1}^{g}\frac{\zeta (1-2k)}{(2k-1)!!}. 
\end{displaymath}
By Corollary~\ref{cor:asy} we obtain the
asymptotic growth of the dimension of the spaces of Siegel--Jacobi
forms and of cusp Siegel--Jacobi forms explicitly, recovering a formula already implicit in Tai's work \cite{tai}. 

\begin{cor}\label{th:asy-dim}
  The asymptotic growth of the dimension of the spaces of Siegel--Jacobi forms
  $J_{\ell k, \ell
    m}(\Gamma)$ and $J^{\cusp}_{\ell k, \ell m}(\Gamma)$ when $\ell$
  goes to infinity is given by the following formulae: 
\begin{equation*}
\begin{split}
\limsup_{\ell \to \infty}\frac{\dim J_{\ell k, \ell m}(\Gamma)}{\ell^\udim/\udim!} & = \limsup_{\ell \to \infty}\frac{\dim J^{\cusp}_{\ell k, \ell m}(\Gamma)}{\ell^\udim/\udim!} \\
& = \binom{\udim}{g}m^gk^G\deg(B|_A)  \int _{\caA(\Gamma)}c_{1}(M,h_M)^{G}\\
& =  (-1)^{G}\udim! m^{g}k^{G}[\Gamma _{0}\colon  \Gamma ] \prod_{k=1}^{g}\frac{\zeta (1-2k)}{(2k-1)!!}\\
 & =  (-1)^{\udim}\udim!m^{g}k^{G}2^{G-g} 
  [\Gamma _{0}\colon \Gamma ]\prod_{k=1}^{g}\frac{(k-1)!B_{2k}}{(2k)!}\\
  & = V_g \cdot \udim!m^gk^G2^{-G-1}\pi^{-G}[\Gamma_0 \colon \Gamma],
\end{split}
\end{equation*}
where $G = g(g+1)/2$, $d = G + g$, $B_{2k}  = \frac{(-1)^{k+1} 2(2k)!}{(2\pi)^{2k}}\zeta(2k)$ are
the Bernoulli numbers and  
\[
V_g = (-1)^\udim 2^{g^2+1}\pi^G\prod_{k=1}^g\frac{(k-1)!B_{2k}}{(2k)!}
\]
is the symplectic volume of $\caA_g$ computed by Siegel in \cite[Section VIII]{siegel}.
\end{cor}
By Remark \ref{rmk:limsup2} the $\limsup$ above is actually a $\lim$ for sufficiently divisible $\ell$.
\begin{rmk}\label{rem:tai-formula}
The above formulas can also be obtained by combining the formulas in the proofs of Propositions 2.1 and 2.5 of Tai's work~\cite{tai}. 
\end{rmk}

For the purpose of open access, a CC BY public copyright licence is
applied to any Author Accepted Manuscript (AAM) arising from this
submission.

\end{document}